\documentclass{amsart}

\usepackage[letterpaper,margin=1.2in]{geometry}

\usepackage{amsmath,amssymb,amsthm}
\usepackage{enumerate}
\usepackage[retainorgcmds]{IEEEtrantools}
\usepackage{hyperref}
\usepackage{mathtools,pifont}
\usepackage{tikz-cd}
\usepackage{tikz}
\usepackage{mathrsfs}
\usepackage{xcolor}
\usepackage{bbding}

\allowdisplaybreaks

\newtheorem{theorem}{Theorem}[section]
\newtheorem{observation}[theorem]{Observation}
\newtheorem{lemma}[theorem]{Lemma}
\newtheorem{proposition}[theorem]{Proposition}
\newtheorem{corollary}[theorem]{Corollary}

\numberwithin{equation}{section}

\theoremstyle{definition}
\newtheorem{definition}[theorem]{Definition}

\newtheorem{example}[theorem]{Example}
\newtheorem{question}[theorem]{Question}

\theoremstyle{remark}
\newtheorem{remark}[theorem]{Remark}

\newcommand\R{\mathbb{R}}
\newcommand\C{\mathbb{C}}

\newcommand\N{\mathbb{N}}

\newcommand\h{\mathbb{H}}

\newcommand{\cB}{\mathcal{B}}

\newcommand{\cT}{\mathcal{T}}
\newcommand{\cP}{\mathcal{P}}

\DeclareMathOperator{\id}{id}

\DeclareMathOperator{\Ran}{Ran}

\DeclareMathOperator{\im}{Im}

\DeclareMathOperator{\Hol}{Hol}

\DeclareMathOperator{\Tree}{Tree}

\DeclareMathOperator{\Perm}{Perm}

\DeclareMathOperator{\br}{br}
\DeclareMathOperator{\free}{free}
\DeclareMathOperator{\mono}{mono}
\DeclareMathOperator{\orth}{orth}
\DeclareMathOperator{\bool}{bool}
\DeclareMathOperator{\sub}{sub}

\DeclarePairedDelimiter{\norm}{\lVert}{\rVert}

\DeclareFontFamily{U}{mathb}{\hyphenchar\font45}
\DeclareFontShape{U}{mathb}{m}{n}{
	<5> <6> <7> <8> <9> <10> gen * mathb
	<10.95> mathb10 <12> <14.4> <17.28> <20.74> <24.88> mathb12
}{}
\DeclareSymbolFont{mathb}{U}{mathb}{m}{n}
\DeclareMathSymbol{\boxright}{3}{mathb}{'151}

\begin{document}
	
	\title{Tree convolution for probability distributions with unbounded support}
	
	\author{Ethan Davis}
	\address{Department of Mathematics, University of California, Los Angeles, Los Angeles, CA 90095}
	\email{edavis21@ucla.edu}
	
	\author{David Jekel}
	\address{Department of Mathematics, University of California, San Diego, La Jolla, CA, 92093}
	\email{djekel@ucsd.edu}
	
	\author{Zhichao Wang}
	\address{Department of Mathematics, University of California, San Diego, La Jolla, CA, 92093}
	\email{zhw036@ucsd.edu}
	
	
	\begin{abstract}
		We develop the complex-analytic viewpoint on the tree convolutions studied by the second author and Weihua Liu in \cite{JekelLiu2020}, which generalize the free, boolean, monotone, and orthogonal convolutions.  In particular, for each rooted subtree $\cT$ of the $N$-regular tree (with vertices labeled by alternating strings), we define the convolution $\boxplus_{\cT}(\mu_1,\dots,\mu_N)$ for arbitrary probability measures $\mu_1$, \dots, $\mu_N$ on $\R$ using a certain fixed-point equation for the Cauchy transforms.  The convolution operations respect the operad structure of the tree operad from \cite{JekelLiu2020}.  We prove a general limit theorem for iterated $\cT$-free convolution similar to Bercovici and Pata's results in the free case \cite{BP1999}, and we deduce limit theorems for measures in the domain of attraction of each of the classical stable laws.
	\end{abstract}
	
	\thanks{Jekel was supported by NSF grant DMS-2002826.  The data for the figures was generated using Sage on Cocalc, and the pictures were created with TikZ.  We thank the referee for suggesting several corrections and improvements to exposition and carefully checking for typos.}
	
	\maketitle
	
	\section{Introduction} \label{sec:intro}
	
	In \cite{Voiculescu1985,Voiculescu1986}, Voiculescu introduced free independence, which provided a probabilistic viewpoint on free products of operator algebras.  Two other forms of non-commutative independence were studied in non-commutative probability theory around the year 2000: boolean independence in \cite{SW1997} and monotone independence in \cite{Muraki2000,Muraki2001}.  Besides classical independence, these are the only types of independence that provide an associative natural product operation on non-commutative probability spaces \cite{Speicher1997,BGS2002,Muraki2003,Muraki2013}.  However, there are many other types of independence broadly defined.  For instance, Lenczewski defined $m$-free independences intermediate between free and boolean independence \cite{Lenczewski1998}.  One can combine several algebras using a mixture of classical and free independence \cite{Mlotkowski2004,SW2016}, boolean and monotone independence \cite{Wysoczanski2010}, or boolean and free independence \cite{KW2013}.  The notions of $c$-free \cite{BLS1996,ABT2019} and $c$-monotone \cite{Hasebe2011,Lenczewski2019} independence are another way of combining free or monotone independence with boolean independence, using pairs of states.
	
	Weihua Liu and the second author defined a general family of non-commutative independences associated to rooted trees whose vertices are labeled by alternating strings \cite{JekelLiu2020}, which would serve as a general framework for studying various convolution operations and the relationships between them, such as the relation between free, monotone, and subordination convolution in \cite{Lenczewski2007,Lenczewski2008}.  The independences defined by trees include free, monotone, and boolean independence; $m$-free independence; mixtures of free, boolean, and monotone independence.  The introduction of \cite{JekelLiu2020} noted three viewpoints on non-commutative independence (1) operator models, (2) combinatorics of moments, and (3) complex analysis of Cauchy transforms, of which that paper focused on only the first two.  Our present goal is to develop the complex-analytic viewpoint.
	
	To set the stage, let us recall some of the main ideas of \cite{JekelLiu2020}.  Let $\cT_{N,\free}$ be the tree whose vertices are alternating strings on the alphabet $[N] = \{1,\dots,N\}$ (strings where consecutive letters are distinct) and where two strings are adjacent precisely when one is obtained by appending one letter to the left of the other.  Let $\Tree(N)$ be the set of rooted subtrees of $\cT_{N,\free}$, where the root is the empty string.  Each $\cT \in \Tree(N)$ describes a way of combining $N$ Hilbert spaces with unit vectors $(H_1,\xi_1)$, \dots, $(H_N,\xi_N)$ into a new Hilbert space $(H,\xi)$ akin to the free product of pointed Hilbert spaces, which is called the $\cT$-free product of pointed Hilbert spaces \cite[\S 3]{JekelLiu2020}.  This in turn leads to a notion of $\cT$-free convolution:  Suppose $X_j$ is a bounded operator on $H_j$ whose spectral measure with respect to $\xi_j$ is $\mu_j$.  If $\tilde{X}_1$, \dots, $\tilde{X}_N$ are the corresponding operators on the product space $(H,\xi)$, then the convolution $\boxplus_{\cT}(\mu_1,\dots,\mu_N)$ is the spectral measure of $\tilde{X}_1 + \dots + \tilde{X}_N$ with respect to $\xi$.  (In fact, all of this was done in \cite{JekelLiu2020} in the more general setting where Hilbert spaces are replaced by $\cB$-$\cB$-correspondences for some $\mathrm{C}^*$-algebra $\cB$, and $\mu_j$ is a $\cB$-valued law.  But at present we are only concerned with the case $\cB = \C$ where the objects reduce to Hilbert spaces and compactly supported probability measures on $\R$.)
	
	In order to relate various convolution operations, the family $(\Tree(N))_{N \in \N}$ was equipped with the structure of a topological symmetric operad, and the convolution operations were shown to respect this operad structure \cite[\S 5]{JekelLiu2020}.  In particular, for $\cT \in \Tree(k)$ and $\cT_1 \in \Tree(n_1)$, \dots, $\cT_k \in \Tree(n_k)$, there is a well-defined composition $\cT(\cT_1,\dots,\cT_k) \in \Tree(n_1 + \dots + n_j)$ which satisfies
	\begin{multline*}
		\boxplus_{\cT(\cT_1,\dots,\cT_k)}(\mu_{1,1},\dots,\mu_{1,n_1},\dots \dots, \mu_{k,1},\dots,\mu_{k,n_k}) \\
		= \boxplus_{\cT}(\boxplus_{\cT_1}(\mu_{1,1},\dots,\mu_{1,n_1}),\dots,\boxplus_{\cT_k}(\mu_{k,1},\dots,\mu_{k,n_k}))
	\end{multline*}
	where $\mu_{i,j}$ is a compactly supported probability measure on $\R$.  Many known convolution identities can be proved in this framework \cite[\S 6]{JekelLiu2020}.
	
	As a consequence, \cite[Proposition 6.8]{JekelLiu2020} gave a decomposition of $\cT$-free convolution into boolean and orthogonal convolutions, which generalizes the decompositions of additive free convolution in \cite{Lenczewski2007}.  Let $\br_j(\cT) = \{s \in \cT_{N,\free}: sj \in \cT\}$, where $js$ denotes the string obtained by appending $j$ to the start of the string $s$.  Let $\uplus$ denote the boolean convolution and $\vdash$ the orthogonal convolution (see Examples \ref{ex:boolean} and \ref{ex:orthogonal} below).  Then
	\begin{equation} \label{eq:decomposition}
		\boxplus_{\cT}(\mu_1,\dots,\mu_N) = \biguplus_{j \in [N] \cap \cT} [\mu_j \vdash \boxplus_{\br_j(\cT)}(\mu_1,\dots,\mu_N)]
	\end{equation}
	for compactly supported probability measures on $\R$.  This relation is convenient for the complex-analytic viewpoint because the boolean and orthogonal convolutions have simple expressions in terms of the $K$-transform (an analytic function related to the Cauchy transform).
	
	In this paper, we will use \eqref{eq:decomposition} to \emph{define} the $\cT$-free convolution for arbitrary probability measures on $\R$.  More precisely, in Theorem \ref{thm:complexconvolution}, we will show that there is a unique family of operations $\boxplus_{\cT}$ on probability measures that satisfies \eqref{eq:decomposition} and depends continuously on $\cT$ (with respect to local convergence with respect to the root vertex).  The convolution $\boxplus_{\cT}(\mu_1,\dots,\mu_N)$ also depends continuously on $\mu_1$, \dots, $\mu_N$ and agrees in the compactly supported case with the prior definition from \cite{JekelLiu2020}.  Because \eqref{eq:decomposition} so directly relates with the $K$-transforms of measures, we can give self-contained proofs of the basic properties of $\cT$-free convolution without relying on operator models or on approximation of general probability measures with compactly supported ones, making the proofs in this paper essentially independent from \cite{JekelLiu2020}.  In particular, in \S \ref{sec:operadstuff}, we show directly from Theorem \ref{thm:complexconvolution} that the convolution operation on arbitrary measures respects the operad structure just as in the compactly supported case.
	
	In \S \ref{sec:limit1} and \S \ref{sec:limit2}, we discuss limit theorems for $\cT$-free independence.  Often when a new type of additive convolution is introduced, a central limit theorem and Poisson limit theorem are proved in the same paper or soon thereafter, as in e.g.\ \cite{Voiculescu1985,BLS1996,SW1997,FL1999,Muraki2001,Wysoczanski2010,KW2013,JekelLiu2020}. In classical probability, more general limit theorems for additive convolution are closely related to the study of infinitely divisible and stable distributions, as well as the L\'evy-Khintchine formula that classifies infinitely divisible distributions $\mu$ in terms of some other measure $\sigma$ and real number $\gamma$; see \cite{GK1954}.  Similar results have been obtained for non-commutative independences, both in the scalar-valued and the operator-valued settings; see for the free case \cite{Voiculescu1986,BV1992,Biane1998,Speicher1998,BP1999,PV2013,ABFN2013}, for the boolean case \cite{SW1997,PV2013,ABFN2013}, for the monotone case \cite{Muraki2001,Belinschi2006,Hasebe2010a,Hasebe2010b,HS2014,AW2014,AW2016,Jekel2020}, for the $c$-free case \cite{Krystek2007,BPV2013}.  One of the most influential works on the topic was Bercovici and Pata's paper \cite{BP1999}.  They showed that if $\mu_\ell$ is a sequence of measures and $k_\ell$ is a sequence of natural numbers tending to $\infty$, then $\mu_\ell^{*k_\ell}$ converges to a measure $\nu_*$ if and only if $\mu_\ell^{\boxplus k_\ell}$ converges to a measure $\nu_\boxplus$ if and only if $\mu_\ell^{\uplus k_\ell}$ converges to a measure $\nu_\uplus$, and the correspondence between $\nu_*$, $\nu_\boxplus$, and $\nu_\uplus$ is described in the terms of the respective L\'evy-Khintchine formulas.  From this general statement, they deduced free and boolean analogs of all classical limit theorems for additive convolution, and in particular limit theorems for the domains of attraction corresponding to each classical stable distribution.
	
	For a general choice of a tree $\cT \in \Tree(N)$, it is unclear how to define the $k$th convolution power for arbitrary $k$, as discussed in \cite[\S 8.1]{JekelLiu2020}.  However, we can define a $k$-fold composition of $\cT$ with itself, denoted $\cT^{\circ k}$; the corresponding convolution is an $N^k$-ary operation.  Let $n(\cT)$ denote the number of neighbors of the root vertex.  When $n(\cT) > 1$, \cite[\S 9]{JekelLiu2020} classified infinitely divisible laws in the $\cB$-valued setting under certain boundedness assumptions.  In this paper, in Theorem \ref{thm:BP1}, we obtain an analog of one direction of Bercovici and Pata's main result for arbitrary probability measures on $\R$.  If $\mu_\ell^{\uplus n(\cT)^{k_\ell}} \to \nu$, then $\boxplus_{\cT^{\circ k_\ell}}(\mu_\ell,\dots,\mu_\ell)$ converges to a measure $\mathbb{BP}(\cT,\nu)$ (Theorem \ref{thm:BP1}).  We do not know whether the converse implication holds.  Nonetheless, the theorem already contains the ``more practical'' implication, where the hypothesis is the relatively easy-to-check condition about boolean convolution and the conclusion describes convergence for general trees $\cT$ (and in fact gives a uniform rate over convergence over all $\cT \in \Tree(N)$).  In particular, Theorem \ref{thm:BP1} allows us to deduce limit theorems corresponding to each of the classical domains of attraction in \S \ref{sec:limit2} using similar techniques as in \cite[\S 5]{BP1999}.  We sketch some of the many open questions about $\cT$-free convolutions and limit theorems in \S \ref{sec:questions}.
	
	The paper is organized as follows:  In \S \ref{sec:Cauchy}, we explain background material on probability measures on $\R$ on their Cauchy transforms.  In \S \ref{sec:trees}, we review the operad of rooted trees from \cite{JekelLiu2020} and establish more of its basic properties.  In \S \ref{sec:convolution}, we define the $\cT$-free convolution of arbitrary probability measures on $\R$.  In \S \ref{sec:operadstuff}, we show that the convolution operations respect the operad structure.  In \S \ref{sec:limit1}, we prove the general limit theorem.  In \S \ref{sec:limit2}, we deduce as special cases limit theorems for each of the domains of attraction from classical probability theory.  In \S \ref{sec:questions}, we propose questions for future research.
	
	\section{Cauchy transforms of probability measures} \label{sec:Cauchy}
	
	$\mathcal{M}(\R)$ denotes the space of finite positive Borel measures on $\R$, $\mathcal{P}(\R)$ denotes the space of probability measures, equipped with the vague topology (that is, the weak-$*$ topology when viewed inside the dual of $C_0(\R)$; for background, see for instance \cite[\S 7.3]{Folland1999}).  Recall that $\mathcal{P}(\R)$ is metrizable using the \emph{L{\'e}vy distance}
	\[
	d_L(\mu,\nu) := \inf \Bigl\{\epsilon > 0: \mu((-\infty,x-\epsilon)) - \epsilon \leq \nu((-\infty,x)) \leq \mu(-\infty,x+\epsilon)) + \epsilon) \text{ for all } x \in \R \Bigr\}.
	\]
	Furthermore, $\mathcal{P}(\R)$ is a complete metric space with respect to $d_L$.  For proof, see for instance \cite[Theorem 6.8]{Billingsley1999}.
	
	\begin{definition}
		For a finite measure $\mu$ on $\R$, the \emph{Cauchy-Stieltjes transform} is given by
		\[
		G_\mu(z) = \int_{\R} \frac{1}{z - t} \,d\mu(t).
		\]
		The \emph{$F$-transform} is given by
		\[
		F_\mu(z) = 1 / G_\mu(z),
		\]
		and we also define
		\[
		K_\mu(z) = z - F_\mu(z).
		\]
		These functions are defined for all $z$ in $\C$ minus the closed support of $\mu$, but we usually view them as functions defined on the upper half-plane
		\[
		\h := \{z \in \C: \im(z) > 0\}.
		\]
	\end{definition}
	
	Let $\Hol(\h,-\overline{\h})$ be the space of holomorphic functions $\h \to -\overline{\h}$.  Then $\Hol(\h,-\overline{\h})$ is a normal family if we view the target space as sitting inside the Riemann sphere, hence the topology of pointwise convergence on $\Hol(\h,-\overline{\h})$ agrees with the (metrizable) topology of local uniform convergence.
	
	\begin{lemma} \label{lem:transformhomeomorphism}
		For each $m > 0$, the map
		\[
		\{\mu \in \mathcal{M}(\R): \norm{\mu} \leq m \} \to \Hol(\h,-\overline{\h}): \mu \mapsto G_\mu
		\]
		is a homeomorphism onto its image, where we use the weak-$*$ topology on $\mathcal{M}(\R)$ and the topology of local uniform convergence on $\Hol(\h,-\h)$.
	\end{lemma}
	
	This lemma is well-known as folklore.  In order to show that $\mu_n \to \mu$ if and only if $G_{\mu_n} \to G_\mu$ pointwise, one can use the fact the functions of the form $\phi_z(t) = 1/(z - t)$ span a dense subspace of $C_0(\R)$, which in turn follows from the Stone-Weierstrass theorem and the fact that $\phi_z(t) \phi_w(t) = (\phi_z(t) - \phi_w(t))/(z - w)$.
	
	Next, we recall the famous theorem of Nevanlinna \cite{Nevanlinna1922} that characterizes Cauchy transforms of probability measures as functions $G(z)$ such that $z G(z) \to 1$ as $z \to \infty$ non-tangentially in $\h$.  The version we state here comes from \cite{BV1992}.  For $a > 0$, let $\Gamma_a \subseteq \overline{\h}$ be the cone
	\[
	\Gamma_a := \{z: \im z \geq a |z| \}.
	\]
	We also define for $a, b, c > 0$, the regions
	\[
	\Gamma_{a,b} := \{z: \im z \geq \max(a|z|, b) \}.
	\]
	
	\begin{definition}
		Let $Y$ be a topological space, and let $F: \h \to Y$.  We say that \emph{$F(z) \to L$ as $z \to \infty$ non-tangentially} if for every $a > 0$,
		\[
		\lim_{\substack{z \to \infty \\ z \in \Gamma_a}} F(z) = y,
		\]
		or equivalently, for every $a > 0$ and every neighborhood $U$ of $y$, there exists $b > 0$ such that $F(z) \in U$ for every $z \in \Gamma_{a,b}$.
	\end{definition}
	
	\begin{theorem}[Nevanlinna]
		Let $G: \h \to \C$ and $m > 0$.  The following are equivalent:	
		\begin{enumerate}[(1)]
			\item $G$ is the Cauchy transform of some measure of total mass $m$.
			\item $G$ maps $\h$ into $-\h$ and $z G(z) \to m$ non-tangentially as $z \to \infty$.
			\item $G$ maps $\h$ into $-\h$ and $\lim_{y \to \infty} iy G(iy) = m$ over $y > 0$.
		\end{enumerate}
	\end{theorem}
	
	Besides Nevanlinna's original paper \cite{Nevanlinna1922}, the proof of (1) $\iff$ (3) can be found for instance in \cite[\S 32.1, Theorem 3]{Lax2002}, and the exact theorem here is in \cite[Proposition 5.1]{BV1992}.
	
	\begin{corollary}[{cf. \cite[Proposition 5.2]{BV1992}}] \label{cor:FandK}
		A function $F$ is the $F$-transform of some probability measure on $\R$ if and only if $F$ maps $\h$ into $\overline{\h}$ and $F(z) / z \to 1$ as $z \to \infty$ non-tangentially.  Similarly, $K$ is the $K$-transform of some probability measure on $\R$ if and only if $K$ maps $\h$ into $-\overline{\h}$ and $K(z) / z \to 0$ as $z \to \infty$ non-tangentially.
	\end{corollary}
	
	\begin{proof}
		The first claim is immediate from the theorem since $F_\mu(z) = 1 / G_\mu(z)$.  Similarly, for the second claim, the only thing that remains to prove is that $\im K_\mu(z) \leq 0$ for any probability measure $\mu$.  For $c > 0$, observe that the region
		\begin{align*}
			\Omega_c &= \{x+iy: \im(1/(x+iy)) \geq c\} \\
			&= \{x+iy: -y/(x^2 + y^2) \geq c\} \\
			&= \{x+iy: c(x^2 + y^2) + y \leq 0\}
		\end{align*}
		is a disk and in particular is convex.  For $z \in \h$ and $t \in \R$, we have $1/(z - t) \in \Omega_{\im z}$, and hence $G_\mu(z) = \int_{\R} (z - t)^{-1}\,d\mu(t) \in \Omega_{\im z}$.  Thus, $\im F_\mu(z) \geq \im z$, or equivalently $\im K_\mu(z) \leq 0$.
	\end{proof}
	
	The following result is contained in \cite[Proof of Proposition 2.6]{BP1999} and thus we leave the reader to look up or reconstruct the proof.
	
	\begin{lemma} \label{lem:NTconvergence2}
	If $Y$ is a compact family of probability measures, then $z G_\mu(z) \to 1$ as $z \to \infty$ non-tangentially, uniformly over $\mu \in Y$.  Similarly, we have $F_\mu(z) / z \to 1$ and $K_\mu(z) / z \to 0$ as $z \to \infty$ non-tangentially, uniformly for $\mu \in Y$.
	\end{lemma}

	\section{An operad of rooted trees} \label{sec:trees}
	
	\begin{definition}
		For $N \in \N$, let $[N] = \{1,\dots,N\}$.  A \emph{string} on the alphabet $[N]$ is a finite sequence $j_1 \dots j_\ell$ with $j_i \in [N]$.  We denote by the $i$th letter of a string $s$ by $s(i)$.  Given two strings $s_1$ and $s_2$, we denote their concatenation by $s_1 s_2$.
	\end{definition}
	
	\begin{definition} \label{def:string}
		A string $j_1 \dots j_\ell$ is called \emph{alternating} if $j_i \neq j_{i+1}$ for every $i \in \{1,\dots,\ell-1\}$.  
	\end{definition}
	
	\begin{definition} \label{def:stringtree}
		Let $\mathcal{T}_{N,\free}$ be the (simple) graph whose vertices are the alternating strings on the alphabet $[N]$ and where the edges are given by $s \sim j s$ for every letter $j$ and every string $s$ that does not begin with $j$.  Note that $\mathcal{T}_{N,\free}$ is an infinite $N$-regular tree.  We denote the empty string by $\emptyset$, and we view $\emptyset$ as the preferred root vertex of the graph $\mathcal{T}_{N,\free}$.
	\end{definition}
	
	\begin{definition}
		We denote by $\Tree(N)$ the set of rooted subtrees of $\mathcal{T}_{N,\free}$ (that is, connected subgraphs containing the vertex $\emptyset$).
		Note that if $\mathcal{T} \in \Tree(N)$, then the edge set is uniquely determined by the vertex set and vice versa.  Thus, we may treat $\mathcal{T}$ merely as a set of vertices when it is notationally convenient.  If $s \in \cT$ and $js \in \cT$ for some string $s$ and some $j \in [N]$, then we say that $js$ is a \emph{child} of $s$ and $s$ is the \emph{parent} of $js$.
	\end{definition}
	
	\begin{observation}
		For a rooted tree $\mathcal{T} \subseteq \mathcal{T}_{N,\free}$ and $\ell \geq 0$, let $B_\ell(\mathcal{T}) \subseteq \mathcal{T}_{N,\free}$ be set of strings in $\mathcal{T}$ of length $\leq \ell$ (or equivalently the closed ball of radius $\ell$ in the graph metric).  Define $\rho_N: \mathcal{T}_{N,\free} \times \mathcal{T}_{N,\free} \to \R$ by
		\[
		\rho_N(\mathcal{T}, \mathcal{T}') = \exp(-\sup \{\ell \geq 0: B_\ell(\mathcal{T}) = B_\ell(\mathcal{T}')\}).
		\]
		Then $\rho_N$ defines a metric on $\Tree(N)$ (and in fact an ultrametric), which makes $\Tree(N)$ into a compact metric space.
	\end{observation}
	
	The space $\Tree(N)$ is similar to the space of locally finite rooted graphs with the topology of local convergence (see e.g.\ \cite{AS2004}), and the observation is proved in a similar way to the literature on local convergence.  To summarize, $\Tree(N)$ by definition is a subset of the power set of $\cT_{N,\free}$.  This power set can be identified with $\{0,1\}^{\cT_{N,\free}}$, and thus we have a injective map $\Tree(N) \to \{0,1\}^{\cT_{N,\free}}$. The space $\{0,1\}^{\cT_{N,\free}}$ is compact in the product topology by Tychonoff's theorem.  A basis for this topology is given by cylinder sets defined by looking at finitely many coordinates.  In particular, we can use the cylinder sets defined by looking at the coordinates index by strings of length $\leq \ell$, for each $\ell \in \N$, which leads to a metric $\tilde{\rho}_N$ on $\{0,1\}^{\cT_{N,\free}}$ given by $\tilde{\rho}_N(x,y) = \exp(-\ell)$ where $\ell$ is the maximum length such that $x$ and $y$ agree on strings of length $\leq \ell$.  It follows that the topology we defined on $\Tree(N)$ is the restriction of the product topology.  It is straightforward to check that $\Tree(N)$ is closed in $\{0,1\}^{\cT_{N,\free}}$ hence compact.
	
	As explained in \cite[\S 5]{JekelLiu2020}, the sets of trees $(\Tree(k))_{k \in \N}$ form a \emph{topological symmetric operad}.  (For general background on operads, see e.g.\ \cite{Leinster2004}, and the complete definition is also explained in \cite{JekelLiu2020}.)  We have already described the topology.  The operad structure consists of composition maps
	\[
	\Tree(k) \times \Tree(n_1) \times \dots \times \Tree(n_k) \to \Tree(n_1 + \dots + n_k): (\cT,\cT_1,\dots,\cT_k) \mapsto \cT(\cT_1,\dots,\cT_k)
	\]
	for each $k \in \N$ and $n_1$, \dots, $n_k \in \N$, which are given as follows.  Let $\mathcal{T} \in \Tree(k)$ and $\mathcal{T}_1 \in \Tree(n_1)$, \dots, $\mathcal{T}_k \in \Tree(n_k)$.  Let $N_j = n_1 + \dots + n_j$ (which by convention includes $N_0 = 0$), and let $N = N_k$.  Define $\iota_j: [n_j] \to [N]$ by $\iota_j(i) = N_{j-1} + i$, so that $[N] = \bigsqcup_{j=1}^k \iota_j([n_j])$.  For a string $s \in \mathcal{T}_{n_j,\free}$, let $(\iota_j)_*(s)$ denote the string obtained by applying $\iota_j$ to each letter of $s$.  Then we define $\mathcal{T}(\mathcal{T}_1,\dots,\mathcal{T}_k) \in \Tree(N)$ to be the rooted subtree with vertex set
	\begin{equation}\label{def:operad_composition}
		\bigcup_{\ell \geq 0} \bigcup_{i_1 \dots i_\ell \in \mathcal{T}} \bigcup_{\substack{s_j \in \mathcal{T}_{i_j} \setminus \{\emptyset\} \\ \text{for } j \in [\ell]}} (\iota_{i_1})_*(s_1) \dots (\iota_{i_\ell})_*(s_\ell).
	\end{equation}
	In other words, the strings in $\mathcal{T}(\mathcal{T}_1, \dots, \mathcal{T}_k)$ are obtained by taking a string $t = i_1 \dots i_\ell$ in $\mathcal{T}$ and replacing each letter $i_j$ by a string $s_j$ from $\mathcal{T}_{i_j}$, with the indices appropriately shifted by $\iota_j: [n_j] \to [N]$.  This composition operation satisfies the operad associativity axioms.  It is also jointly continuous, and in fact, we have
	\[
	\rho_N(\cT(\cT_1,\dots,\cT_k),\cT'(\cT_1',\dots,\cT_k')) \leq \max(\rho_k(\cT,\cT'), \rho_{n_1}(\cT_1,\cT_1'),\dots,\rho_{n_k}(\cT_k,\cT_k')),
	\]
	where $\cT$, $\cT' \in \Tree(k)$ and $\cT_j$, $\cT_j' \in \Tree(n_j)$ for $j = 1$, \dots, $k$.  This is because every string of length $L$ in the composed tree has the form $(\iota_{i_1})_*(s_1) \dots (\iota_{i_\ell})_*(s_\ell)$ as above, where $\ell \leq L$ and $s_1$, \dots, $s_\ell$ have length $\leq L$.
	
	Finally, $\Tree$ is a \emph{symmetric} operad, which means that there is a right action of the permutation group $\Perm(N)$ on $\Tree(N)$ that satisfies natural compatibility properties with the operad composition (see \cite{Leinster2004}).  The permutation action on $\Tree(N)$ is defined as follows:  For a string $s = j_1 \dots j_\ell$, let $\sigma(s) = \sigma(j_1) \dots \sigma(j_\ell)$.  Then for a tree $\cT \subseteq \cT_{N,\free}$, let $\cT_\sigma = \{\sigma^{-1}(s): s \in \cT\}$.  This permutation action is continuous (and in fact isometric) on $\Tree(N)$.
	
	Central to this paper is the iterative formula from \cite[Proposition 6.8]{JekelLiu2020} which expresses convolutions over a tree $\cT$ in terms of the convolutions over the branches of $\cT$ for each neighbor of the root vertex; see \eqref{eq:fixedpointequation} and \eqref{eq:fixedpointequation2} below.  To set the stage, we define the branch operations and describe how they interact with the topological symmetric operad structure of $\Tree$.
	
	\begin{definition}
		For $j \in [N]$, we define $\br_j: \Tree(N) \to \Tree(N) \cup \{\varnothing\}$ by
		\begin{equation}\label{def:branch}
			\br_j(\mathcal{T}) = \{s \in \mathcal{T}_{N,\free}: sj \in \mathcal{T} \}.
		\end{equation}
		This gives the branch of $\mathcal{T}$ rooted at the vertex $j$ if $j \in \mathcal{T}$ and $\varnothing$ otherwise.
	\end{definition}
	
	We show an example of in Figure \ref{fig:branch-example}.
	
	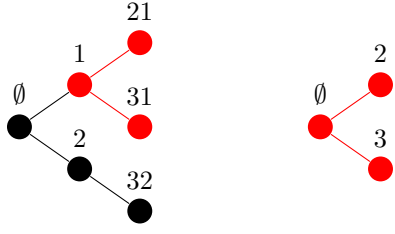
\begin{figure}
	
	\begin{center}
		
		\begin{tikzpicture}[scale=0.8]
			
			\node[circle,fill,label = above:$\emptyset$] (0) at (0,0) {};
			\node[circle,fill,red,label = above:$1$] (1) at (1,0.7) {};
			\node[circle,fill,label = above:$2$] (2) at (1,-0.7) {};
			\node[circle,fill,red,label = above:$21$] (21) at (2,1.4) {};
			\node[circle,fill,red,label = above:$31$] (31) at (2,0) {};
			\node[circle,fill,label = above:$32$] (32) at (2,-1.4) {};
			
			\draw (1) to (0) to (2) to (32); \draw[red] (21) to (1) to (31);
			
			\node[circle,fill,red,label = above:$\emptyset$] (O) at (5,0) {};
			\node[circle,fill,red,label = above:$2$] (A) at (6,0.7) {};
			\node[circle,fill,red,label = above:$3$] (B) at (6,-0.7) {};
			\draw[red] (A) to (O) to (B);
			
		\end{tikzpicture}
		
		\caption{The tree $\cT = \{\emptyset,1,2,21,31,32\}$ and a branch $\br_1(\cT) = \{\emptyset,2,3\}$.} \label{fig:branch-example}
		
	\end{center}
	
	\end{figure}
	
	\begin{observation}
	The map $\br_j$ is a continuous (and in fact $e$-Lipschitz) function from the clopen set $\{\mathcal{T} \in \Tree(N): j \in \mathcal{T}\}$ into $\Tree(N)$.
	\end{observation}
	
	The reason for this is of course that if $\cT$ and $\cT'$ agree on strings up to the length $\ell$ and both contain $j$, then $\br_j(\cT)$ and $\br_j(\cT')$ agree on strings up to length $\ell - 1$.
	
	\begin{observation}
	For $\cT \in \Tree(N)$ and $\sigma \in \Perm(N)$, we have $\br_j(\cT_\sigma) = \br_{\sigma(j)}(\cT)_\sigma$.
	\end{observation}
	
	In order to describe the relationship between the branch operation and operad composition, we need some auxiliary notions.  Let $\psi: [N] \to [N']$.  For a string $s = j_1 \dots j_\ell$ on $[N]$, let $\psi_*(s) = \psi(j_1) \dots \psi(j_\ell)$.  Viewing a tree $\cT \in \Tree(N)$ as a set of strings, we may compute the image $\psi_*(\cT)$ under the map $\psi_*$.  Of course, if $s$ is alternating, then $\psi_*(s)$ is not necessarily alternating.  Thus, $\psi_*(\cT)$ will be an element of $\Tree(N')$ if and only if $\psi_*(s)$ is alternating for every $s \in \cT$, or in other words, $\psi_*(\cT) \subseteq \cT_{N',\free}$.
	
	The branches of the composition will be expressed using $\mathcal{T}_{2,\mono} := \{\emptyset, 1, 2, 21\}$, a tree related to monotone convolution (see Example \ref{ex:monotone}).  Let $\cT_1 \in \Tree(m_1)$ and $\cT_2 \in \Tree(m_2)$.  Let $\phi_1: [m_1] \to [m_1 + m_2]$ map $[m_1]$ monotonically onto the first $m_1$ coordinates, and let $\phi_2: [m_2] \to [m_1+m_2]$ map $[m_2]$ monotonically onto the last $m_2$ coordinates.  Then $\cT_{2,\mono}(\cT_1,\cT_2)$ consists of four types of strings:\ the root vertex $\emptyset$, $(\phi_1)_*(s_1)$ for nonempty strings $s_1 \in \cT_1$, $(\phi_2)_*(s_2)$ for nonempty strings $s_2 \in \cT_2$, and $(\phi_2)_*(s_2)(\phi_1)_*(s_1)$ for nonempty strings $s_1 \in \cT_1$ and $s_2 \in \cT_2$.  This can be rewritten as
	\begin{equation} \label{eq:monotonecomposition}
	\mathcal{T}_{2,\mono}(\cT_1,\cT_2) = \{(\phi_2)_*(s_2)(\phi_1)_*(s_1): s_1 \in \cT_1, s_2 \in \cT_2 \}
	\end{equation}
	since $\emptyset = (\phi_2)_*(\emptyset) (\phi_1)_*(\emptyset)$ and $(\phi_1)_*(s_1) = (\phi_2)_*(\emptyset) (\phi_1)_*(s_1)$ and $(\phi_2)_*(s_2) = (\phi_2)_*(s_2) (\phi_1)_*(\emptyset)$.  Thus, $\cT_{2,\mono}(\cT_1,\cT_2)$ represents all strings obtained by concatenating a string from $\cT_2$ and a string from $\cT_1$ with the appropriate relabeling.
	
	\begin{lemma} \label{lem:branch}
		Let $\cT \in \Tree(k)$ and $\cT_1 \in \Tree(n_1)$, \dots, $\cT_k \in \Tree(n_k)$.  Let $N = n_1 + \dots + n_k$ and let $\iota_j: [n_j] \to [N]$ be the inclusions as above.  Fix $j \in [k]$ and $i \in [n_j]$.  Then
		\begin{equation} \label{eq:branch_composition}
		\br_{\iota_j(i)}(\cT(\cT_1,\dots,\cT_k)) = \{\tilde{s} \, (\iota_j)_*(s'): \tilde{s} \in \br_j(\cT)(\cT_1,\dots,\cT_k), s' \in \br_i(\cT_j) \},
		\end{equation}
		where $\tilde{s} \, (\iota_j)_*(s')$ denotes the concatenation of $\tilde{s}$ and $(\iota_j)_*(s')$.
		Let $\psi: [n_j+N] \to [N]$ map the first $n_j$ points monotonically onto $\iota_j([n_j])$ and map the last $N$ points monotonically onto $[N]$.  Then $\psi_*$ defines a bijection from $\cT_{2,\mono}(\br_i(\cT_j),\br_j(\cT)(\cT_1,\dots,\cT_k))$ to $\br_{\iota_j(i)}(\cT(\cT_1,\dots,\cT_k))$.
	\end{lemma}
	
	\begin{proof}
		To compute the left-hand side of \eqref{eq:branch_composition}, suppose that $s$ is a string on the alphabet $[N]$ with $s \, \iota_j(i) \in \cT(\cT_1,\dots,\cT_k)$.  Since $s \iota_j(i)$ is not the empty string, we can express it as
		\[
		s \, \iota_j(i) = (\iota_{j_1})_*(s_1) \dots (\iota_{j_\ell})_*(s_\ell),
		\]
		where $j_1 \dots j_\ell \in \cT \setminus \{\emptyset\}$ and $s_t \in \cT_{j_t}$ for $t = 1$, \dots, $\ell$.  Since the last letter is $\iota_j(i)$, we must have $j_\ell = j$.  Since $j_1 \dots j_{\ell-1} j \in \cT$, we have by definition $j_1 \dots j_{\ell-1} \in \br_j(\cT)$ and hence
		\[
		(\iota_{j_1})_*(s_1) \dots (\iota_{j_{\ell-1}})_*(s_{\ell-1}) \in \br_j(\cT)(\cT_1,\dots,\cT_k).
		\]
		Moreover, the string $s_\ell$ from $\cT_j$ has $i$ as its last letter, and therefore, $s_\ell = s'i$ for some $s' \in \br_i(\cT_j)$.  Hence,
		\[
		s = (\iota_{j_1})_*(s_1) \dots (\iota_{j_{\ell-1}})_*(s_{\ell-1}) (\iota_j)_*(s_\ell') = \tilde{s} (\iota_j)_*(s'),
		\]
		where $\tilde{s} := (\iota_{j_1})_*(s_1) \dots (\iota_{j_{\ell-1}})_*(s_{\ell-1})$ is in $\br_j(\cT)(\cT_1,\dots,\cT_k)$ and $s' \in \br_i(\cT_j)$.  Conversely, similar reasoning shows that whenever $\tilde{s}$ is in $\br_j(\cT)(\cT_1,\dots,\cT_k)$ and $s' \in \br_i(\cT_j)$, then the concatenation $\tilde{s} \, (\iota_j)_*(s')$ is in $\br_{\iota_j(i)}(\cT(\cT_1,\dots,\cT_k))$, and hence \eqref{eq:branch_composition} holds.
		
		Next, we show that $\psi_*$ maps $\cT_{2,\mono}(\br_i(\cT_j),\br_j(\cT)(\cT_1,\dots,\cT_k))$ onto $\br_{\iota_j(i)}(\cT(\cT_1,\dots,\cT_k))$.  Let $\phi_1: [n_j] \to [n_j + N]$ be the map sending $[n_j]$ monotonically onto the first $n_j$ coordinates, and let $\phi_2: [N] \to [n_j + N]$ be the map sending $[N]$ monotonically onto the last $N$ coordinates.  By our earlier observation \eqref{eq:monotonecomposition} about composition with $\cT_{2,\mono}$,
		\[
		\cT_{2,\mono}(\br_i(\cT_j),\br_j(\cT)(\cT_1,\dots,\cT_k)) = \{(\phi_2)_*(\tilde{s}) (\phi_1)_*(s'): \tilde{s} \in \br_j(\cT)(\cT_1,\dots,\cT_k), s' \in \br_i(\cT_j) \}.
		\]
		When we apply $\psi_*$ to this set, because $\psi \circ \phi_1 = \iota_j$ and $\psi \circ \phi_2 = \id_{[N]}$, we obtain the right-hand side of \eqref{eq:branch_composition}.  Thus, the image of $\cT_{2,\mono}(\br_i(\cT_j),\br_j(\cT)(\cT_1,\dots,\cT_k))$ under $\psi_*$ is what we asserted.
		
		In order to show that $\psi_*$ is injective on $\cT_{2,\mono}(\br_i(\cT_j),\br_j(\cT)(\cT_1,\dots,\cT_k))$, it suffices to show uniqueness of the decomposition of $s \in \br_{\iota_j(i)}(\cT(\cT_1,\dots,\cT_k))$ into $\tilde{s}$ and $(\iota_j)_*(s')$.  Note that if $\tilde{s}$ is not the empty string, then the last letter of $\tilde{s}$ cannot be in $\iota_j([n_j])$.  Thus, let $m$ be the position of the last letter in $s$ that does not come from $\iota_j([n_j])$, and let $m = 0$ if all the letters come from $\iota_j([n_j])$.   Then $\tilde{s}$ is the substring consisting of the first $m$ letters of $s$, and $\iota_j(s')$ is the remainder of $s$.  Since $\iota_j$ is injective, $s'$ is also uniquely determined.
	\end{proof}
	
	Next, we define isomorphism of rooted trees and describe how isomorphism relates to the branch maps.
	
	\begin{definition} \label{def:homomorphism}
		Let $\cT_1 \in \Tree(N_1)$ and $\cT_2 \in \Tree(N_2)$.  We say a map $\phi: \cT_1 \to \cT_2$ is a \emph{homomorphism} if $\phi(\emptyset)=\emptyset$ and for each vertex $s\in \cT_1$ and each child $s'$ of $s$, $\phi(s')$ is a child of $\phi(s)$.  We say that $\phi$ is an \emph{isomorphism} if it is a bijective homomorphism, and in this case, we write $\cT_1 \cong \cT_2$.
	\end{definition}
	
	\begin{observation} \label{obs:isomorphismbranch}
		Let $\cT_1 \in \Tree(N_1)$ and $\cT_2 \in \Tree(N_2)$ and let $\phi: \cT_1 \to \cT_2$ be an isomorphism.  Then $\phi$ defines a bijection $[N_1] \cap \cT_1 \to [N_2] \cap \cT_2$ and we have $\br_j(\cT_1) \cong \br_{\phi(j)}(\cT_2)$.
	\end{observation}
	
	\begin{definition} \label{def:maxchildren}
		For $\cT \in \Tree(N)$, let us write
		\begin{align*}
			n(\cT) &= |[N] \cap \cT| \\
			m(\cT) &= \max_{s \in \cT \setminus \{\emptyset\}} |\{j \in [N]: js \in \cT\}|,
		\end{align*}
		that is, $n(\cT)$ is the number of children of the root vertex and $m(\cT)$ is the maximum number of children of any other vertex of the tree.
	\end{definition}
	
	\begin{observation} \label{obs:isomorphictrees}
		The quantities $n(\cT)$ and $m(\cT)$ are isomorphism-invariant.  If $\cT \in \Tree(N)$, then $\max(n(\cT),m(\cT) + 1) \leq N$, and $\cT$ is isomorphic to some $\cT' \in \Tree(N')$ with $N' = \max(n(\cT),m(\cT)+1)$.
	\end{observation}
	
	\begin{proof}
		The first claim is immediate.  By construction, for $\cT \in \Tree(N)$, the root has at most $N$ children and the other vertices have at most $N - 1$ children.  Finally, letting $N' = \max(n(\cT),m(\cT) + 1)$, any isomorphism class of trees where the root has at most $N'$ children and the other vertices have at most $N'-1$ children can be realized by some subtree of $\cT_{N',\free}$.
	\end{proof}
	
	The final set of notation and results relates to compositions of several copies of the same tree; these remarks will be used in \S \ref{sec:limit1}. 
	
	\begin{definition} \label{def:samecomposition}
		Given trees $\cT_1 \in \Tree(N_1)$ and $\cT_2 \in \Tree(N_2)$, let
		\[
		\cT_1 \circ \cT_2 := \cT_1(\underbrace{\cT_2,\dots,\cT_2}_{N_1 \text{ times}}).
		\]
	\end{definition}
	
	This operation is associative because of the operad associativity property for $\Tree$.  Thus, the following definition also makes sense without parentheses.
	
	\begin{definition} \label{def:selfcomposition}
		For $\cT \in \Tree(N)$, let $\cT^{\circ k}$ be given by
		\[
		\cT^{\circ k} := \underbrace{\cT \circ \dots \circ \cT}_{k \text{ times}}.
		\]
	\end{definition}
	
	\begin{lemma} \label{lem:maxchildren}
		For $\cT_1 \in \Tree(N_1)$ and $\cT_2 \in \Tree(N_2)$, we have
		\begin{align*}
			n(\cT_1 \circ \cT_2) &= n(\cT_1) n(\cT_2) \\
			m(\cT_1 \circ \cT_2) &= m(\cT_1) n(\cT_2) + m(\cT_2).
		\end{align*}
		Moreover, (as in \cite[Lemma 8.7]{JekelLiu2020}) for $\cT \in \Tree(N)$ with $n(\cT) > 1$, we have
		\begin{align*}
			n(\cT^{\circ k}) &= n(\cT)^k \\
			m(\cT^{\circ k}) &= m(\cT) \frac{n(\cT)^k - 1}{n(\cT) - 1}.
		\end{align*}
	\end{lemma}
	
	\begin{proof}
		For $j \in [N_1]$, let $\iota_j: [N_2] \to [N_1N_2]$ be given by $\iota_j(i) = (j-1)N_2 + i$.  The neighbors of $\emptyset$ in $\cT_1 \circ \cT_2$ have the form $\iota_j(i)$ where $j$ is a neighbor of $\emptyset$ in $\cT_1$ and $i$ is a neighbor of $\emptyset$ in $\cT_2$, and hence $n(\cT_1 \circ \cT_2) = n(\cT_1) n(\cT_2)$.
		
		Next, consider the children of some non-root vertex of $\cT_1 \circ \cT_2$.  This vertex has the form $s = (\iota_{j_1})_*(s_1) \dots (\iota_{j_\ell})_*(s_\ell)$ where $j_1 \dots j_\ell \in \cT_1$ and $s_1$, \dots, $s_\ell \in \cT_2 \setminus \{\emptyset\}$.  There are two ways to append a letter to the front of this string and remain in $\cT_1 \circ \cT_2$.  First, we could append a letter $i$ to the front of $s_1$ in $\cT_2$ to obtain $(\iota_{j_1})_*(is_1) (\iota_{j_2})_*(s_2) \dots (\iota_{j_\ell})_*(s_\ell)$; there are at most $m(\cT_2)$ possible ways to do this.  Second, we could append $\iota_j(i)$ to $s$ for some $j$ such that $j j_1 \dots j_\ell \in \cT_1$ and some $i \in [N_2] \cap \cT_2$; there are at most $m(\cT_1) n(\cT_2)$ possible ways to do this.  Thus, the number of children of $s$ in $\cT_1 \circ \cT_2$ is at most $m(\cT_1) n(\cT_2) + m(\cT_2)$.  To show that this number of children is achieved in $\cT_1 \circ \cT_2 \setminus \{\emptyset\}$, pick some $j_1 \dots j_\ell \in \cT_1 \setminus \{\emptyset\}$ with $m(\cT_1)$ children, pick $s_1 \in \cT_2 \setminus \{\emptyset\}$ with $m(\cT_2)$ children, and pick $s_2$, \dots, $s_\ell \in \cT_2 \setminus \{\emptyset\}$ arbitrarily.  Then $s = (\iota_{j_1})_*(s_1) \dots (\iota_{j_\ell})_*(s_\ell)$ will have exactly $m(\cT_1) n(\cT_2) + m(\cT_2)$ children in $\cT_1 \circ \cT_2$ by the foregoing argument.
		
		Clearly, $n(\cT^{\circ k}) = n(\cT)^k$ follows by induction on $k$.  For the next formula, note that
		\[
		m(\cT^{\circ (k+1)}) = m(\cT \circ \cT^{\circ k}) = m(\cT) n(\cT^{\circ k}) + m(\cT^{\circ k}) = m(\cT) n(\cT)^k + m(\cT^{\circ k}).
		\]
		Hence,
		\[
		m(\cT^{\circ k}) = m(\cT) + \sum_{j=1}^{k-1} [m(\cT^{\circ (j+1)}) - m(\cT^{\circ j})] = \sum_{j=0}^{k-1} m(\cT) n(\cT)^j = m(\cT) \frac{n(\cT)^k - 1}{n(\cT) - 1}.  \qedhere
		\]
	\end{proof}
	
	\begin{lemma} \label{lem:isomorphismcomposition}
		Let $\cT_1 \in \Tree(N_1)$, $\cT_2 \in \Tree(N_2)$, $\cT_1' \in \Tree(N_1')$, $\cT_2' \in \Tree(N_2')$.  If $\cT_1 \cong \cT_1'$ and $\cT_2 \cong \cT_2'$ as rooted trees, then $\cT_1 \circ \cT_2 \cong \cT_1' \circ \cT_2'$.
	\end{lemma}
	
	\begin{proof}
		Let $\phi_1: \cT_1 \to \cT_1'$ and $\phi_2: \cT_2 \to \cT_2'$ be isomorphisms.  Let $\iota_j: [N_2] \to [N_1N_2]$ be given by $\iota_j(i) = (j-1)N_2 + i$, and define $\iota_j'$ analogously for $N_1'$ and $N_2'$ instead of $N_1$ and $N_2$.  Then we define an isomorphism $\psi: \cT_1 \circ \cT_2 \to \cT_1' \circ \cT_2'$ as follows.  Any vertex of $\cT_1 \circ \cT_2$ has the form $(\iota_{j_1})_*(s_1) \dots (\iota_{j_\ell})_*(s_\ell)$, where $j_1 \dots j_\ell \in \cT_1$ (here $\ell \geq 0$) and $s_i \in \cT_2 \setminus \{\emptyset\}$.  Now $\phi_1(j_1 \dots j_\ell)$ must be a string of the same length, so suppose that $\phi_1(j_1 \dots j_\ell) = j_1' \dots j_\ell'$.  Then we define
		\[
		\psi((\iota_{j_1})_*(s_1) \dots (\iota_{j_\ell})_*(s_\ell)) = (\iota_{j_1'}')_*(\phi_2(s_1)) \dots (\iota_{j_\ell'}')_*(\phi_2(s_\ell)).
		\]
		Since $\phi_1$ and $\phi_2$ are isomorphisms, the right-hand side will realize every possible string from $\cT_1' \circ \cT_2'$, and in fact will be a bijection.  The only thing left to prove is that $\psi$ preserves parent-child relationships, and this is done by examining the two cases of children as in the proof of the previous lemma.
	\end{proof}
	
	\section{Tree convolutions} \label{sec:convolution}
	
	The main result of this section is the following theorem:
	
	\begin{theorem} \label{thm:complexconvolution}
		There exists a unique function
		\[
		\Tree(N) \times \mathcal{P}(\R)^N \to \mathcal{P}(\R): (\mathcal{T},\mu_1,\dots,\mu_N) \mapsto \boxplus_{\mathcal{T}}(\mu_1,\dots,\mu_N)
		\]
		that is continuous in $\mathcal{T}$ and satisfies
		\begin{equation} \label{eq:fixedpointequation}
			K_{\boxplus_{\mathcal{T}}(\mu_1,\dots,\mu_N)}(z) = \sum_{j \in [N] \cap \mathcal{T}} K_{\mu_j}(z - K_{\boxplus_{\br_j(\mathcal{T})}(\mu_1,\dots,\mu_N)}(z)).
		\end{equation}
		In fact, this map is jointly continuous $\Tree(N) \times \mathcal{P}(\R)^N \to \mathcal{P}(\R)$.
	\end{theorem}
	
	The convolution will be constructed by iteration to a fixed point similar to the description of free and subordination convolutions in \cite{BMS2013}.  One of the main ingredients in the proof is the Earle-Hamilton theorem, which is a fixed-point theorem for holomorphic functions between Banach spaces.
	
	\begin{definition}[See {\cite{Zorn1945a,Zorn1945b,Zorn1946}}]
		Let $\mathcal{X}$ and $\mathcal{Y}$ be Banach spaces and let $\Omega$ be an open subset of $\mathcal{X}$.  A function $f: \Omega \to \mathcal{Y}$ is \emph{holomorphic} if
		\begin{enumerate}[(1)]
			\item For each $x \in \mathcal{X}$, there exists $r > 0$ such that $B(x,r) \subseteq \Omega$ and $f(B(x,r))$ is bounded.
			\item For each $x, x' \in \mathcal{X}$ and $\phi \in \mathcal{Y}'$, the function $\C \to \C$ mapping $z$ to $\phi[f(x + zx')]$ is holomorphic on the region where it is defined.
		\end{enumerate}
	\end{definition}
	
	\begin{theorem}[Earle-Hamilton {\cite{EH1970}}]
		Let $\mathcal{X}$ be a Banach space and $\Omega$ a connected open subset of $\mathcal{X}$.  Suppose that $\mathcal{F}: \Omega \to \Omega$ is holomorphic, $\mathcal{F}(\Omega)$ is bounded, and $d(\mathcal{F}(\Omega), \Omega^c) > 0$.  Then $\mathcal{F}$ has a unique fixed point in $\Omega$ and for any $x \in \Omega$, the iterates $\mathcal{F}^{\circ n}(x)$ converge to the fixed point as $n \to \infty$.
	\end{theorem}
	
	We also use the following lemma about $K$-transforms and truncated cones.  For $a \in (0,1)$ and $0 < b < c$, we define
	\[
	\Gamma_{a,b,c} := \{z: \im z \geq \max(a|z|,b), |z| \leq c\}.
	\]
	Note that $\Gamma_{a,b,c}$ is convex.  Of course, if $a$ were greater than $1$, this set would be empty since $\im z \leq |z|$.  We also remark that
	\[
	\Gamma_{a_1,b_1,c_1} \subseteq \Gamma_{a_2,b_2,c_2} \text{ if and only if } a_1 \geq a_2, b_1 \geq b_2, c_1 \leq c_2.
	\]
	
	\begin{lemma} \label{lem:mappingregions}
		Let $Y \subseteq \mathcal{P}(\R)$ be compact, let $N > 0$, and suppose that
		\[
		1 > a_0 > a_1 > 0, \qquad 0 < b_1 < b_0 < c_0 < c_1.
		\]
		and
		\[
		1 > a_2 > 0, \qquad 0 < b_2 < c_2.
		\]
		Then for sufficiently large $t$, we have
		\[
		\mu \in Y, z \in \Gamma_{a_0,tb_0,tc_0}, w \in \Gamma_{a_2,tb_2,tc_2} \implies z - N K_\mu(w) \in \Gamma_{a_1,tb_1,tc_1}.
		\]
	\end{lemma}
	
	\begin{proof}
		Let
		\[
		\epsilon(t) = \sup_{\mu \in Y} \sup_{w \in \Gamma_{a_2,tb_2,tc_2}} \frac{N|K_\mu(w)|}{|w|}.
		\]
		Note that $\epsilon(t) \to 0$ as $t \to \infty$ using Lemma \ref{lem:NTconvergence2}.  Let $z \in \Gamma_{a_0,tb_0,tc_0}$ and $w \in \Gamma_{a_2,tb_2,tc_2}$.  Note that
		\[
		\frac{\im(z - N K_\mu(w))}{|z - NK_\mu(w)|} \geq \frac{\im(z) - \epsilon(t) |w|}{|z| + \epsilon(t) |w|}
		\geq \frac{\im(z) - \epsilon(t)(tc_2/tb_0) \im(z) }{|z| + \epsilon(t) (tc_2/tb_0)|z|} = \frac{b_0 - \epsilon(t) c_2}{b_0 + \epsilon(t) c_2} \frac{\im z}{|z|} \geq \frac{b_0 - \epsilon(t) c_2}{b_0 + \epsilon(t) c_2} a_0,
		\]
		where we have used the fact that $|z| \geq \im z \geq tb_0$.  Since $\epsilon(t) \to 0$, we have for sufficiently large $t$ that
		\[
		\frac{b_0 - \epsilon(t) c_2}{b_0 + \epsilon(t) c_2} a_0 \geq a_1.
		\]
		Next, note that
		\[
		\im(z - N K_\mu(w))
		\geq \im(z) - \epsilon(t) |w|
		\geq t[b_0 - \epsilon(t) c_2].
		\]
		This will be greater than or equal to $tb_1$ provided that $t$ is large enough that $b_0 - \epsilon(t) c_2 \geq b_1$.  Finally,
		\[
		|z - NK_\mu(w)| \leq |z| + \epsilon(t) |w| \leq tc_0 + \epsilon(t) tc_2 = t(c_0 + \epsilon(t) c_2).
		\]
		This will be less than or equal to $t c_1$ provided that $t$ is large enough that $c_0 + \epsilon(t) c_2 \leq c_1$.
	\end{proof}
	
	\begin{proof}[Proof of Theorem \ref{thm:complexconvolution}]
		First, let us prove the uniqueness claim.  Note that if $\mathcal{T}$ is a finite tree of depth $d$, then \eqref{eq:fixedpointequation} expresses $K_{\boxplus_{\mathcal{T}}(\mu_1,\dots,\mu_N)}$ in terms of the branches of $\mathcal{T}$, which are trees of depth at most $d - 1$.  Therefore, by induction, $\boxplus_{\mathcal{T}}(\mu_1,\dots,\mu_N)$ is uniquely determined for all finite trees in $\Tree(N)$.  However, finite trees are dense in $\Tree(N)$, so by continuity, $\boxplus_{\mathcal{T}}(\mu_1,\dots,\mu_N)$ is uniquely determined for every tree.
		
		To prove the existence and continuity claims, we begin more generally.  Let $Y$ be a compact subset of $\mathcal{P}(\R)$, and fix some
		\[
		1 > a_0 > a_1 > a_2 > 0, \qquad 0 < b_2 < b_1 < b_0 < c_0 < c_1 < c_2.
		\]
		Fix some $t$ as in the conclusion of Lemma \ref{lem:mappingregions}.  We will apply the Earle-Hamilton theorem with
		\[
		\mathcal{X} = C(\Tree(N) \times Y^N \times \Gamma_{a_0,tb_0,tc_0}), \qquad \Omega = C(\Tree(N) \times Y^N \times \Gamma_{a_0,tb_0,tc_0},\, (\Gamma_{a_2,tb_2,tc_2})^\circ).
		\]
		To check that $\Omega$ is open, note that because $\Tree(N) \times Y^N \times \Gamma_{a_0,tb_0,tc_0}$ is compact, any continuous function $f$ from this space into $(\Gamma_{a_2,tb_2,tc_2})^\circ$ will have compact image, hence the image will be separated by a positive distance $\delta$ from $\C \setminus \Gamma_{a_2,tb_2,tc_2}$, and then $C(\Tree(N) \times Y^N \times \Gamma_{a_0,tb_0,tc_0}, (\Gamma_{a_2,tb_2,tc_2})^\circ)$ contains the ball of radius $\delta / 2$ around $f$ in $C(\Tree(N) \times Y^N \times \Gamma_{a_0,tb_0,tc_0})$.  Clearly, $\Omega$ is connected (and in fact convex) because it consists of functions with the convex target space $\Gamma_{a_0,tb_0,tc_0}$.
		
		Now let $\mathcal{F}: \Omega \to \mathcal{X}$ be given by
		\[
		\mathcal{F}(f)(\mathcal{T},\mu_1,\dots,\mu_N,z) = z - \sum_{j \in \mathcal{T} \cap [N]} K_{\mu_j}(f(\br_j(\mathcal{T}),\mu_1,\dots,\mu_N,z)).
		\]
		Because $\br_j$ is continuous, it is straightforward to check that $\mathcal{F}(f)$ is continuous, hence is an element of $\mathcal{X}$.  Because $K_{\mu_j}$ is holomorphic, it follows that $\mathcal{F}$ is a holomorphic function $\Omega \to \mathcal{X}$.  Indeed, it suffices to check for each $j$ the holomorphicity of the map $\mathcal{F}_j$ given by
		\[
		\mathcal{F}_j(f)(\mathcal{T},\mu_1,\dots,\mu_N,z) = \mathbf{1}_{j \in \mathcal{T}} K_{\mu_j}(f(\br_j(\mathcal{T}),\mu_1,\dots,\mu_N,z)).
		\]
		Letting $\Tree(N)_j = \{\mathcal{T} \in \Tree(N): j \in \mathcal{T}\}$, we can write $\mathcal{F}_j$ as the composition of the following maps:
		\begin{itemize}
			\item The map $C(\Tree(N) \times Y^N \times \Gamma_{a_0,tb_0,tc_0},\Gamma_{a_2,tb_2,tc_2}) \to C(\Tree(N)_j \times Y^N \times \Gamma_{a_0,tb_0,tc_0},\Gamma_{a_2,tb_2,tc_2})$ given by precomposition in the $\mathcal{T}$-coordinate with $\br_j: \Tree(N)_j \to \Tree(N)$.  This is the restriction of a linear transformation $C(\Tree(N) \times Y^N \times \Gamma_{a_0,tb_0,tc_0}) \to C(\Tree(N)_j \times Y^N \times \Gamma_{a_0,tb_0,tc_0})$, hence is holomorphic.
			\item Pointwise application of $K_{\mu_j}$, which maps $C(\Tree(N)_j \times Y^N \times \Gamma_{a_0,tb_0,tc_0},\Gamma_{a_2,tb_2,tc_2})$ holomorphically into $C(\Tree(N)_j \times Y^N \times \Gamma_{a_0,tb_0,tc_0})$.
			\item The inclusion map $C(\Tree(N)_j \times Y^N \times \Gamma_{a_0,tb_0,tc_0}) \to C(\Tree(N) \times Y^N \times \Gamma_{a_0,tb_0,tc_0})$ given by extension by zero (recall that $\Tree(N)_j$ is clopen in $\Tree(N)$).  This map is linear, hence holomorphic.
		\end{itemize}
		We claim that
		\begin{equation} \label{eq:imageinclusion}
			\mathcal{F}(\Omega) \subseteq C(\Tree(N) \times Y^N \times \Gamma_{a_0,tb_0,tc_0}, \Gamma_{a_1,tb_1,tc_1}) \subseteq \Omega.
		\end{equation}
		Fix $f \in \Omega$, and fix $\mathcal{T}$, $\mu_1$, \dots, $\mu_N$, and $z$.  By our choice of $t$ (see Lemma \ref{lem:mappingregions}), since $\mu_j \in Y$ and $z \in \Gamma_{a_0,tb_0,tc_0}$ and $f(\br_j(\mathcal{T}),\mu_1,\dots,\mu_N,z) \in \Gamma_{a_2,tb_2,tc_2}$, we have
		\[
		z - N K_{\mu_j}(f(\br_j(\mathcal{T}),\mu_1,\dots,\mu_N,z)) \in \Gamma_{a_1,tb_1,tc_1}.
		\]
		Now because $\Gamma_{a_1,tb_1,tc_1}$ is convex and contains $z$, the point
		\[
		z - \sum_{j \in \mathcal{T} \cap [N]} K_{\mu_j}(f(\br_j(\mathcal{T}),\mu_1,\dots,\mu_N,z)) = \frac{|[N] \setminus \mathcal{T}|}{N} z + \sum_{j \in [N] \cap \mathcal{T}} \frac{1}{N} \left( z - NK_{\mu_j}(f(\br_j(\mathcal{T}),\mu_1,\dots,\mu_N,z)) \right)
		\]
		is in $\Gamma_{a_1,tb_1,tc_1}$.  Therefore, $\mathcal{F}(f)(\mathcal{T},\mu_1,\dots,\mu_N,z) \in \Gamma_{a_1,tb_1,tc_1}$, demonstrating \eqref{eq:imageinclusion}.
		
		Now $\Gamma_{a_1,tb_1,tc_1}$ is separated by a positive distance $\delta$ from $\Gamma_{a_2,tb_2,tc_2}^c$.  This implies that $C(\Tree(N) \times Y^N \times \Gamma_{a_0,tb_0,tc_0}, \Gamma_{a_1,tb_1,tc_1})$ is separated by $\delta$ from the complement of $\Omega = C(\Tree(N) \times Y^N \times \Gamma_{a_0,tb_0,tc_0}, (\Gamma_{a_2,tb_2,tc_2})^\circ)$.  Therefore, the Earle-Hamilton theorem applies and there is a unique $f \in \Omega$ that satisfies $\mathcal{F}(f) = f$.  Moreover, the iterates $\mathcal{F}^{\circ n}(z)$ (where $z$ represents the constant function with value $z$) converge to $f$ in $C(\Tree(N) \times Y^N \times \Gamma_{a_0,tb_0,tc_0}, (\Gamma_{a_2,tb_2,tc_2})^\circ)$ as $n \to \infty$.
		
		Now we can prove the existence claim.  Fix $\mu_1$, \dots, $\mu_N$.  In the foregoing argument, we can take $Y = \{\mu_1,\dots,\mu_N\}$, which is clearly compact.  Let
		\[
		F_{\mathcal{T},\mu_1,\dots,\mu_N}^{(0)}(z) = z
		\]
		and
		\[
		F_{\mathcal{T},\mu_1,\dots,\mu_N}^{(n+1)}(z) = z - \sum_{j \in [N] \cap \mathcal{T}} K_{\mu_j}(F_{\br_j(\mathcal{T}),\mu_1,\dots,\mu_N}^{(n)}(z)).
		\]
		By a straightforward induction argument, $F_{\mathcal{T},\mu_1,\dots,\mu_N}^{(n+1)}(z)$ is well-defined and is a holomorphic map from the upper half-plane to itself.  Moreover,
		\[
		F_{\mathcal{T},\mu_1,\dots,\mu_N}^{(n)}|_{\Gamma_{a_0,tb_0,tc_0}} = \mathcal{F}^{\circ n}(z)(\mathcal{T},\mu_1,\dots,\mu_N,\cdot).
		\]
		Hence, the preceding argument shows that $F_{\mathcal{T},\mu_1,\dots,\mu_N}^{(n)}$ converges uniformly on $\Gamma_{a_0,tb_0,tc_0}$ as $n \to \infty$.  Because $\Hol(\h,\h)$ is a normal family when the target space is viewed as a subset of the Riemann sphere, it follows that $F_{\mathcal{T},\mu_1,\dots,\mu_N}^{(n)}$ converges locally uniformly on all of $\h$ as $n \to \infty$ to some function $F_{\mathcal{T},\mu_1,\dots,\mu_N}$ taking values in the closure of $\h$ in the Riemann sphere.  But $F_{\cT,\mu_1,\dots,\mu_N}$ maps $\Gamma_{a_0,tb_0,tc_0}$ into $\Gamma_{a_2,tb_2,tc_2}$ which is in $\h$, and therefore, the open mapping theorem implies that $F_{\cT,\mu_1,\dots,\mu_N}$ maps $\h$ into $\h$.  The identity
		\[
		F_{\mathcal{T},\mu_1,\dots,\mu_N}(z) = z - \sum_{j \in [N] \cap \mathcal{T}} K_{\mu_j}(F_{\br_j(\mathcal{T}),\mu_1,\dots,\mu_N}(z))
		\]
		holds on $\Gamma_{a_0,tb_0,tc_0}$ by the foregoing argument, and hence it holds on all of $\h$ by the identity theorem.
		
		Next, we argue that $F_{\mathcal{T},\mu_1,\dots,\mu_N}$ is the $F$-transform of some probability measure $\mu$.  By Nevanlinna's theorem, it suffices to show that $F_{\mathcal{T},\mu_1,\dots,\mu_N}(it) / t \to i$ as $t \to +\infty$ on the positive real axis.  For this purpose, let us forget the original values of $a_j, b_j, c_j$ and $t$.  Given a neighborhood $U$ of $i$, we may choose $a_0 > a_1 > a_2 > 0$ and $0 < b_2 < b_1 < b_0 < c_0 < c_1 < c_2$ such that
		\[
		\Gamma_{a_1,b_1,c_1} \subseteq U.
		\]
		If $t$ is sufficiently large, then the foregoing argument shows that
		\[
		F_{\mathcal{T},\mu_1,\dots,\mu_N}(\Gamma_{a_0,tb_0,tc_0}) \subseteq \Gamma_{a_1,tb_1,tc_1}
		\]
		since the fixed point of $\mathcal{F}$ must clearly be in $\mathcal{F}(\Omega)$.  In particular, since $it \in t \Gamma_{a_0,b_0,c_0} = \Gamma_{a_0,tb_0,tc_0}$, we get $F_{\mathcal{T},\mu_1,\dots,\mu_N}(it) \in \Gamma_{a_1,tb_1,tc_1} = t \Gamma_{a_1,b_1,c_1}$ and hence $F_{\mathcal{T},\mu_1,\dots,\mu_N}(it) / t \in U$.  Thus, $F_{\mathcal{T},\mu_1,\dots,\mu_N}$ is the $F$-transform of some probability measure $\boxplus_{\mathcal{T}}(\mu_1,\dots,\mu_N)$.  This concludes the existence claim.
		
		Finally, we must show joint continuity of $(\mathcal{T},\mu_1,\dots,\mu_N) \mapsto \boxplus_{\mathcal{T}}(\mu_1,\dots,\mu_N)$.  Since $\mathcal{P}(\R)$ is metrizable by Prokhorov's theorem, it suffices to show sequential continuity, which in turn will follow if we show that the map is continuous on $\Tree(N) \times Y^N$ for every compact $Y \subseteq \mathcal{P}(\R)$.  Fix constants $a_0 > a_1 > a_2 > 0$ and $0 < b_2 < b_1 < b_0 < c_0 < c_1 < c_2$ and let $t$ be as in Lemma \ref{lem:mappingregions}.  Then by the previous argument involving the Earle-Hamilton theorem, the map
		\[
		\Tree(N) \times Y^N \times \Gamma_{a_0,tb_0,tc_0} \to \Gamma_{a_1,tb_1,tc_1}: (\mathcal{T},\mu_1,\dots,\mu_N,z) \mapsto F_{\mathcal{T},\mu_1,\dots,\mu_N}(z)
		\]
		is jointly continuous, due to the definition of the set $\Omega$.  Since the domain of this function is compact, it is uniformly continuous, and hence
		\[
		\Tree(N) \times Y^N \to C(\Gamma_{a_0,tb_0,tc_0},\Gamma_{a_1,tb_1,tc_1}): (\mathcal{T},\mu_1,\dots,\mu_N) \mapsto F_{\mathcal{T},\mu_1,\dots,\mu_N}|_{\Gamma_{a_0,tb_0,tc_0}}
		\]
		is continuous.  Because $\Hol(\h,\h)$ is a normal family, uniform convergence on $\Gamma_{a_0,tb_0,tc_0}$ of a sequence $F_n$ in $\Hol(\h,\h)$ to some $F \in \Hol(\h,\h)$ implies local uniform convergence $F_n \to F$ on all of $\h$.  Hence, we have continuity of the map
		\[
		\Tree(N) \times Y^N \to \Hol(\h,\h): (\mathcal{T},\mu_1,\dots,\mu_N) \mapsto F_{\mathcal{T},\mu_1,\dots,\mu_N}.
		\]
		But by Lemma \ref{lem:transformhomeomorphism}, this is equivalent to continuity of $(\mathcal{T},\mu_1,\dots,\mu_N) \mapsto \boxplus_{\mathcal{T}}(\mu_1,\dots,\mu_N)$, which is what we wanted to prove.
	\end{proof}
	
	\begin{corollary}
		The convolution operation defined in Theorem \ref{thm:complexconvolution} in the case of compactly supported measures agrees with the one defined in \cite{JekelLiu2020} for $\cB = \C$.
	\end{corollary}
	
	\begin{proof}
		Let $\cT \in \Tree(N)$, and let $\mu_1$, \dots, $\mu_N$ be probability measures supported in $[-R,R]$.  The paper \cite{JekelLiu2020} took the viewpoint of treating the measures as positive linear functionals on the polynomial algebra, which is equivalent in the case of compactly supported measures.  The convolution operation in \cite{JekelLiu2020} was shown to satisfy \eqref{eq:fixedpointequation}; see \cite[\S 6.3, equation (6.3)]{JekelLiu2020}.  Now the $\cT$-free convolution of $\mu_1$, \dots, $\mu_N$ is supported in $[-NR,NR]$, and the moments depend continuously on $\cT$; see \cite[\S 5.2]{JekelLiu2020}.  Therefore, the convolution in \cite{JekelLiu2020} for $\mu_1$, \dots, $\mu_N$ satisfies the fixed point equation and continuity property of Theorem \ref{thm:complexconvolution}, so it agrees with the convolution defined in that theorem.
	\end{proof}
	
	Here are a few simple cases of convolution operations that we will use later.
	
	\begin{example}
		Let $\cT = \{\emptyset\} \in \Tree(N)$.  Then
		\[
		K_{\boxplus_{\cT}(\mu_1,\dots,\mu_N)}(z) = 0
		\]
		because it is the sum over an empty index set.  Hence, $\boxplus_{\cT}(\mu_1,\dots,\mu_N) = \delta_0$.
	\end{example}
	
	\begin{example} \label{ex:boolean}
		Let $\cT_{N,\bool} = \{\emptyset\} \cup [N] \in \Tree(N)$.  Then $\br_j(\cT_{N,\bool}) = \{\emptyset\}$.  Therefore,
		\[
		K_{\boxplus_{\cT_{N,\bool}}(\mu_1,\dots,\mu_N)}(z) = \sum_{j=1}^N K_{\mu_j}(z - K_{\delta_0}(z)) = \sum_{j=1}^N K_{\mu_j}(z).
		\]
		The convolution $\boxplus_{\cT_{N,\bool}}(\mu_1,\dots,\mu_N)$ is called the \emph{boolean convolution} of $\mu_1$, \dots, $\mu_N$ and it is commonly denoted $\biguplus_{j=1}^N \mu_j$ or $\mu_1 \uplus \dots \uplus \mu_N$ (see \cite{SW1997}).  Boolean convolution corresponds to addition of the $K$-transforms.  Hence, the binary boolean convolution operation $\uplus$ is commutative and associative.  Since the boolean convolution is independent of the order of the measures, we may unambiguously write $\biguplus_{s \in S} \mu_s$ where $S$ is a finite set.
	\end{example}
	
	\begin{example} \label{ex:orthogonal}
		Let $\cT_{\orth} = \{\emptyset, 1, 21\}$.  Then $\boxplus_{\cT_{\orth}}(\mu_1,\mu_2)$ is called the \emph{orthogonal convolution} and is denoted by $\mu_1 \vdash \mu_2$ (see \cite{Lenczewski2007}).  Note that $\br_1(\cT_{\orth}) = \{\emptyset,2\}$ and $\br_2(\cT_{\orth}) = \{\emptyset\}$ and $\br_2(\br_1(\cT_{\orth})) = \{\emptyset\}$.  Therefore,
		\[
		K_{\mu_1 \vdash \mu_2}(z) = K_{\mu_1}(z - K_{\mu_2}(z)) = K_{\mu_1} \circ F_{\mu_2}(z).
		\]
	\end{example}
	
	\begin{remark} \label{rem:decomposition}
		The fixed point equation \eqref{eq:fixedpointequation} can be expressed alternatively in terms of the boolean and orthogonal convolution as
		\begin{equation} \label{eq:fixedpointequation2}
			\boxplus_{\cT}(\mu_1,\dots,\mu_N) = \biguplus_{j \in [N] \cap \cT} \left( \mu_j \vdash \boxplus_{\br_j(\cT)}(\mu_1,\dots,\mu_N) \right).
		\end{equation}
		Iterating this formula enables us to express the convolution associated to any finite tree in terms of the boolean and orthogonal convolutions.  The case of compactly supported measures was already done in \cite[\S 6.3]{JekelLiu2020}, and this is a generalization of Lenczewski's earlier work on decompositions of the free convolution \cite{Lenczewski2007}.
	\end{remark}
	
	\begin{example} \label{ex:monotone}
		Let $\cT_{2,\mono} = \{\emptyset, 1, 2, 21\}$.  Then $\boxplus_{\cT_{2,\mono}}(\mu_1,\mu_2)$ is called the \emph{monotone convolution} of $\mu_1$ and $\mu_2$ and is denoted $\mu_1 \rhd \mu_2$ (see \cite{Muraki2000,Muraki2001}).  Computing iteratively with \eqref{eq:fixedpointequation} yields
		\begin{equation} \label{eq:monotonedecomposition}
			\mu_1 \rhd \mu_2 = (\mu_1 \vdash \mu_2) \uplus \mu_2
		\end{equation}
		or equivalently
		\[
		K_{\mu_1 \rhd \mu_2}(z) = K_{\mu_1}(z - K_{\mu_2}(z)) + K_{\mu_2(z)},
		\]
		which implies that $F_{\mu_1 \rhd \mu_2} = F_{\mu_1} \circ F_{\mu_2}$.  In the next section, we will use two more simple identities relating boolean, monotone, and orthogonal convolution.  First,
		\begin{equation} \label{eq:orthmonoassoc}
			\lambda \vdash (\mu \rhd \nu) = (\lambda \vdash \mu) \vdash \nu,
		\end{equation}
		holds because $K_\lambda \circ (F_\mu \circ F_\nu) = (K_\lambda \circ F_\mu) \circ F_\nu$.  Second,
		\begin{equation} \label{eq:boolorthoassoc}
			\left( \biguplus_{j=1}^N \mu_j \right) \vdash \nu = \biguplus_{j=1}^N (\mu_j \vdash \nu)
		\end{equation}
		holds because $(\sum_{j=1}^N K_{\mu_j}) \circ F_\nu = \sum_{j=1}^N K_{\mu_j} \circ F_\nu$.
	\end{example}
	
	\begin{example} \label{ex:free}
		Let us explain the connection between $\cT_{2,\free}$ and prior work on free convolution more precisely.  The free convolution $\mu \boxplus \nu$ is defined in \cite{BV1992} by the relation that
		\[
		F_{\mu \boxplus \nu}^{-1} = F_\mu^{-1} + F_\nu^{-1} - \id
		\]
		holds in a non-tangential neighborhood of $\infty$ in the upper half-plane.  In order to show that $F_{\boxplus_{\cT_{2,\free}}(\mu,\nu)}^{-1} = F_\mu^{-1} + F_\nu^{-1} - \id$, we look at \eqref{eq:fixedpointequation2} says in the case of $\cT_{2,\free}$, which of course entails looking at the branches of $\cT_{2,\free}$.  Let $\cT_{\sub} = \{\emptyset,1,21,121,\dots\}$, and let $\cT_{\sub}^\dagger = \{\emptyset,2,12,212,\dots\}$.  We observe that
		\begin{align*}
		\br_1(\cT_{2,\free}) &= \cT_{\sub}^\dagger, &
		\br_2(\cT_{2,\free}) &= \cT_{\sub}, \\
		\br_1(\cT_{\sub}) &= \cT_{\sub}^\dagger, &
		\br_2(\cT_{\sub}) &= \varnothing, \\
		\br_1(\cT_{\sub}^\dagger) &= \varnothing &
		\br_2(\cT_{\sub}^\dagger) &= \cT_{\sub} .
		\end{align*}
		Thus, \eqref{eq:fixedpointequation2} yields
		\begin{align*}
			\boxplus_{\cT_{2,\free}}(\mu,\nu) &= (\mu \vdash \boxplus_{\cT_{\sub}^\dagger}(\mu,\nu)) \uplus (\nu \vdash \boxplus_{\cT_{\sub}}(\mu,\nu)) \\
			\boxplus_{\cT_{\sub}}(\mu,\nu) &= \mu \vdash \boxplus_{\cT_{\sub}^\dagger}(\mu,\nu) \\
			\boxplus_{\cT_{\sub}^\dagger}(\mu,\nu) &= \mu \vdash \boxplus_{\cT_{\sub}}(\mu,\nu).
		\end{align*}
		Back-substituting the last two equations into the first yields
		\begin{align*}
		\boxplus_{\cT_{2,\free}}(\mu,\nu) = \boxplus_{\cT_{\sub}}(\mu,\nu) \uplus \boxplus_{\cT_{\sub}^\dagger}(\mu,\nu).
		\end{align*}
		Similarly, using back-substitution and \eqref{eq:monotonedecomposition},
		\[
		\boxplus_{\cT_{2,\free}}(\mu,\nu) = (\mu \vdash \boxplus_{\cT_{\sub}^\dagger}(\mu,\nu)) \uplus \boxplus_{\cT_{\sub}^\dagger}(\mu,\nu) = \mu \rhd \boxplus_{\cT_{\sub}^\dagger}(\mu,\nu).
		\]
		and symmetrically, $\boxplus_{\cT_{2,\free}}(\mu,\nu) = \nu \rhd \boxplus_{\cT_{\sub}}(\mu,\nu)$.  In terms of the $F$-transform, this means that
		\[
		F_{\boxplus_{\cT_{2,\free}}(\mu,\nu)} = F_{\boxplus_{\cT_{\sub}}(\mu,\nu)} + F_{\boxplus_{\cT_{\sub}^\dagger}(\mu,\nu)} - \id = F_\mu \circ F_{\boxplus_{\cT_{\sub}^{\dagger}}(\mu,\nu)} = F_\nu \circ F_{\boxplus_{\cT_{\sub}}(\mu,\nu)}.
		\]
		Hence, in a non-tangential neighborhood of $\infty$, we have
		\begin{align*}
		(F_\mu^{-1} + F_\nu^{-1} - \id) \circ F_{\boxplus_{\cT_{2,\free}}(\mu,\nu)} &= F_\mu^{-1} \circ F_{\boxplus_{\cT_{2,\free}}(\mu,\nu)} + F_\nu^{-1} \circ F_{\boxplus_{\cT_{2,\free}}(\mu,\nu)} - F_{\boxplus_{\cT_{2\,free}}(\mu,\nu)} \\
		&= F_{\boxplus_{\cT_{\sub}^\dagger}(\mu,\nu)} + F_{\boxplus_{\cT_{\sub}}(\mu,\nu)} - F_{\boxplus_{\cT_{2,\free}}(\mu,\nu)} \\
		&= \id,
		\end{align*}
		so that $F_\mu^{-1} + F_\nu^{-1} - \id = F_{\boxplus_{\cT_{2,\free}}}^{-1}$ and therefore, $\boxplus_{\cT_{2,\free}}(\mu,\nu) = \mu \boxplus \nu$ as desired.
		
		In the process of the argument, we showed that $F_{\mu \boxplus \nu} = F_\nu \circ F_{\boxplus_{\cT_{\sub}}(\mu,\nu))}$, which means in particular that $F_{\mu \boxplus \nu}$ is \emph{analytically subordinated} to $F_\nu$ as functions on the upper half-plane; this result has been studied by many authors in free probability \cite[Proposition 4.4]{VoiculescuFE1}, \cite[Theorem 3.1]{Biane1998}, \cite{Voiculescu2000}, \cite{Voiculescu2002b}, and \cite{BMS2013}, \cite[\S 7]{Lenczewski2007}, \cite{Nica2009}, \cite[Proposition 7.2]{Liu2018}.  The convolution operation associated to $\cT_{2,\sub}$ is called the \emph{subordination convolution} and is denoted $\mu \boxright \nu$. Furthermore, it is easy to check (and follows from Proposition \ref{prop:convolutionidentity} below) that $\boxplus_{\cT_{\sub}^\dagger}(\mu,\nu) = \boxplus_{\cT_{\sub}}(\nu,\mu) = \nu \boxright \mu$.
		
		The above relations between free and subordination convolutions imply that $F_{\mu \boxright \nu}$ and $F_{\nu \boxright \mu}$ satisfy the fixed-point equation system
		\begin{align*}
		F_{\mu \boxright \nu} &= \id - K_\mu \circ F_{\nu \boxright \mu} \\
		F_{\nu \boxright \mu} &= \id - K_\nu \circ F_{\mu \boxright \nu}.
		\end{align*}
		In order to study the subordination theory for free convolution, \cite{BMS2013} used iteration to construct solutions for this fixed-point equation system (and this was done in the more general operator-valued setting).  In fact, the iterates from their paper are, in the notation of our proof of Theorem \ref{thm:complexconvolution}, exactly $F_{\cT_{\sub},\mu,\nu}^{(n)}$ and $F_{\cT_{\sub}^\dagger,\mu,\nu}^{(n)}$.  Hence, our fixed-point iteration is a direct generalization of the one used for subordination convolution.  However, the subordination case is simpler in that $F_{\cT_{\sub},\mu,\nu}^{(n+1)}$ and $F_{\cT_{\sub}^\dagger,\mu,\nu}^{(n+1)}$ are computed in terms of $F_{\cT_{\sub},\mu,\nu}^{(n)}$ and $F_{\cT_{\sub}^\dagger,\mu,\nu}^{(n)}$ and $K_\mu$ and $K_\nu$; no other trees besides $\cT_{\sub}$ and $\cT_{\sub}^\dagger$ are involved in the computation because $\{\cT_{\sub},\cT_{\sub}^\dagger,\varnothing\}$ is closed under the branch operations.
		
		One can check also that $F_{\cT_{\sub},\mu,\nu}^{(n)} = F_{\cT_{\sub}^{(n)},\mu,\nu}$, where $\cT_{\sub}^{(n)}$ is the truncation of the tree $\cT_{\sub}$ to depth $n$, and
		\[
		\boxplus_{\cT_{\sub}^{(n)}}(\mu,\nu) = \underbrace{\mu \vdash (\nu \vdash (\mu \vdash \dots ))}_{n \text{ terms}}.
		\]
		Hence, $\mu \boxright \nu$ is the limit of iterated orthogonal convolutions of $\mu$ and $\nu$, which was observed by Lenczewski \cite{Lenczewski2007}.
	\end{example}
	
	\section{Convolution and the operad structure} \label{sec:operadstuff}
	
	In this section, we describe how the convolution operation of Theorem \ref{thm:complexconvolution} relates to the operations in the operad $\Tree$.  We start out with two propositions that prove permutation-equivariance as well as more general convolution identities.  We remark that Propositions \ref{prop:convolutionidentity} and \ref{prop:convolutionidentity2} imply that all the same convolution identities as in \cite[\S 6]{JekelLiu2020} hold for arbitrary probability measures on $\R$ since the only ingredients needed in the proofs are the relations \eqref{eq:convolutionidentity} and \eqref{eq:convolutionidentity2}.
	
	\begin{proposition} \label{prop:convolutionidentity}
		Let $\psi: [N] \to [N']$ be surjective.  Let $\Tree(\psi)$ be the set of trees $\cT \in \Tree(N)$ such that $\psi_*(s)$ is alternating for every $s \in \cT$ and such that $\psi_*|_{\cT}$ is injective.  Let $\mu_1$, \dots, $\mu_{N'} \in \cP(\R)$ and $\cT \in \Tree(\psi)$.  Then
		\begin{equation} \label{eq:convolutionidentity}
			\boxplus_{\cT}(\mu_{\psi(1)},\dots,\mu_{\psi(N)}) = \boxplus_{\psi_*(\cT)}(\mu_1,\dots,\mu_{N'}).
		\end{equation}
	\end{proposition}
	
	In the case of compactly supported measures, this proposition follows from \cite[Corollary 5.15]{JekelLiu2020}.  One can deduce the general case by continuity because compactly supported measures are dense in $\cP(\R)$.  But below we give an alternative self-contained argument directly from Theorem \ref{thm:complexconvolution}.
	
	\begin{proof}[Proof of Proposition \ref{prop:convolutionidentity}]
		Note that $\Tree(\psi)$ is closed under taking branches and rooted subtrees.  Therefore, finite trees are dense in $\Tree(\psi)$, so by continuity, it suffices to prove \eqref{eq:convolutionidentity} when $\cT$ is finite.  We proceed by induction on the depth of $\cT$.  When the depth of $\cT$ is zero, \eqref{eq:convolutionidentity} holds because both sides are $\delta_0$.  For the inductive step, consider a finite tree $\cT$ of depth $d$.  By Theorem \ref{thm:complexconvolution},
		\[
		\boxplus_{\cT}(\mu_{\psi(1)},\dots,\mu_{\psi(N)}) = \biguplus_{j \in [N] \cap \cT} \left( \mu_{\psi(j)} \vdash \boxplus_{\br_j(\cT)}(\mu_{\psi(1)},\dots,\mu_{\psi(N)}) \right).
		\]
		Since $\psi_*|_{\cT}$ is injective, each neighbor $i$ of the root vertex in $\psi_*(\cT)$ is the image of a single neighbor $j$ of the root vertex in $\cT$.  Moreover, $\br_i(\psi_*(\cT)) = \psi_*(\br_j(\cT))$.  Since $\br_j(\cT)$ has depth strictly less than $d$, the inductive hypothesis implies that
		\[
		\boxplus_{\br_j(\cT)}(\mu_{\psi(1)},\dots,\mu_{\psi(N)}) = \boxplus_{\psi_*(\br_j(\cT))}(\mu_1,\dots,\mu_{N'}).
		\]
		Therefore, the above expression equals
		\[
		\boxplus_{\cT}(\mu_{\psi(1)},\dots,\mu_{\psi(N)}) = \biguplus_{i \in [N'] \cap \psi_*(\cT)} \left( \mu_i \vdash \boxplus_{\br_i(\psi_*(\cT))}(\mu_1,\dots,\mu_{N'}) \right) = \boxplus_{\psi_*(\cT)}(\mu_1,\dots,\mu_{N'}),
		\]
		which completes the inductive step and hence the proof.
	\end{proof}
	
	Since any function is the composition of a surjection and injection, to understand the general case of $\psi: [N] \to [N']$, all that is left is to handle the injective case.  In order to simplify notation, we restrict our attention to the canonical inclusion $[N] \to [N']$ for $N' > N$ that maps $j$ to itself.  Because of permutation-equivariance, whatever results we prove for this will have analogs for a general injective map.
	
	\begin{proposition} \label{prop:convolutionidentity2}
		Let $N < N'$.  Let $\iota: [N] \to [N']$ be the canonical inclusion.  Then for $\cT \in \Tree(N)$ and $\mu_1$, \dots, $\mu_{N'} \in \cP(\R)$, we have
		\begin{equation} \label{eq:convolutionidentity2}
			\boxplus_{\cT}(\mu_1,\dots,\mu_N) = \boxplus_{\iota_*(\cT)}(\mu_1,\dots,\mu_{N'}).
		\end{equation}
	\end{proposition}
	
	\begin{proof}
		Let $M(\cT) = \boxplus_{\iota_*(\cT)}(\mu_1,\dots,\mu_{N'})$.  Note that
		\[
		M(\cT) = \biguplus_{j \in [N'] \cap \iota_*(\cT)} \mu_j \vdash \boxplus_{\br_j(\iota_*(\cT))}(\mu_1,\dots,\mu_{N'}).
		\]
		But $[N'] \cap \iota_*(\cT) = \iota([N] \cap \cT)$ and $\br_j(\iota_*(\cT)) = \iota_*(\br_j(\cT))$.  Thus,
		\[
		M(\cT) = \biguplus_{j \in [N] \cap \cT} \mu_j \vdash \boxplus_{\iota_*(\br_j(\cT))}(\mu_1,\dots,\mu_{N'}) = \biguplus_{j \in [N] \cap \cT} \mu_j \vdash M(\br_j(\cT)).
		\]
		Thus, $M(\cT)$ satisfies the fixed-point equation \eqref{eq:fixedpointequation}.  It also depends continuously on $\cT$ since $\cT \mapsto \iota_*(\cT)$ is isometric.  Therefore, by Theorem \ref{thm:complexconvolution}, $M(\cT) = \boxplus_{\cT}(\mu_1,\dots,\mu_N)$.
	\end{proof}
	
	The next theorem shows that the convolution operation $\boxplus_{\cT}$ respects operad composition.
	
	\begin{theorem} \label{thm:convolutioncomposition}
		Let $\cT \in \Tree(k)$ and $\cT_1 \in \Tree(n_1)$, \dots, $\cT_k \in \Tree(n_k)$.  Let $N = n_1 + \dots + n_k$.  For each $j \in [k]$ and $i \in [n_j]$, let $\mu_{j,i} \in \mathcal{P}(\R)$.  Then we have
		\[
		\boxplus_{\cT(\cT_1,\dots,\cT_k)}(\mu_{1,1},\dots,\mu_{1,n_1},\dots \dots, \mu_{k,1},\dots,
		\mu_{k,n_k}) = \boxplus_{\cT}(\boxplus_{\cT_1}(\mu_{1,1},\dots,\mu_{1,n_1}),\dots,\boxplus_{\cT_k}(\mu_{k,1},\dots,\mu_{k,n_k})).
		\]
	\end{theorem}
	
	In the case of compactly supported measures, this result follows immediately from \cite[Corollary 5.13]{JekelLiu2020} taking $\cB = \C$.  Because compactly supported measures are dense in $\mathcal{P}(\R)$ and because of continuity of the convolution operations in Theorem \ref{thm:complexconvolution}, it follows that the identity holds for all measures in $\mathcal{P}(\R)$.  Although this argument is satisfactory, we will also present an alternative proof directly from Theorem \ref{thm:complexconvolution} that is self-contained and elucidates the connection between the fixed-point equation in Theorem \ref{thm:complexconvolution}, the operad structure, and the branch maps.
	
	\begin{proof}[Proof of Theorem \ref{thm:convolutioncomposition}]
		First, we prove the case of the theorem where $\cT = \cT_{2,\mono}$.  In other words, we want to establish the identity
		\begin{equation}\label{eq:identity_2mono}
			\boxplus_{\cT_{2,\mono}(\cT_1,\cT_2)}(\mu_{1,1},\dots,\mu_{1,n_1},\mu_{2,1},\dots,\mu_{2,n_2}) = \boxplus_{\cT_1}(\mu_{1,1},\dots,\mu_{1,n_1}) \rhd \boxplus_{\cT_2}(\mu_{2,1},\dots,\mu_{2,n_2})
		\end{equation}
		for $\cT_1 \in \Tree(n_1)$ and $\cT_2 \in \Tree(n_2)$ and for probability measures $\mu_{j,i}$.  Note that both sides depend continuously on $\cT_1$ and $\cT_2$, using continuity of the operad composition in $\Tree$ and continuity of the convolution operation in Theorem \ref{thm:complexconvolution}.  Therefore, it suffices to prove the statement when $\cT_1$ and $\cT_2$ are finite trees.
		
		We proceed by induction on the depth of $\cT_1$ plus the depth of $\cT_2$.  In the base case of combined depth $0$, we have $\cT_1 = \{\emptyset\}$, and hence $\cT_{2,\mono}(\cT_1,\cT_2) = \cT_2$ and $\boxplus_{\cT_1}(\mu_{1,1},\dots,\mu_{1,n_1}) = \delta_0$, so the claim holds.
		
		For the inductive step, consider trees $\cT_1$ and $\cT_2$ with combined depth $d$.  Let $\cT' = \cT_{2,\mono}(\cT_1,\cT_2)$. Let $\mu_1 = \boxplus_{\cT_1}(\mu_{1,1},\dots,\mu_{1,n_1})$ and $\mu_2 = \boxplus_{\cT_2}(\mu_{2,1},\dots,\mu_{2,n_2})$.  Note that 
		\[
		[n_1+n_2] \cap \cT' = \iota_1([n_1] \cap \cT_1) \sqcup \iota_2([n_2] \cap \cT_2).
		\]
		Thus, by equation (\ref{eq:fixedpointequation2}),
		\begin{multline*}
			\boxplus_{\cT'}(\mu_{1,1},\dots,\mu_{1,n_1},\mu_{2,1},\dots,\mu_{2,n_2})
			= \biguplus_{i \in [n_1] \cap \cT_1} \left( \mu_{1,i} \vdash \boxplus_{\br_{\iota_1(i)}(\cT')}(\mu_{1,1},\dots,\mu_{1,n_1},\mu_{2,1},\dots,\mu_{2,n_2}) \right) \\ \uplus \biguplus_{i \in [n_2] \cap \cT_2} \left( \mu_{2,i} \vdash \boxplus_{\br_{\iota_2(i)}(\cT')}(\mu_{1,1},\dots,\mu_{1,n_1},\mu_{2,1},\dots,\mu_{2,n_2}) \right).
		\end{multline*}
		Now by Lemma \ref{lem:branch}, letting $\psi_1: [2n_1+n_2] \to [n_1+n_2]$ be the map sending the first $n_1$ points monotonically onto $[n_1]$ and the last $n_1 + n_2$ points monotonically onto $[n_1+n_2]$, we have
		\begin{align*}
			\br_{\iota_1(i)}(\cT') &= \br_{\iota_1(i)}(\cT_{2,\mono}(\cT_1,\cT_2)) \\
			&= (\psi_1)_*[\cT_{2,\mono}(\br_i(\cT_1), \br_1(\cT_{2,\mono})(\cT_1,\cT_2))] \\
			&= (\psi_1)_*[\cT_{2,\mono}(\br_i(\cT_1), (\iota_2)_*(\cT_2))] \\
			&= \cT_{2,\mono}(\br_i(\cT_1),\cT_2).
		\end{align*}
		Similarly,
		\[
		\br_{\iota_2(i)}(\cT') = \br_{\iota_2(i)}(\cT_{2,\mono}(\cT_1,\cT_2)) = (\psi_2)_*[\cT_{2,\mono}(\br_i(\cT_2), \br_2(\cT_{2,\mono})(\cT_1,\cT_2))] = (\iota_2)_* (\br_i(\cT_2)).
		\]
		where $\psi_2: [n_1+2n_2] \to [n_1+n_2]$ sends the first $n_2$ coordinates monotonically onto $\iota_2([n_2])$ and the last $n_1 + n_2$ coordinates monotonically onto $[n_1+n_2]$.  Therefore, using the induction hypothesis,
		\begin{align*}
			\boxplus_{\br_{\iota_1(i)}(\cT')}(\mu_{1,1},\dots,\mu_{1,n_1},\mu_{2,1},\dots,\mu_{2,n_2}) &= \boxplus_{\br_i(\cT_1)}(\mu_{1,1},\dots,\mu_{1,n_1}) \rhd \boxplus_{\cT_2}(\mu_{2,1},\dots,\mu_{2,n_2}) \\
			&= \boxplus_{\br_i(\cT_1)}(\mu_{1,1},\dots,\mu_{1,n_1}) \rhd \mu_2,
		\end{align*}
		and by applying Proposition \ref{prop:convolutionidentity2} to $\iota_2$,
		\begin{align*}
			\boxplus_{\br_{\iota_2(i)}(\cT')}(\mu_{1,1},\dots,\mu_{1,n_1},\mu_{2,1},\dots,\mu_{2,n_2}) 
			&= \boxplus_{(\iota_2)_*[\br_i(\cT_2)]}(\mu_{1,1},\dots,\mu_{1,n_1},\mu_{2,1},\dots,\mu_{2,n_2}).\\
			&= \boxplus_{\br_i(\cT_2)}(\mu_{2,1},\dots,\mu_{2,n_2}).
		\end{align*}
		Therefore, using \eqref{eq:orthmonoassoc} and \eqref{eq:boolorthoassoc},
		\begin{align*}
			& \biguplus_{i \in [n_1] \cap \cT_1} \left( \mu_{1,i} \vdash \boxplus_{\br_{\iota_1(i)}(\cT')}(\mu_{1,1},\dots,\mu_{1,n_1},\mu_{2,1},\dots,\mu_{2,n_2}) \right) \\
			=& \biguplus_{i \in [n_1] \cap \cT_1} \left( \mu_{1,i} \vdash \left(
			\boxplus_{\br_i(\cT_1)}(\mu_{1,1},\dots,\mu_{1,n_1}) \rhd \mu_2 \right) \right) \\
			=& \biguplus_{i \in [n_1] \cap \cT_1} \left( \mu_{1,i} \vdash \left(
			\boxplus_{\br_i(\cT_1)}(\mu_{1,1},\dots,\mu_{1,n_1}) \rhd \mu_2 \right) \right) \\
			=& \biguplus_{i \in [n_1] \cap \cT_1} \left( \left( \mu_{1,i} \vdash
			\boxplus_{\br_i(\cT_1)}(\mu_{1,1},\dots,\mu_{1,n_1}) \right) \vdash \mu_2 \right) \\
			=& \left( \biguplus_{i \in [n_1] \cap \cT_1} \left( \mu_{1,i} \vdash
			\boxplus_{\br_i(\cT_1)}(\mu_{1,1},\dots,\mu_{1,n_1}) \right) \right) \vdash \mu_2 \\
			&= \mu_1 \vdash \mu_2.
		\end{align*}
		Similarly, 
		\begin{align*}
			&  \biguplus_{i \in [n_2] \cap \cT_2} \left( \mu_{2,i} \vdash \boxplus_{\br_{\iota_2(i)}(\cT')}(\mu_{1,1},\dots,\mu_{1,n_1},\mu_{2,1},\dots,\mu_{2,n_2}) \right) \\
			=& \biguplus_{i \in [n_2] \cap \cT_2} \left( \mu_{2,i} \vdash \boxplus_{\br_i(\cT_2)}(\mu_{2,1},\dots,\mu_{2,n_2}) \right) \\
			=& \mu_2.
		\end{align*}
		Therefore,
		\[
		\boxplus_{\cT'}(\mu_{1,1},\dots,\mu_{1,n_1},\mu_{2,1},\dots,\mu_{2,n_2}) = (\mu_1 \vdash \mu_2) \uplus \mu_2 = \mu_1 \rhd \mu_2
		\]
		as desired, which completes the inductive step.
		
		Finally, we begin the main argument to prove the general case of the theorem.  Let $\cT$, $\cT_1$, \dots, $\cT_k$ and $\mu_{j,i}$ be as in the theorem statement.  Let
		\[
		\mu_j = \boxplus_{\cT_j}(\mu_{j,1},\dots,\mu_{j,n_j}).
		\]
		Let
		\[
		M(\cT) = \boxplus_{\cT(\cT_1,\dots,\cT_k)}(\mu_{1,1},\dots,\mu_{1,n_1},\dots\dots, \mu_{k,1},\dots,\mu_{k,n_k}).
		\]
		Note that $\cT \mapsto M(\cT)$ is continuous because composition and convolution are continuous.  Thus, by Theorem \ref{thm:complexconvolution}, to show that $M(\cT) = \boxplus_{\cT}(\mu_1,\dots,\mu_k)$, it suffices to show that
		\begin{equation} \label{eq:compositionfixedpoint}
			M(\cT) = \biguplus_{j \in [k] \cap \cT} \left(\mu_j \vdash M(\br_j(\cT)) \right).
		\end{equation}
		Let $\cT' = \cT(\cT_1,\dots,\cT_k)$.  Let $\psi_j: [n_j + N] \to [N]$ map the first $n_j$ elements monotonically onto $\iota_j([n_j])$ and the last $N$ elements monotonically onto $[N]$.  Applying (\ref{eq:fixedpointequation2}), Lemma \ref{lem:branch}, Proposition \ref{prop:convolutionidentity}, \eqref{eq:identity_2mono}, and \eqref{eq:orthmonoassoc}, 
		\begin{align*}
			M(\cT) =& \biguplus_{\substack{j \in [k], i \in [n_j], \\  \iota_{j}(i)\in [N]\cap\cT'}} \left(\mu_{j,i} \vdash \boxplus_{\br_{\iota_j(i)}(\cT(\cT_1,\dots,\cT_k))}(\mu_{1,1},\dots,\mu_{k,n_k}) \right)\\
			=& \biguplus_{\substack{j \in [k], i \in [n_j], \\  \iota_{j}(i)\in [N]\cap\cT'}} \left(\mu_{j,i} \vdash \boxplus_{(\psi_j)_*[\cT_{2,\mono}(\br_i(\cT_j), \br_j(\cT)(\cT_1,\dots, \cT_k))]}(\mu_{1,1},\dots,\mu_{k,n_k}) \right)\\
			=& \biguplus_{j \in [k] \cap \cT} \biguplus_{i \in [n_j] \cap \cT_j} \left(\mu_{j,i} \vdash \boxplus_{\cT_{2,\mono}(\br_i(\cT_j), \br_j(\cT)(\cT_1,\dots, \cT_k))}(\mu_{j,1},\dots,\mu_{j,n_j}, \mu_{1,1},\dots,\mu_{k,n_k}) \right) \\
			=& \biguplus_{j \in [k] \cap \cT} \biguplus_{i \in [n_j] \cap \cT_j} \left(\mu_{j,i} \vdash (\boxplus_{\br_i(\cT_j)}(\mu_{j,1},\dots,\mu_{j,n_j})\rhd \boxplus_{ \br_j(\cT)(\cT_1,\dots, \cT_k)}(\mu_{1,1},\dots,\mu_{k,n_k}) )\right)\\
			=& \biguplus_{j \in [k] \cap \cT} \biguplus_{i \in [n_j] \cap \cT_j} \left((\mu_{j,i} \vdash \boxplus_{\br_i(\cT_j)}(\mu_{j,1},\dots,\mu_{j,n_j}))\vdash M(\br_j(\cT)) \right)\\
			=& \biguplus_{j \in [k] \cap \cT} \left(\mu_j \vdash M(\br_j(\cT)) \right),
		\end{align*}
		which demonstrates \eqref{eq:compositionfixedpoint} and hence finishes the proof.
	\end{proof}
	
	Knowing that the convolution operations respect the operad structure of $\Tree$, we can now discuss the examples of boolean, free, and monotone convolution in more generality.  However, we will not give detailed justification for the claims here because the boolean, free, and monotone convolution were already discussed in depth in \cite[\S 3.2, \S 5.5, and throughout]{JekelLiu2020}.
	
	\begin{example} \label{ex:boolean2}
	Let $\cT_{N,\bool} = \{\emptyset\} \cup [N]$.  We saw in Example \ref{ex:boolean} that $\boxplus_{\cT_{N,\bool}}$ is the $N$-fold boolean convolution.  Let $\id = \{\emptyset,1\} \in \Tree(1)$.  The operad identity
	\[
	\cT_{2,\bool}(\id,\cT_{2,\bool}) = \cT_{2,\bool}(\cT_{2,\bool},\id)
	\]
	can be checked by direct computation, and it implies that $\mu_1 \uplus (\mu_2 \uplus \mu_3) = (\mu_1 \uplus \mu_2) \uplus \mu_3$, that is, the binary boolean convolution operation is associative.  Furthermore, $\cT_{2,\bool}(\id,\cT_{2,\bool}) = \cT_{3,\bool}$ implies that the ternary boolean convolution can be obtained by iterating the binary boolean convolution.  More generally,
	\[
	\cT_{k,\bool}(\cT_{n_1,\bool},\dots,\cT_{n_k,\bool}) = \cT_{n_1+\dots+n_k,\bool}.
	\]
	Hence, the $N$-ary boolean convolution can be obtained by iterating lower order boolean convolutions.  Finally, $\cT_{N,\bool}$ is permutation-invariant and therefore $\boxplus_{\cT_{N,\bool}}$ is permutation-invariant.
	\end{example}
	
	\begin{example} \label{ex:free2}
	We saw in Example \ref{ex:free} that $\cT_{2,\free}$ produces the binary free convolution operation.  One can check that
	\[
	\cT_{2,\free}(\id,\cT_{2,\free}) = \cT_{2,\free}(\cT_{2,\free},\id) = \cT_{3,\free},
	\]
	and hence the binary free convolution is associative.  We also deduce that $\boxplus_{\cT_{3,\free}}(\mu_1,\mu_2,\mu_3) = (\mu_1 \boxplus \mu_2) \boxplus \mu_3$, so that $\boxplus_{\cT_{3,\free}}$ agrees with any other definition of the ternary free convolution.  Similar reasoning shows that $\cT_{N,\free}$ produces the $N$-ary free convolution; the free convolution can be obtained by iterating lower-order free convolutions; the free convolution is permutation-invariant.  Alternatively, the argument in Example \ref{ex:free} can be generalized to $N$ variables to show that
	\[
	F_{\boxplus_{\cT_{N,\free}}(\mu_1,\dots,\mu_N)}^{-1} - \id = \sum_{j=1}^N (F_{\mu_j}^{-1} - \id)
	\]
	on an appropriate domain.
	\end{example}
	
	\begin{example} \label{ex:monotone2}
	Let
	\[
	\cT_{N,\mono} := \{\emptyset\} \cup \{j_1 \dots j_\ell: N \geq j_1 > j_2 > \dots > j_\ell \geq 1, \ell \geq 1\}.
	\]
	Similar to the previous examples,
	\[
	\cT_{2,\mono}(\id,\cT_{2,\mono}) = \cT_{2,\mono}(\cT_{2,\mono},\id) = \cT_{3,\mono}.
	\]
	Hence, we have associativity of monotone convolution, and $\cT_{3,\mono}$ produces the ternary monotone convolution.  More generally,
	\[
	\cT_{k,\mono}(\cT_{n_1,\mono},\dots,\cT_{n_k,\mono}) = \cT_{n_1+\dots+n_k,\mono}.
	\]
	The mirror image of $\cT_{N,\mono}$ is
	\[
	\cT_{N,\mono \dagger} := \{\emptyset\} \cup \{j_1 \dots j_\ell: 1 \leq j_1 < j_2 < \dots < j_\ell \leq N, \ell \geq 1\},
	\]
	which relates to the anti-monotone convolution instead of the monotone convolution.  The permutation of $[N]$ that reverses the order of all the elements transforms $\cT_{N,\mono}$ into $\cT_{N,\mono \dagger}$, which corresponds to the fact that the anti-monotone convolution and monotone convolution are related by reversing the order of indices.
	\end{example}
	
	Our final observation is that the $\cT$-free convolution of several copies of the same measure depends only on the isomorphism class of $\cT$.  We remark that the case of compactly supported measures also follows from Theorem 7.8 and Proposition 7.19 (1) of \cite{JekelLiu2020}.
	
	\begin{lemma} \label{lem:isomorphismconvolution}
		Suppose $\cT_1 \in \Tree(N_1)$ and $\cT_2 \in \Tree(N_2)$.  If $\cT_1 \cong \cT_2$, then $\boxplus_{\cT_1}(\mu, \dots, \mu) = \boxplus_{\cT_2}(\mu, \dots, \mu)$ for all $\mu \in \mathcal{P}(\R)$.
	\end{lemma}
	
	\begin{proof}
		Clearly, if $\cT_1$ and $\cT_2$ are isomorphic, then their truncations $\cT_1^{(k)}$ and $\cT_2^{(k)}$ to depth $k$ are also isomorphic for every $k$.  Since $\cT_1^{(k)} \to \cT_1$ and $\cT_2^{(k)} \to \cT_2$ in $\Tree(N_1)$ and $\Tree(N_2)$ respectively, and since the convolution operations are continuous, it suffices to show that $\boxplus_{\cT_1^{(k)}}(\mu, \dots, \mu) = \boxplus_{\cT_2^{(k)}}(\mu, \dots, \mu)$.
		
		Therefore, to prove the lemma, it suffices to prove the case where $\cT_1$ and $\cT_2$ are finite trees.  We proceed by induction on the depth, the depth-$0$ case being trivial.  Let $\phi: \cT_1 \to \cT_2$ be an isomorphism.  By Observation \ref{obs:isomorphismbranch}, $\phi$ defines a bijection $[N_1] \cap \cT_1 \to [N_2] \cap \cT_2$, and $\br_j(\cT_1) \cong \br_{\phi(j)}(\cT_2)$ for each $j \in [N_1] \cap \cT_1$.  We may apply the induction hypothesis to each of these branches since they have strictly smaller depth than the original trees.  Hence,
		\begin{align*}
			\boxplus_{\cT_1}(\mu,\dots,\mu) &= \biguplus_{j \in [N_1] \cap \cT} \mu \vdash \boxplus_{\br_j(\cT_1)}(\mu,\dots,\mu) \\
			&= \biguplus_{j' \in [N_2] \cap \cT_2} \mu \vdash \boxplus_{\br_{j'}(\cT_2)}(\mu,\dots,\mu) = \boxplus_{\cT_2}(\mu,\dots,\mu),
		\end{align*}
		which completes the inductive step and hence the proof.
	\end{proof}

	\section{A general limit theorem} \label{sec:limit1}
	
	Bercovici and Pata \cite[Theorem 6.3]{BP1999} showed a bijection between limit theorems for classical, free, and boolean convolution in the following sense:  Given a sequence $(\mu_\ell)_{\ell \in \N}$ in $\mathcal{P}(\R)$ and a sequence $(k_\ell)_{\ell \in \N}$ in $\N$ tending to infinity, $\mu_\ell^{*k_\ell}$ converges weakly as $\ell \to \infty$ if and only if $\mu_\ell^{\boxplus k_\ell}$ converges if and only if $\mu_\ell^{\uplus k_\ell}$ converges weakly as $\ell \to \infty$.  Theorem \ref{thm:BP1} will generalize one direction of this result to trees $\cT$ with $n(\cT) > 1$; namely, we will show that if convergence holds for the boolean case, then it holds for all such trees $\cT$.  Applications of this result as well as open questions will be discussed in \S \ref{sec:limit2}.
	
	In preparation, we establish some notation.  For $\cT \in \Tree(N)$ and $\mu \in \mathcal{P}(\R)$, let
	\[
	\boxplus_{\cT}(\mu) := \boxplus_{\cT}(\underbrace{\mu,\dots,\mu}_{N \text{ times}}).
	\]
	We also use boolean convolution powers defined as follows:  For $c > 0$ and $\mu \in \mathcal{P}(\R)$, let $\mu^{\uplus c}$ be given by
	\[
	K_{\mu^{\uplus c}} = c K_\mu.
	\]
	For each $\mu$, such a measure $\mu^{\uplus c}$ exists because a function $K$ is the $K$-transform of a measure if and only if $K$ maps $\h$ to $-\overline{\h}$ and $K(z) / z \to 0$ as $z \to \infty$ in $\h$ non-tangentially.  Clearly, $\mu^{\uplus c}$ is well-defined since a measure is uniquely determined by its $K$-transform.  If $N \in \N$, then $\biguplus_{j=1}^N \mu = \mu^{\uplus N}$.  We also have $(\mu^{\uplus c_1})^{\uplus c_2} = \mu^{\uplus c_1 c_2}$.  Recall also Definition \ref{def:maxchildren} and Lemma \ref{lem:maxchildren}.
	
	\begin{theorem} \label{thm:BP1}
		Let $(\mu_\ell)_{\ell \in \N}$ be a sequence in $\mathcal{P}(\R)$ and let $(k_\ell)_{\ell \in \N}$ be a sequence of natural numbers tending to $\infty$.  Let $N \in \N$ and $\cT \in \Tree(N)$ with $n(\cT) > 1$.  If $\mu_\ell^{\uplus n(\cT)^{k_\ell}}$ converges to some probability measure $\nu$ as $\ell \to \infty$, then $\boxplus_{\cT^{\circ k_\ell}}(\mu_\ell)$ converges as $\ell \to \infty$ to some probability measure $\mathbb{BP}(\cT,\nu)$ only depending on $\cT$ and $\nu$.  Moreover, the convergence is uniform over all $\cT \in \Tree(N)$ with $n(\cT) > 1$.
	\end{theorem}
	
	Our proof relies on the following result, which gives certain continuity estimates for the $\cT$-free convolution operations that are independent of $N$.
	
	\begin{theorem} \label{thm:equicontinuity}
		For $N \in \N$, we define
		\[
		\Phi_N: \Tree(N) \times [0,1] \times \cP(\R) \to \cP(\R)
		\]
		by
		\[
		\Phi_N(\cT,c,\mu) := \begin{cases}
			\boxplus_{\cT}(\mu^{\uplus c/N}, \dots, \mu^{\uplus c/N})^{\uplus 1/c}, & c \in (0,1], \\
			\mu^{\uplus \frac{n(\cT)}{N}}, & c = 0.
		\end{cases}
		\]
		The map $\Phi_N$ satisfies the fixed-point equation
		\begin{equation} \label{eq:Phifixedpoint}
			K_{\Phi_N(\cT,c,\mu)}(z) = \frac{1}{N} \sum_{j \in [N] \cap \cT} K_\mu(z - c K_{\Phi_N(\br_j(\cT),c,\mu)}(z)).
		\end{equation}
		Moreover, the maps $(\Phi_N)_{N \in \N}$ have the following equicontinuity property:  For each compact $Y \subseteq \mathcal{P}(\R)$ and $\epsilon > 0$, there exists $\delta > 0$ such that for all $N$, for all $\cT_1$, $\cT_2 \in \Tree(N)$ and $c_1$, $c_2 \in [0,1]$ and $\mu \in Y$ and $\nu \in \cP(\R)$, if $\rho_N(\cT_1,\cT_2) + |c_1 - c_2| + d_L(\mu,\nu) < \delta$, then $d_L(\Phi_N(\cT_1,c_1,\mu),\Phi_N(\cT_2,c_2,\nu)) < \epsilon$.
	\end{theorem}
	
	\begin{proof}
		To check \eqref{eq:Phifixedpoint} for $c > 0$, observe that
		\begin{align*}
			K_{\Phi_N(\cT,c,\mu)}(z) &= \frac{1}{c} K_{\boxplus_{\cT}(\mu^{\uplus c/N},\dots,\mu^{\uplus c/N})}(z) \\
			&= \sum_{j \in [N] \cap \cT} \frac{1}{c} K_{\mu^{\uplus c/N}}(z - K_{\boxplus_{\br_j(\cT)}(\mu^{\uplus c/N},\dots,\mu^{\uplus c/N})}(z)) \\
			&= \frac{1}{N} \sum_{j \in [N] \cap \cT} K_\mu(z - c K_{\Phi_N(\br_j(\cT),c,\mu)}(z)).
		\end{align*}
		The case $c = 0$ is immediate and left to the reader.
		
		Now we turn to the claim about continuity.  We will show below that the family $(\Phi_N)_{N \in \N}$ is uniformly equicontinuous on $\Tree(N) \times [0,1] \times Y$.  By this, we mean more precisely that the functions are uniformly continuous with a modulus of continuity that is independent of $N$; even though the domains are different, equicontinuity still makes sense because we have fixed a metric $\rho_N$ for each $\Tree(N)$ from the beginning.  This claim about equicontinuity for each compact $Y$ is enough to finish the proof.  Indeed, if the conclusion of the theorem failed, then there would be a compact set $Y$ and $\epsilon > 0$ such that for each $k > 0$, there exist $\mu_k \in Y$ and $\nu_k \in \mathcal{P}(\R)$ and $N_k \in \N$ and $\cT_k, \cT_k' \in \Tree(N_k)$ and $c_k, c_k' \in [0,1]$ such that
		\[
		\rho_N(\cT_k,\cT_k') + |c_k - c_k'| + d_L(\mu_k,\nu_k) < 1/k, \qquad d_L(\Phi_{N_k}(\cT_k,c_k,\mu_k), \Phi_{N_k}(\cT_k',c_k',\nu_k)) \geq \epsilon.
		\]
		Then $Y' = Y \cup \{\nu_k: k \in \N\}$ would be compact, and the above conditions would contradict the equicontinuity on $\Tree(N) \times [0,1] \times Y'$.
		
		As before, the strategy is to reframe \eqref{eq:Phifixedpoint} as a fixed-point equation for some analytic function $\mathcal{F}$ on a Banach space $\mathcal{X}$ and apply the Earle-Hamilton theorem.  Fix $Y \subseteq \mathcal{P}(\R)$ compact, and fix
		\[
		1 > a_0 > a_1 > a_2 > 0, \qquad 0 < b_2 < b_1 < b_0 < c_0 < c_1 < c_2,
		\]
		and let $t$ be as in the conclusion of Lemma \ref{lem:transformhomeomorphism}.  Let $\mathcal{X}$ to be the space of sequences $(f_N)_{N \in \N}$ where $f_N: \Tree(N) \times [0,1] \times Y \times \Gamma_{a_0,tb_0,tc_0} \to \C$ and where $(f_N)_{N \in \N}$ is uniformly bounded and uniformly equicontinuous, with the norm given by
		\[
		\norm{(f_N)_{N \in \N}}_{\mathcal{X}} = \sup_{N \in \N} \norm{f_N}_{C(\Tree(N) \times [0,1] \times Y \times \Gamma_{a_0,tb_0,tc_0})};
		\]
		it is easy to check that this is a Banach space because uniform equicontinuity is preserved under limits in this norm.  Let
		\[
		\Omega = \left\{(f_N)_{N \in \N} \in \mathcal{X}: \overline{\bigcup_{N \in \N} \Ran(f_N)} \subseteq (\Gamma_{a_2,tb_2,tc_2})^\circ \right\},
		\]
		where $\Ran(f_N)$ denotes the range (image) of $f_N$.  Note that $\Omega$ is open in $\mathcal{X}$.  Define $\mathcal{F}: \Omega \to \mathcal{X}$ by
		\[
		\mathcal{F}((f_N)_{N \in \N}) := (g_N)_{N \in \N}, \quad \text{where} \quad  g_N(\cT,c,\mu,z) = z - \frac{1}{N} \sum_{j \in [N] \cap \cT} K_\mu((1-c)z + c f_N(\br_j(\cT),c,\mu,z)).
		\]
		The motivation for this definition is that $f_N(\cT,c,\mu,z)$ is intended to approximate $z - K_{\Phi_N(\cT,c,\mu)}(z)$, and hence the intended approximation for $z - c K_{\Phi_N(\br_j(\cT),c,\mu)}(z)$ is $(1 - c)z + c f_N(\cT,c,\mu,z)$.
		
		We must check that $(g_N)_{N \in \N}$ is actually in $\mathcal{X}$, that $\mathcal{F}$ is analytic, and $\mathcal{F}(\Omega)$ is separated by a positive distance from $\Omega^c$.  First, to show that $(g_N)_{N \in \N}$ is uniformly bounded and equicontinuous, one combines the following facts:
		\begin{enumerate}[(1)]
			\item The modulus of continuity of the map $\br_j$ (on its domain) is independent of $N$ since it is $e^{-1}$ Lipschitz.
			\item The map $\mu \mapsto K_\mu$ is continuous on $Y$ where we use the weak topology on $Y \subseteq \mathcal{P}(\R)$ and the topology of uniform convergence on $\Gamma_{a_2,tb_2,tc_2}$.  Hence, the map $(\mu,z) \mapsto K_\mu(z)$ is jointly continuous on $Y \times \Gamma_{a_2,tb_2,tc_2}$, hence uniformly continuous and uniformly bounded by compactness of $Y$ and $\Gamma_{a_2,tb_2,tc_2}$.
		\end{enumerate}
		To show the separation of $\mathcal{F}(\Omega)$ from $\Omega^c$, we proceed similarly to the proof of Theorem \ref{thm:complexconvolution}.  By our choice of $f_N$, we have
		\[
		f_N(\br_j(\cT),c,\mu,z) \in \Gamma_{a_2,tb_2,tc_2},
		\]
		and by convexity of $\Gamma_{a_2,tb_2,tc_2}$, we have
		\[
		(1 - c)z + c f_N(\br_j(\cT),c,\mu,z) \in \Gamma_{a_2,tb_2,tc_2}.
		\]
		Then by our choice of $t$,
		\[
		z - K_\mu((1 - c)z + c f_N(\br_j(\cT),c,\mu,z)) \in \Gamma_{a_0,tb_0,tc_0}.
		\]
		Hence, as in Theorem \ref{thm:complexconvolution}, $g_N(\cT,c,\mu,z)$ is a convex combination of points in $\Gamma_{a_0,tb_0,tc_0}$ and so is in $\Gamma_{a_0,tb_0,tc_0}$.  This implies the separation of $\mathcal{F}(\Omega)$ from $\Omega^c$.  The analyticity of $\mathcal{F}$ is straightforward to check as in the proof of Theorem \ref{thm:complexconvolution}.
		
		For each $N \in \N$, let $f_N^0(\cT,c,\mu,z) = z$.  Then the Earle-Hamilton theorem implies that $\mathcal{F}^{\circ n}((f_N^0)_{N \in \N})$ converges as $n \to \infty$ to the unique fixed point $(f_N)_{N \in \N}$.  As in the proof of Theorem \ref{thm:complexconvolution}, the iterates themselves are $F$-transforms of measures, and therefore the convergence extends to the entire upper half-plane.  And there is a measure $\Psi_N(\cT,c,\mu)$ such that $f_N(\cT,c,\mu,z) = F_{\Psi_N(\cT,c,\mu)}(z)$.  Because $(f_N)$ is uniformly equicontinuous, we see that $(\Psi_N)_{N \in \N}$ is uniformly equicontinuous, since uniform convergence of a sequence of $F$-transforms on $\Gamma_{a_2,tb_2,tc_2}$ is equivalent to weak convergence of the associated sequence of measures.  Finally, reversing our computations above shows that for $c > 0$, $\Phi_N(\cT,c,\mu)^{\uplus c}$ satisfies the fixed point equation defining $\boxplus_{\cT}(\mu^{\uplus (c/N)}, \dots, \mu^{\uplus (c/N)}) = \Phi_N(\cT,c,\mu)^{\uplus c}$.  We have also just shown that $\Phi_N(\cT,c,\mu)$ depends continuously on $\cT$, and thus Theorem \ref{thm:complexconvolution} implies that $\Psi_N = \Phi_N$ for $c > 0$; the $c = 0$ case can be checked directly.  Therefore, the equicontinuity properties proved for $\Psi_N$ hold for $\Phi_N$.
	\end{proof}
	
	\begin{theorem} \label{thm:BP2}
		Let $N \in \N$ and $\cT \in \Tree(N)$ with $n(\cT) > 1$.  For $\mu \in \mathcal{P}(\R)$, we have existence of the limit
		\[
		\mathbb{BP}(\cT,\mu) := \lim_{k \to \infty} \boxplus_{\cT^{\circ k}}(\mu^{\uplus \frac{1}{n(\cT)^k}}).
		\]
		Moreover, for each $N$, the convergence is uniform on $\{\cT \in \Tree(N): n(\cT) > 1\} \times Y$ for every compact subset of $\mathcal{P}(\R)$, and hence $\mathbb{BP}$ is a continuous map $\{\cT \in \Tree(N): n(\cT) > 1\} \times \mathcal{P}(\R) \to \mathcal{P}(\R)$.
	\end{theorem}
	
	\begin{remark}
		We call the map ``$\mathbb{BP}$''in honor of Bercovici and Pata's work \cite{BP1999}.
	\end{remark}
	
	\begin{proof}
		Let $\Tree(N,n) = \{\cT \in \Tree(N): n(\cT) = n\}$, which is a clopen subset of $\Tree(N)$.  Note that $\{\cT \in \Tree(N): n(\cT) > 1\} = \bigcup_{n=2}^N \Tree(N,n)$.
		
		Fix $n \in \{2,\dots,N\}$, and let $Y$ be a compact subset of $\mathcal{P}(\R)$, and we will show uniform convergence of $\boxplus_{\cT^{\circ k}}(\mu^{\uplus \frac{1}{n^k}})$ on $\Tree(N,n) \times Y$.  This of course will imply continuity of the limit function.  And to show uniform convergence, it suffices to show that the sequence is uniformly Cauchy with respect to the L\'evy distance $d_L$ since $(\mathcal{P}(\R),d_L)$ is complete.  Fix an integer
		\[
		M \geq \frac{N - 1}{n - 1} \geq 1.
		\]
		For each $\cT \in \Tree(N,n)$, we have by Lemma \ref{lem:maxchildren} and Observation \ref{obs:isomorphictrees} that
		\[
		m(\cT^{\circ k}) + 1 = m(\cT) \frac{n^k - 1}{n - 1} \leq \frac{N-1}{n-1}(n^k - 1) + 1 \leq \frac{N-1}{n-1} n^k \leq Mn^k.
		\]
		Therefore, by Observation \ref{obs:isomorphictrees}, there exists some tree $\cT_k \in \Tree(Mn^k)$ such that $\cT_k \cong \cT^{\circ k}$.  Hence, by Lemmas \ref{lem:isomorphismconvolution} and \ref{lem:isomorphismcomposition}, we have for $k, \ell \geq 1$, and $\mu \in \mathcal{P}(\R)$ that
		\[
		\boxplus_{\cT^{\circ (k+\ell)}}(\mu) = \boxplus_{\cT^{\circ k}}(\boxplus_{\cT^{\circ \ell}}(\mu)) = \boxplus_{\cT_k}(\boxplus_{\cT_\ell}(\mu)) = \boxplus_{\cT_k \circ \cT_\ell}(\mu).
		\]
		In particular,
		\begin{align*}
			\boxplus_{\cT^{\circ(k+\ell)}}(\mu^{\uplus \frac{1}{n^{k+\ell}}})
			&= \boxplus_{\cT_k}(\boxplus_{\cT_\ell}(\mu^{\uplus \frac{1}{n^{k+\ell}}})) \\
			&= \boxplus_{\cT_k}\left( \left( \boxplus_{\cT_\ell}((\mu^{\uplus M^2})^{\uplus \frac{1}{Mn^\ell Mn^k}})^{\uplus Mn^k} \right)^{\uplus \frac{1}{Mn^k}} \right) \\
			&= \Phi_{Mn^k}\left(\cT_k,1, \Phi_{Mn^\ell}\left(\cT_\ell,\frac{1}{Mn^k}, \mu^{\uplus M^2}\right) \right).
		\end{align*}
		To show the sequence is Cauchy, fix $\epsilon > 0$.  Note that $Y^{\uplus M} = \{\mu^{\uplus M}: \mu \in Y\}$ is compact because $Y$ is compact and $\mu \mapsto \mu^{\uplus M}$ is continuous by Theorem \ref{thm:complexconvolution}, and the same holds for $Y^{\uplus M^2}$.  Thus, by Theorem \ref{thm:equicontinuity}, there exists $\eta > 0$ such that for all $k$, for all $\lambda \in Y^{\uplus M}$ and $\nu \in \mathcal{P}(\R)$ and $\cT' \in \Tree(Mn^k)$, we have
		\[
		d_L(\lambda,\nu) < \eta \implies d_L(\Phi_{Mn^k}(\cT',1,\lambda), \Phi_{Mn^k}(\cT',1,\nu)) < \frac{\epsilon}{2}.
		\]
		In particular, this estimate applies with $\cT' = \cT_k$, for any $\cT \in \Tree(N,n)$.  Applying Theorem \ref{thm:equicontinuity} again, there exists $\delta > 0$ such that for all $\ell \in \N$, for all $\lambda \in Y^{\uplus M^2}$, for all $\cT \in \Tree(N,n)$, we have
		\[
		c \in [0,\delta) \implies d_L(\Phi_{Mn^\ell}(\cT_\ell,c,\lambda), \Phi_{Mn^\ell}(\cT_\ell,0,\lambda)) < \eta.
		\]
		Note that
		\[
		\Phi_{Mn^\ell}(\cT_\ell,0,\mu^{\uplus M^2}) = (\mu^{\uplus M^2})^{\uplus \frac{n^\ell}{Mn^\ell}} = \mu^{\uplus M}.
		\]
		Hence, if $k > -\log_n(M \delta)$ and $\mu \in Y$ and $\ell \geq 1$, then
		\[
		d_L(\Phi_{Mn^\ell}\left(\cT_\ell,\frac{1}{Mn^k},\mu^{\uplus M^2}), \mu^{\uplus M}\right) < \eta,
		\]
		hence
		\[
		d_L\left(\Phi_{Mn^k}\left(\cT_k,1, \Phi_{Mn^\ell}\left(\cT_\ell,\frac{1}{Mn^k}, \mu^{\uplus M^2}\right) \right), \Phi_{Mn^k}(\cT_k,1,\mu^{\uplus M}) \right) < \frac{\epsilon}{2}.
		\]
		So for $\mu \in Y$ and $\ell, \ell' \geq 1$,
		\[
		d_\ell\left( \boxplus_{\cT^{\circ(k+\ell)}}(\mu^{\uplus \frac{1}{n^{k+\ell}}}), \boxplus_{\cT^{\circ(k+\ell')}}(\mu^{\uplus \frac{1}{n^{k+\ell'}}}) \right) < \epsilon.
		\]
		Therefore, the sequence is uniformly Cauchy, as desired.
	\end{proof}
	
	\begin{proof}[Proof of Theorem \ref{thm:BP1}]
		Let $\cT \in \Tree(N)$ and $\mu_\ell, \nu \in \mathcal{P}(\R)$.  Suppose that $\nu_\ell := \mu_\ell^{\uplus n(\cT)^{k_\ell}} \to \nu$ as $\ell \to \infty$.  Let $Y \subseteq \mathcal{P}(\R)$ be a compact set containing all the measures $\nu_\ell$.  Theorem \ref{thm:BP2} implies uniform convergence of $\boxplus_{\cT^{\circ k_\ell}}(\lambda^{\uplus n(\cT)^{-k_\ell}}) \to \mathbb{BP}(\cT,\lambda)$ over $\lambda \in Y$ as $\ell \to \infty$.  Since $\mathbb{BP}(\cT,\lambda)$ is continuous and because of the uniform convergence, we can still take limits as $\ell \to \infty$ with $\lambda$ replaced by the sequence $\nu_\ell$ that depends on $\ell$.  Thus,
		\[
		\boxplus_{\cT^{\circ k_\ell}}(\mu_\ell) = \boxplus_{\cT^{\circ k_\ell}}(\nu_\ell^{\uplus \frac{1}{n(\cT)^{k_\ell}}}) \to \mathbb{BP}(\cT,\nu)
		\]
		as desired.  The convergence is uniform over $\cT \in \Tree(N)$ because the convergence in Theorem \ref{thm:BP2} is uniform.
	\end{proof}
	
	\begin{example}
	Theorem \ref{thm:BP1} relates to earlier work on free and monotone convolution as follows.  In light of Example \ref{ex:free2}, $\cT_{n,\free}^{\circ k} = \cT_{n^k,\free}$.  Hence, if $\mu_\ell \in \mathcal{P}(\R)$ and $k_\ell \in \N$ with $k_\ell \to \infty$, and if $\mu_\ell^{\uplus n^{k_\ell}} \to \nu$ as $\ell \to \infty$, then $\mu_\ell^{\boxplus n^{k_\ell}} \to \mathbb{BP}(\cT_{n,\free},\nu)$ as $\ell \to \infty$.  The same can be said for monotone convolution using Example \ref{ex:monotone2}.  This result can be deduced from \cite{BP1999} in the free case and \cite{AW2014} in the monotone case.  Of course, the results of \cite{BP1999} and \cite{AW2014} apply to arbitrary sequences $m_\ell$ tending to $\infty$ rather than only those of the form $m_\ell = n^{k_\ell}$.  The restriction on the size of indices is an artifact of our working with general trees $\cT \in \Tree(N)$, since in the general case it is unclear how to define an $m$-ary $\cT$-free convolution for all $m$.
	\end{example}
	
	\begin{remark}
	Although Theorem \ref{thm:BP1} does not recover the full free and monotone results, the techniques in this paper could still be useful in future work about more general limit theorems.  For instance, suppose that $(\cT_k)_{k \in \N}$ is a sequence of trees with $\cT_k \in \Tree(Mn_k,n_k)$ for some $M \in \N$ and $n_k \in \N$ with $n_k \to \infty$.  Suppose we could show using combinatorial methods that for every compactly supported measure $\mu \in \mathcal{P}(\R)$, the sequence $\boxplus_{\cT_k}(\mu^{\uplus 1/n_k}) = \Phi_{Mn_k}(\cT_k,1,\mu^{\uplus M})$ converges to some measure $\Lambda(\mu)$.  Then using the equicontinuity result of Theorem \ref{thm:BP2} and the density of compactly supported measures, $\Phi_{Mn_k}(\cT_k,1,\mu)$ converges as $k \to \infty$ for arbitrary $\mu \in \mathcal{P}(\R)$, and the limiting function $\Lambda(\mu)$ is continuous on $\mathcal{P}(\R)$.  Furthermore, the same argument as in Theorem \ref{thm:BP1} would show that if $\mu_k^{\uplus n_k} \to \nu$ as $k \to \infty$, then $\Phi_{Mn_k}(\cT_k,1,\mu_k) \to \Lambda(\nu)$.  For similar remarks in the context of the central limit theorem, see Proposition 8.9 and the following discussion in \cite{JekelLiu2020}.
	\end{remark}
	
	\section{Limit theorems for classical domains of attraction} \label{sec:limit2}
	
	Practically speaking, Theorem \ref{thm:BP1} means that any known limit theorems for additive boolean convolution implies a corresponding theorem for $\cT$-free convolution.  First, we have the following central limit theorem.  Below, if $\mu \in \mathcal{P}(\R)$ and $c \in \R$, then $c \cdot \mu$ denotes the dilation of $\mu$ by $c$, that is, the push-forward of $\mu$ by the function $t \mapsto ct$.
	
	\begin{proposition}[Central limit theorem] \label{prop:CLT}
		Let $\cT \in \Tree(N)$ with $n(\cT) > 1$ and let $\mu \in \mathcal{P}(\R)$ be a measure with mean zero and variance $1$.  Let $\nu_2$ be the Bernoulli distribution $(1/2)(\delta_{-1} + \delta_1)$.  Then
		\[
		\lim_{k \to \infty} n(\cT)^{-k/2} \boxplus_{\cT^{\circ k}}(\mu) = \mathbb{BP}(\cT,\nu_2),
		\]
		and the convergence is uniform in the L\'evy distance over all $\cT \in \Tree(N)$ with $n(\cT) > 1$.
	\end{proposition}
	
	We use the notation $\nu_2$ because the central limit theorem fits into a general class of limit theorems corresponding to stability indices $\alpha \in (0,2]$, which we discuss below.  The central limit distributions for boolean, free, and monotone independence were computed early on in the development of non-commutative probability theory, \cite{SW1997} for the boolean case, \cite{Voiculescu1985,Voiculescu1986} for the free case, and \cite{Muraki2000,Muraki2001} for the monotone case.  For another example, see \cite[Corollary 9.23]{JekelLiu2020}, which computes the central limit distribution for a tree $\cT$ where the root vertex has $n$ children, and all the other vertices have $d$ children.
	
	The proposition will be an immediate consequence of \cite[Theorem 3.4]{SW1997} and Theorem \ref{thm:BP1}, once we first establish the basic properties of dilations.
	
	\begin{lemma} ~ \label{lem:dilation}
		\begin{enumerate}[(1)]
			\item For $c \neq 0$, have $K_{c \cdot \mu}(z) = c K_\mu(z/c)$.
			\item For $\cT \in \Tree(N)$, we have
			\[
			\boxplus_{\cT}(c\cdot\mu_1,\dots,c\cdot\mu_N) = c \cdot \boxplus_{\cT}(\mu_1,\dots,\mu_N).
			\]
			\item When $n(\cT) > 1$, the map $\mathbb{BP}$ from Theorem \ref{thm:BP2} satisfies
			\[
			\mathbb{BP}(\cT,c \cdot \mu) = c \cdot \mathbb{BP}(\cT,\mu).
			\]
			\item When $n(\cT) > 1$, we have
			\[
			\boxplus_{\cT}(\mathbb{BP}(\cT,\mu)) = \mathbb{BP}(\cT,\mu^{\uplus n(\cT)}).
			\]
		\end{enumerate}
	\end{lemma}
	
	\begin{proof}
		(1) Note that
		\[
		G_{c \cdot \mu}(z) = \int_{\R} \frac{1}{z - ct}\,d\mu(t) = \frac{1}{c} \int_{\R} \frac{1}{z/c - t}\,d\mu(t) = \frac{1}{c} G_\mu(z/c).
		\]
		Hence, $F_{c \cdot \mu}(z) = c F_\mu(z/c)$ and $K_{c \cdot \mu}(z) = K_\mu(z / c)$.
		
		(2) In the case $c = 0$, both sides are $\delta_0$.  For $c \neq 0$, note that $c \cdot \boxplus_{\cT}(\mu_1,\dots,\mu_N)$ depends continuously on $\cT$ and satisfies the fixed-point equation
		\[
		K_{c \cdot \boxplus_{\cT}(\mu_1,\dots,\mu_N)}(z) = \sum_{j \in [N] \cap \cT} K_{c \cdot \mu_j}(z - K_{c \cdot \boxplus_{\br_j(\cT)}(\mu_1,\dots,\mu_N)}(z)),
		\]
		hence, by Theorem \ref{thm:complexconvolution}, we have the desired equality.
		
		(3) From (1) it follows that $(c \cdot \mu)^{\uplus t} = c \cdot (\mu^{\uplus t})$ for $c \in \R$ and $t > 0$.  Therefore, using (2),
		\begin{align*}
			\mathbb{BP}(\cT,c \cdot \mu)
			&= \lim_{k \to \infty} \boxplus_{\cT^{\circ k}}((c \cdot \mu)^{\uplus \frac{1}{n(\cT)^k}}) \\
			&= \lim_{k \to \infty} \boxplus_{\cT^{\circ k}}(c \cdot (\mu^{\uplus \frac{1}{n(\cT)^k}})) \\
			&= \lim_{k \to \infty} c \cdot \boxplus_{\cT^{\circ k}}(\mu^{\uplus \frac{1}{n(\cT)^k}}) \\
			&= c \cdot \mathbb{BP}(\cT,\mu).
		\end{align*}
		
		(4) Observe that
		\begin{align*}
			\boxplus_{\cT}(\mathbb{BP}(\cT,\mu)) &= \lim_{k \to \infty} \boxplus_{\cT}(\boxplus_{\cT^{\circ k}}(\mu^{\uplus \frac{1}{n(\cT)^k}})) \\
			&= \lim_{k \to \infty} \boxplus_{\cT^{\circ (k+1)}}((\mu^{\uplus n(\cT)})^{\uplus \frac{1}{n(\cT)^{k+1}}}) \\
			&= \mathbb{BP}(\cT,\mu^{\uplus n(\cT)}).  \qedhere
		\end{align*}
	\end{proof}
	
	\begin{proof}[Proof of Proposition \ref{prop:CLT}]
		It follows from \cite[Theorem 3.4]{SW1997} that 
		\[
		n(\cT)^{-k/2} \cdot \mu^{\uplus n(\cT)^k} = (n(\cT)^{-k/2} \cdot \mu)^{\uplus n(\cT)^k} \to (1/2)(\delta_{-1} + \delta_1).
		\]
		Therefore, the proposition follows from Theorem \ref{thm:BP1} and the fact that $n(\cT)^{-k/2} \cdot \boxplus_{\cT^{\circ k}}(\mu) = \boxplus_{\cT^{\circ k}}(n(\cT)^{-k/2} \cdot \mu)$.
	\end{proof}
	
	Following a similar strategy as Bercovici and Pata \cite{BP1999}, we can use Theorem \ref{thm:BP1} to prove analogs of classical limit theorems associated to other stable distributions.  To set the stage, we recall some terminology used in the classification of domains of attraction in classical probability theory; see \cite[\S 5]{BP1999}.
	
	\begin{definition}
		We say that two measures $\mu$ and $\nu$ are \emph{equivalent} if $\mu = a + b \cdot \nu$ for some $a \in \R$ and $b > 0$.  A measure $\mu$ is said to be \emph{$*$-stable} if its equivalence class is closed under the classical convolution operation $*$; \emph{$\boxplus$-stable} is defined analogously.
	\end{definition}
	
	\begin{definition}
		A function $f: [0,\infty) \to [0,\infty)$ \emph{varies slowly} if
		\[
		\lim_{y \to \infty} \frac{f(ty)}{f(y)} = 1 \text{ for } t > 0.
		\]
		We say that $f$ \emph{varies regularly with index $\alpha$} if $f(y) / y^\alpha$ varies slowly, or equivalently
		\[
		\lim_{y \to \infty} \frac{f(ty)}{t^\alpha f(y)} = 1 \text{ for } t > 0.
		\]
		We make the same definitions for functions only defined on $[m,\infty)$ for some $m > 0$.
	\end{definition}
	
	\begin{definition}
		We say that a measure $\mu$ belongs to $\mathcal{C}_2$ if the function $y \mapsto \int_{-y}^y t^2\,d\mu(t)$ varies slowly.
	\end{definition}
	
	\begin{definition}
		For $\alpha \in (0,2)$ and $\theta \in [-1,1]$, we say that $\mu$ belongs to $\mathcal{C}_{\alpha,\theta}$ if
		\begin{enumerate}[(1)]
			\item the function $y \mapsto \int_{-y}^y t^2 \,d\mu(t)$ varies regularly with index $2 - \alpha$;
			\item we have
			\[
			\lim_{t \to \infty} \frac{\mu((t,\infty)) - \mu((-\infty,-t)}{\mu((t,\infty)) + \mu((-\infty,-t)} = \theta.
			\]
		\end{enumerate}
	\end{definition}
	
	The following is a classical result due to \cite{Levy1937}, \cite{GK1954}.
	
	\begin{theorem}
		There exists a unique equivalence class of $*$-stable laws in each of the sets $\mathcal{C}_2$ and $\mathcal{C}_{\alpha,\theta}$ for $\alpha \in (0,2)$ and $\theta \in [-1,1]$.  Let $\nu_2^*$ and $\nu_{\alpha,\theta}^*$ be representatives of these equivalence classes.  Then for each $\mu \in \mathcal{C}_2$ or $\mu \in \mathcal{C}_{\alpha,\theta}$, there exists a sequence of measures $\mu_n \sim \mu$ such that $\mu_n^{*n} \to \nu_2$ or $\mu_n^{*n} \to \nu_{\alpha,\theta}$ respectively (that is, $\mu$ is in the domain of attraction of $\nu_2^*$ or $\nu_{\alpha,\theta}^*$).
	\end{theorem}
	
	Bercovici and Pata used this theorem together with their \cite[Theorem 6.3]{BP1999} to deduce limit laws for free and boolean convolution.  We want to do the same thing for $\cT$-free convolution.  One obstacle for the general case is that translation of measures does not behave well with respect to $\cT$-free convolutions.  If $c + \mu$ denotes the translation of $\mu$ by $c \in \R$, then we do not have $\boxplus_{\cT}(c + \mu) = n(\cT) c + \boxplus_{\cT}(\mu)$.  For instance, the measure $\delta_c \boxplus \mu = \delta_c * \mu = \mu \rhd \delta_c = c + \mu$ has $K$-transform equal to $K_\mu(z - c) + c$; however, $\delta_c \uplus \mu = \delta_c \rhd \mu$ has $K$-transform $K_\mu(z) + c$, and hence does not agree with $c + \mu$.
	
	In the case $\alpha \in (0,1)$, the measure has a large enough tail that the translation is irrelevant to the limiting behavior.  In the case $\alpha \in (1,2)$, it is known that any measure in $\mathcal{C}_{\alpha,\theta}$ has finite mean, and hence we will restrict our attention to the set of measures in $\mathcal{C}_{\alpha,\theta}$ with mean zero, which we denote by $\mathcal{C}_{\alpha,\theta}^0$.  The case $\alpha = 1$ is difficult because the mean may or may not be defined, and one must inevitably deal with drift, which brings up the tricky question of translation.  In Theorem \ref{thm:limittheorem}, we handle the cases of $\mathcal{C}_{\alpha,\theta}$ with $\alpha \in (0,1)$ and $\mathcal{C}_{\alpha,\theta}^0$ with $\alpha \in (1,2)$; the proof is based on Cauchy transforms and thus independent of the classical results.  For the case of $\alpha = 1$ and $\alpha = 2$, we will deduce a less sharp result from the classical theory and the work of Bercovici and Pata.
	
	\begin{proposition} \label{prop:limitmeasures}
		For $\alpha \in (0,2)$ and $\theta \in [-1,1]$, there is a measure $\nu_{\alpha,\theta}$ with
		\[
		K_{\nu_{\alpha,\theta}}(z) = \begin{cases} -(i - \theta \tan \frac{\pi \alpha}{2}) (-iz)^{1 - \alpha}, & \alpha \neq 1, \\
			2 \theta \log (-iz) - i \pi, & \alpha = 1,
		\end{cases}
		\]
		for $z$ in the upper half-plane, where we use the branch of the logarithm with argument in $(-\pi,\pi]$.  For $\alpha \in (0,1) \cup (1,2)$ and $c > 0$, we have $c \cdot \nu_{\alpha,\theta} = \nu_{\alpha,\theta}^{\uplus c^\alpha}$.  Moreover, for $c > 0$, we have $c \cdot \nu_{\alpha,\theta} = (\delta_{-2 \theta \log c} \uplus \nu_{\alpha,\theta})^{\uplus c}$.
	\end{proposition}
	
	\begin{proof}
		Let $K_{\alpha,\theta}$ be the function on the right-hand side.  One can verify by direct computation that $K_{\alpha,\theta}$ maps the upper half-plane into the lower half-plane and that $K_{\alpha,\theta}(z) / z \to 0$ as $z \to \infty$ non-tangentially.  Thus, by Corollary \ref{cor:FandK}, $K_{\alpha,\theta}$ is the $K$-transform of some measure $\nu_{\alpha,\theta}$.  The final claim follows from direct computation using Lemma \ref{lem:dilation} (1) and the definition of boolean convolution powers.
	\end{proof}
	
	\begin{theorem} \label{thm:limittheorem}
		Suppose that $\alpha \in (0,1) \cup (1,2)$, $\theta \in [-1,1]$, and $\mu \in \mathcal{C}_{\alpha,\theta}$.  If $\alpha \in (1,2)$, then assume in addition that $\mu$ has mean zero.  Then there exists some $\phi: [0,+\infty) \to [0,+\infty)$ which varies regularly with index $-1/\alpha$ such that for all $N$ and for all $\cT \in \Tree(N)$ with $n(\cT) > 1$,
		\[
		\phi(n(\cT)^k) \cdot \boxplus_{\cT^{\circ k}}(\mu) \to \mathbb{BP}(\cT,\nu_{\alpha,\theta}).
		\]
		For each $(\alpha,\theta)$ and for each $N$, the convergence is uniform over $\cT \in \Tree(N)$ with $n(\cT) > 1$.
	\end{theorem}
	
	For examples of the distributions $\mathbb{BP}(\cT,\nu_{\alpha,\theta})$, see Figures \ref{fig:pictures1} and \ref{fig:pictures2} in \S \ref{sec:questions}.  The proof of the theorem relies on the following characterization of $\mathcal{C}_{\alpha,\theta}$ in terms of the Cauchy transform, which is due to Bercovici and Pata.
	
	\begin{proposition}[{\cite[Proposition 5.10-5.11]{BP1999}}]
		Let $\alpha \in (0,1) \cup (1,2)$, $\theta \in [-1,1]$, and $\mu \in \mathcal{P}(\R)$.  In the case $\alpha > 1$, assume in addition that $\mu$ has mean zero.  Then $\mu \in \mathcal{C}_{\alpha,\theta}$ if and only if there exists some $f$ that varies regularly with index $-1 - \alpha$ such that
		\[
		G_\mu(iy) - \frac{1}{iy} = \left( i - \theta \tan \frac{\pi \alpha}{2} \right) f(y)(1 + o(1)) \text{ as } y \to \infty.
		\]
	\end{proposition}
	
	Although the proof of this proposition in \cite{BP1999} is correct, the statement contains a sign error.  Thus, we have corrected the $\theta$ to $-\theta$ in the statement of the proposition and in the definition of $\nu_{\alpha,\theta}$.  This result can be restated in terms of the $K$-transform and boolean convolution as follows.
	
	\begin{proposition} \label{prop:CandK}
		Let $\alpha \in (0,1) \cup (1,2)$.  Then the following are equivalent:
		\begin{enumerate}[(1)]
			\item $\mu \in \mathcal{C}_{\alpha,\theta}$ for $\alpha < 1$ or $\mu \in \mathcal{C}_{\alpha,\theta}^0$ for $\alpha > 1$.
			\item There exists a function $g$ that varies regularly with index $1 - \alpha$ such that
			\[
			K_\mu(iy) = -\left( i - \theta \tan \frac{\pi \alpha}{2} \right) g(y)(1 + o(1)) \text{ as } y \to \infty.
			\]
			\item There exists a slowly varying function $h$ such that
			\[
			c^{-1/\alpha} \cdot \mu^{\uplus c / h(c^{1/\alpha})} \to \nu_{\alpha,\theta}.
			\]
		\end{enumerate}
		Furthermore, in (2), we can take $g(y) = -\im K_\mu(iy)$.
	\end{proposition}
	
	\begin{remark}
		It follows immediately that $\nu_{\alpha,\theta} \in \mathcal{C}_{\alpha,\theta}$ when $\alpha \in (0,1)$ and $\nu_{\alpha,\theta} \in \mathcal{C}_{\alpha,\theta}^0$ when $\alpha \in (1,2)$.
	\end{remark}
	
	\begin{proof}
		(1) $\iff$ (2).  Observe that
		\[
		K_\mu(iy) = iy - F_\mu(iy) = iy F_\mu(iy)\left(G_\mu(iy) - \frac{1}{iy}\right) = -y^2 \left(G_\mu(y) - \frac{1}{iy}\right)(1 + o(1)).
		\]
		Of course, $f(y)$ varies regularly with index $-1-\alpha$ if and only if $g(y) = y^2 f(y)$ varies regularly with index $1-\alpha$.  Thus, the previous proposition immediately implies the case where $\alpha < 1$ and the case $\alpha > 1$ and $\mu$ has mean zero.  Now consider a general measure $\mu \in \mathcal{C}_{\alpha,\theta}$ for $\alpha > 1$. Let $c$ be the mean and let $\nu = (-c) + \mu$.  Then $K_\nu(z) = -c + K_\mu(z + c)$ and hence
		\[
		K_\mu(iy) = c - \left( i - \theta \tan \frac{\pi \alpha}{2} \right) g(y)(1 + o(1)) \text{ as } y \to \infty,
		\]
		where $g$ varies regularly with index $1 - \alpha < 0$. In particular, this implies that the second term on the right-hand side goes to zero.  Hence, the mean $c$ is uniquely recoverable from $K_\mu$.  Furthermore, the right-hand side has the form $-\left( i - \theta \tan \frac{\pi \alpha}{2} \right) g(y)(1 + o(1))$ where $g$ varies regularly with index $1 - \alpha$ if and only if $c = 0$.
		
		(2) $\implies$ (3).  Let $g$ be as in (2), and write $g(y) = y^{1-\alpha} h(y)$ for some slowly varying function $h$.  Because $K$-transforms are contained in the normal family $\Hol(\h,-\overline{\h})$ (where the target space is the closure in the Riemann sphere), to show $c^{-1/\alpha} \cdot \mu^{\uplus c / h(c^{1/\alpha})} \to \nu_{\alpha,\theta}$ it suffices to prove pointwise convergence of the $K$-transforms on the imaginary axis.  By the definition of the boolean convolution power,
		\[
		K_{c^{-1/\alpha} \cdot \mu^{\uplus c/h(c^{1/\alpha})}}(iy) = \frac{c^{1-1/\alpha}}{h(c^{1/\alpha})} K_\mu(c^{1/\alpha} iy).
		\]
		By (2), this is equal to
		\[
		-\frac{c^{1-1/\alpha}}{h(c^{1/\alpha})} \left(i - \theta \tan \frac{\pi \alpha}{2} \right) (c^{1/\alpha} y)^{1-\alpha} h(c^{1/\alpha}y)(1 + o_{c^{1/\alpha} y}(1)) = -\left(i - \theta \tan \frac{\pi \alpha}{2} \right) y^{1-\alpha} \frac{h(c^{1/\alpha}y)}{h(c^{1/\alpha})}(1 + o_{c^{1/\alpha} y}(1)),
		\]
		where the subscript on the $o(1)$ term means that it vanishes as $c^{1/\alpha} y \to \infty$.  If $y$ is fixed and $c \to \infty$, then because $g$ varies slowly, we obtain
		\[
		-\left(i - \theta \tan \frac{\pi \alpha}{2} \right) y^{1-\alpha} \frac{h(c^{1/\alpha}y)}{h(c^{1/\alpha})}(1 + o_{c^{1/\alpha} y}(1)) \to -\left(i - \theta \tan \frac{\pi \alpha}{2} \right) y^{1-\alpha} = K_{\nu_{\alpha,\theta}}(iy).
		\]
		
		(3) $\implies$ (2).  Suppose that (3) holds for some function $h$.  Let $g(y) = y^{1-\alpha} h(y)$, so that $g$ varies regularly with index $1 - \alpha$.  Observe that
		\begin{align*}
			K_\mu(c^{1/\alpha} i) &= c^{1-1/\alpha} h(c^{1/\alpha}) K_{c^{-1/\alpha} \mu^{\uplus c/h(c^{1/\alpha})}}(i) \\
			&= c^{1-1/\alpha} h(c^{1/\alpha}) (K_{\nu_{\alpha,\theta}}(i) + o(1)) \\
			&= \left(i - \theta \tan \frac{\pi \alpha}{2} \right) g(c^{1/\alpha})(1 + o(1)).
		\end{align*}
		where the error $o(1)$ goes to zero as $c \to \infty$.  Then we substitute $c = y^\alpha$ and obtain (2).
		
		For the final claim regarding $g$ in (2), observe that $-\im K_\mu(iy) \geq 0$ and $-\im K_\mu(iy) = g(y)(1 + o(1))$.  It is straightforward to check that this function varies regularly of index $1 - \alpha$.  (The $1 + o(1)$ term in the original theorem statement is complex-valued, but the one used here is positive.)  Thus, we can replace $g(y)$ with $K_\mu(iy)$ by absorbing $g(y) / K_\mu(iy)$ into the $1 + o(1)$ term.
	\end{proof}
	
	We also need the following facts about regularly varying functions.  They can be found in \cite{BGT1987}, but we include an elementary proof here for the reader's convenience.
	
	\begin{lemma}~ \label{lem:regularvariation}
		\begin{enumerate}[(1)]
			\item If $f$ varies regularly with index $\alpha$ and $a > 0$, then $a f$ varies regularly with index $\alpha$.
			\item If $f$ varies regularly with index $\alpha$ and if $\beta > 0$, then $f(y)^\beta$ and $f(y^\beta)$ regularly with index $\alpha \beta$.
			\item If $f$ varies regularly with index $\alpha$ and if $\beta \in \R$, then $y^\beta f(y)$ varies regularly with index $\alpha + \beta$.
			\item If $f$ is bounded above and below on any compact set and varies regularly with index $\alpha \neq 0$, then
			\[
			\lim_{y \to \infty} f(y) = \begin{cases} \infty, & \alpha > 0, \\
				0, & \alpha < 0.
			\end{cases}
			\]
			\item Let $f: [m,\infty) \to [0,\infty)$ be continuous and vary regularly with index $\alpha > 0$.  Let
			\[
			g(y) = \inf \{x \geq m: f(x) \geq y\}.
			\]
			Then $f \circ g(y) = y$ for $y > f(m)$, and $g$ varies regularly with index $1/\alpha$.
		\end{enumerate}
	\end{lemma}
	
	\begin{proof}
		Claims (1), (2), (3) are straightforward to check from the definition.
		
		To prove (4), suppose $\alpha > 0$ and write $f(y) = y^\alpha g(y)$, where $g$ varies slowly.  Then there exists $M > 0$ such that
		\[
		2^{-\alpha/2} \leq \frac{f(2y)}{f(y)} \leq 2^{\alpha/2} \text{ for } y \geq M.
		\]
		By hypothesis, $f$ is bounded below by some $\delta$ on the set $[M,2M]$. Any $y \geq M$ can be written as $2^n y'$ for $y' \in [M,2M]$ and $n \geq 0$, and then we have
		\[
		g(y) = g(2^n y') \geq 2^{-n \alpha/2} g(y') \geq M^{\alpha/2} (2^nM)^{-\alpha/2} \delta \geq M^{\alpha/2} \delta y^{-\alpha/2}.
		\]
		Therefore, $y^{\alpha/2} g(y)$ has a positive lower bound for sufficiently large $y$, which implies that $y^\alpha g(y) \to \infty$.  The case for $\alpha < 0$ follows by considering $1 / f$.
		
		(5) Because $f(x) \to \infty$ as $x \to \infty$, the infimum in the definition of $g$ is well-defined.  If $y > f(m)$, then by continuity $f$ must achieve the value $y$ by the intermediate value theorem.  Furthermore, $f(x) < y$ for $x$ in a neighborhood of $m$, and hence the infimum $x_0$ of $\{x: f(x) \geq y\}$ must be strictly larger than $m$.  Then we have $f(x) < y$ for $x < x_0$ and there is a sequence of points converging to $x_0$ from above that satisfy $f(x) \geq y$, so by continuity $f(x_0) = y$, or $f(g(y)) = y$.
		
		Because $f$ is bounded on any compact set, we must have $g(y) \to \infty$ as $y \to \infty$.  Given any $t > 0$ and $\epsilon > 0$, since $f$ varies regularly with index $\alpha$, we have
		\[
		\lim_{y \to \infty} \frac{f((1 + \epsilon)^{1/\alpha} c^{1/\alpha} g(y)}{(1 + \epsilon) c f(g(y))} = 1.
		\]
		If $y$ is large enough that the left-hand side is larger than $1 / (1 + \epsilon)$, then we obtain
		\[
		f((1 + \epsilon)^{1/\alpha} c^{1/\alpha} g(y)) \geq cy.
		\]
		Thus, by definition of $g$,
		\[
		g(cy) \leq (1 + \epsilon)^{1/\alpha} c^{1/\alpha} g(y),
		\]
		so that
		\[
		\limsup_{n \to \infty} \frac{g(cy)}{c^{1/\alpha} g(y)} \leq (1 + \epsilon)^{1/\alpha}.
		\]
		Since $\epsilon$ was arbitrary, the $\limsup$ is bounded above by $1$.  However, because the same thing holds with $c$ replaced by $1/c$, we get
		\[
		\liminf_{n \to \infty} \frac{c^{-1/\alpha} g(cy)}{g(c(1/c)y)} \geq 1.
		\]
		Therefore, $g$ varies regularly with index $1/\alpha$.
	\end{proof}
	
	We can conclude the proof as follows:
	
	\begin{proof}[Proof of Theorem \ref{thm:limittheorem}]
		Let $\mu \in \mathcal{C}_{\alpha,\theta}$ if $\alpha \in (0,1)$ and $\mu \in \mathcal{C}_{\alpha,\theta}^0$ if $\alpha \in (1,2)$.  Let $h$ be as in Proposition \ref{prop:CandK}.  By Lemma \ref{lem:regularvariation}, the function $c \mapsto c / h(c^{1/\alpha})$ varies regularly with index $1$.  Let $\psi(t)$ be the function associated to $t / h(t^{1/\alpha})$ as in Lemma \ref{lem:regularvariation} (5), so that $\psi(t) / h(\psi(t)^{1/\alpha}) = t$ for sufficiently large $t$ and $\psi(t)$ varies regularly with index $1$.  Then let $\phi(t) = \psi(t)^{-1/\alpha}$.  Then $\phi$ varies regularly with index $-1/\alpha$, and 
		\[
		\phi(t) \cdot \mu^{\uplus t} = \psi(t)^{-1/\alpha} \cdot \mu^{\uplus \psi(t) / h(\psi(t)^{1/\alpha})} \to \nu_{\alpha,\theta}.
		\]
		In particular, for each $\cT \in \Tree(N)$ with $n(\cT) > 1$, we have
		\[
		\phi(n(\cT)^k) \cdot \mu^{\uplus n(\cT)^k} \to \nu_{\alpha,\theta},
		\]
		and hence by Theorem \ref{thm:BP1}, we have
		\[
		\phi(n(\cT)^k) \cdot \boxplus_{\cT^{\circ k}}(\mu) \to \mathbb{BP}(\cT,\nu_{\alpha,\theta})
		\]
		for $\cT \in \Tree(N)$ with $n(\cT) > 1$.  That theorem also implies that the convergence is uniform over $\cT$.
	\end{proof}
	
	By appealing to Theorem \ref{thm:BP1}, we did not have to check that $\mathcal{C}_{\alpha,\theta}$ is closed under the operations $\mu \mapsto \boxplus_{\cT}(\mu)$ or $\mu \mapsto \mathbb{BP}(\cT,\mu)$ in order to prove Theorem \ref{thm:limittheorem}.  However, as one would intuitively hope, this is indeed the case.
	
	\begin{proposition} \label{prop:Cclosure}
		Let $\alpha \in (0,1)$ and $\theta \in [-1,1]$.  Suppose $\mu \in \mathcal{C}_{\alpha,\theta}$.
		\begin{enumerate}[(1)]
			\item For $t > 0$, we have $\mu^{\uplus t} \in \mathcal{C}_{\alpha,\theta}$.
			\item For any $\cT \in \Tree(N)$ with $n(\cT) > 1$, we have $\boxplus_{\cT}(\mu) \in \mathcal{C}_{\alpha,\theta}$
			\item For $\cT \in \Tree(N)$ with $n(\cT) > 1$, we have $\mathbb{BP}(\cT,\mu) \in \mathcal{C}_{\alpha,\theta}$.
		\end{enumerate}
		The same claims hold with $\mathcal{C}_{\alpha,\theta}$ replaced by $\mathcal{C}_{\alpha,\theta}^0$ for $\alpha \in (1,2)$.
	\end{proposition}
	
	\begin{proof}
		Let $\alpha \in (0,1)$.  By Proposition \ref{prop:CandK} (3) and Lemma \ref{lem:regularvariation} (2) and (3), we have $\mu \in \mathcal{C}_{\alpha,\theta}$ if and only if there is a function $f$ that varies regularly with index $1$ such that $c^{-1/\alpha} \mu^{\uplus f(c)} \to \nu_{\alpha,\theta}$ as $c \to \infty$.
		
		In the remainder of the argument, assume $\mu \in \mathcal{C}_{\alpha,\theta}$ and let $f$ be a function that varies regularly with index $1$ with $c^{-1/\alpha} \mu^{\uplus f(c)} \to \nu_{\alpha,\theta}$.
		
		(1) If $t > 0$, then $c^{-1/\alpha} \cdot (\mu^{\uplus t})^{\uplus f(c)/t} = c^{-1/\alpha} \cdot \mu^{\uplus f(c)} \to \nu_{\alpha,\theta}$.  The function $f(c) / t$ also varies regularly with index $1$, so $\mu^{\uplus t} \in \mathcal{C}_{\alpha,\theta}$.
		
		(2) Let $\Phi_N$ be as in Theorem \ref{thm:equicontinuity}.  Then using Lemma \ref{lem:dilation}, we have
		\begin{align*}
			c^{-1/\alpha} \cdot \boxplus_{\cT}(\mu)^{\uplus f(c)/n(\cT)} &= \boxplus_{\cT}((c^{-1/\alpha} \cdot \mu^{\uplus f(c)})^{\uplus 1/f(c)})^{\uplus f(c)/n(\cT)} \\
			&= \Phi_N\Bigl(\cT,N/f(c),(c^{-1/\alpha} \cdot \mu^{\uplus f(c)})^{\uplus 1/f(c)} \Bigr)^{\uplus N / n(\cT)}.
		\end{align*}
		Of course, $f(c) \to \infty$ as $c \to \infty$.  Thus, by joint continuity of $\Phi_N$, we obtain that
		\[
		\lim_{c \to \infty} c^{-1/\alpha} \cdot \boxplus_{\cT}(\mu)^{\uplus f(c)/n(\cT)} = \Phi_N(\cT,0,\nu_{\alpha,\theta})^{\uplus N/n(\cT)} = (\nu_{\alpha,\theta}^{\uplus n(\cT)/N})^{\uplus N/n(\cT)} = \nu_{\alpha,\theta}.
		\]
		Thus, $\boxplus_{\cT}(\mu)$ satisfies the desired condition with the function $f / n(\cT)$.
		
		(3) Let $n = n(\cT)$.  As in the proof of Theorem \ref{thm:BP2}, fix $M \geq (N-1)/(n - 1)$, and let $\cT_k \in \Tree(Mn^k)$ be isomorphic to $\mathcal{T}^{\circ k}$.  Recall that
		\[
		\mathbb{BP}(\cT,\mu) = \lim_{k \to \infty} \boxplus_{\cT^{\circ k}}(\mu^{\uplus 1/n^k}) = \lim_{k \to \infty} \boxplus_{\cT_k}(\mu^{\uplus 1/n^k})
		\]
		Then observe that
		\begin{align*}
			c^{-1/\alpha} \cdot \left(\boxplus_{\cT_k}\left(\mu^{\uplus \frac{1}{n^k}}\right) \right)^{\uplus f(c)} &= \boxplus_{\cT_k}\left((c^{-1/\alpha} \cdot \mu^{\uplus f(c)})^{\frac{1}{f(c)n^k}} \right)^{\uplus f(c)} \\
			&= \Phi_{Mn^k}\left(\cT_k, M / f(c), (c^{-1/\alpha} \cdot \mu^{\uplus f(c)}) \right)^{\uplus M}.
		\end{align*}
		By Theorem \ref{thm:equicontinuity}, we have
		\[
		\lim_{c \to \infty} \Phi_{Mn^k}\left(\cT_k, M / f(c), (c^{-1/\alpha} \cdot \mu^{\uplus f(c)}) \right)^{\uplus M} = \Phi_{Mn^k}(\cT_k,0,\nu_{\alpha,\theta})^{\uplus M} = \left(\nu_{\alpha,\theta}^{\uplus \frac{n^k}{Mn^k}} \right)^{\uplus M} = \nu_{\alpha,\theta},
		\]
		and the rate of convergence is uniform for all $k$.  Uniform convergence implies that
		\[
		\lim_{c \to \infty} \lim_{k \to \infty} c^{-1/\alpha} \cdot \left(\boxplus_{\cT_k}\left(\mu^{\uplus \frac{1}{n^k}}\right) \right)^{\uplus f(c)} = \lim_{k \to \infty} \lim_{c \to \infty} c^{-1/\alpha} \cdot \left(\boxplus_{\cT_k}\left(\mu^{\uplus \frac{1}{n^k}}\right) \right)^{\uplus f(c)},
		\]
		and hence
		\[
		\lim_{c \to \infty} c^{-1/\alpha} \mathbb{BP}(\cT,\mu)^{\uplus f(c)} = \nu_{\alpha,\theta},
		\]
		so $\mathbb{BP}(\cT,\mu) \in \mathcal{C}_{\alpha,\theta}$.
		
		This concludes the proof for $\alpha \in (0,1)$.  The same proof works for $\alpha \in (1,2)$ with $\mathcal{C}_{\alpha,\theta}$ replaced by $\mathcal{C}_{\alpha,\theta}^0$. 
	\end{proof}
	
	In the cases of $\alpha = 1$ and $\alpha = 2$, the tools which Bercovici and Pata used to prove the characterization of $\mathcal{C}_{\alpha,\theta}$ in terms of Cauchy transforms \cite[Propositions 5.10 and 5.11]{BP1999} are not available in the same form; specifically, \cite[Proposition 5.8]{BP1999} does not handle the case $\alpha = 2$ and the later parts of that proposition do not handle the case $\alpha = 1$.  To study the $\cT$-free convolution for the regions $\mathcal{C}_{1,\theta}$ and $\mathcal{C}_{2}$ requires either a much more delicate analysis or a different approach.  We will be content here to deduce limit theorems from the classical theory and Bercovici and Pata's results.
	
	\begin{proposition} \label{prop:limit1}
		Let $\theta \in [-1,1]$ and let $\mu \in \mathcal{C}_{1,\theta}$.  Then there exists a sequence of measures $(\mu_j)_{j \in \N}$ equivalent to $\mu$ such that, for all $N$, for all $\cT \in \Tree(N)$ with $n(\cT)$, we have
		\[
		\boxplus_{\cT^{\circ k}}(\mu_{n(\cT)^k}) \to \mathbb{BP}(\cT,\nu_{1,\theta}),
		\]
		where for each $N$, the convergence is uniform over $\cT$.
	\end{proposition}
	
	\begin{proposition} \label{prop:limit2}
		Let $\mu \in \mathcal{C}_2$ with mean zero.  Then there exists a sequence $R_j$ tending to infinity such that for all $N$, for all $\cT \in \Tree(N)$ with $n(\cT)$, we have
		\[
		R_{n(\cT)^k}^{-1} \cdot \boxplus_{\cT^{\circ k}}(\mu) \to \mathbb{BP}(\cT,\nu_2),
		\]
		where for each $N$, the convergence is uniform over $\cT$.
	\end{proposition}
	
	To set the stage for the proof, we recall the results from \cite{BP1999} in more detail.  The infinitely divisible distributions for $*$, $\boxplus$, and $\uplus$ are parametrized by a $\gamma \in \R$ and finite measure $\sigma$ on $\R$, and the infinitely divisible distributions corresponding to $(\gamma,\sigma)$ for the three convolutions are denoted respectively by $\nu_*^{\gamma,\sigma}$, $\nu_{\boxplus}^{\gamma,\sigma}$, and $\nu_{\uplus}^{\gamma,\sigma}$.  For a sequence of probability measures $\mu_j$ and $k_j \to \infty$, we have $\mu_j^{*k_j} \to \nu_*^{\gamma,\sigma}$ if and only if $\mu_j^{\boxplus k_j} \to \nu_{\boxplus}^{\gamma,\sigma}$ if and only if $\mu_j^{\uplus k_j} \to \nu_{\uplus}^{\gamma,\sigma}$.
	
	It follows that for every $N$,
	\[
	\mathbb{BP}(\cT_{N,\free}, \nu_{\uplus}^{\gamma,\sigma}) = \lim_{k \to \infty} ((\nu_{\uplus}^{\gamma,\sigma})^{\uplus 1/N^k})^{\boxplus N^k} = \nu_{\boxplus}^{\gamma,\sigma}.
	\]
	Moreover, let $\Phi_\mu$ denote the Voiculescu transform $\Phi_\mu(z) = F_\mu^{-1}(z) - z$, defined in a non-tangential neighborhood of $\infty$.  The correspondence between the free and boolean cases is such that
	\begin{equation} \label{eq:freebooleanBP}
		\Phi_{\nu_{\boxplus}^{\gamma,\sigma}}(z) = \gamma + \int_{\R} \frac{1 + tz}{z - t}\,d\mu(t) = K_{\nu_{\uplus}^{\gamma,\sigma}}(z).
	\end{equation}
	
	It follows from \cite[\S 5]{BP1999} that the freely stable laws correspond precisely to the classically stable laws. However, these do \emph{not} correspond to boolean stable laws in the na{\"i}ve sense.  Rather, for $a \in \R$, it follows from \eqref{eq:freebooleanBP} that
	\[
	\mathbb{BP}(\cT_{2,\free},\delta_a \uplus \mu) = \delta_a \boxplus \mathbb{BP}(\cT_{2,\free},\mu) = a + \mathbb{BP}(\cT_{2,\free},\mu),
	\]
	and thus stability in the boolean setting should be understood with respect to the shift operations $\mu \mapsto \delta_a \uplus \mu$ for $a \in \R$ rather than $\mu \mapsto a + \mu$.  The laws $\nu_{\alpha,\theta}$ in Proposition \ref{prop:limitmeasures} above are the boolean stable laws with this modified notion of stability, and the freely stable distributions in \cite[Proposition 5.12]{BP1999} are exactly the distributions $\mathbb{BP}(\cT_{N,\free},\nu_{\alpha,\theta})$, where $\nu_{\alpha,\theta}$.
	
	Proposition \ref{prop:limit1} is now proved as follows:  Let $\rho_{\alpha,\theta}$ be the classical stable distribution corresponding to the boolean infinitely divisible distribution $\nu_{\alpha,\theta}$. From classical results, if $\mu \in \mathcal{C}_{1,\theta}$, there are measures $\mu_j$ equivalent to $\mu$ such that $\mu_j^{*j} \to \rho_{1,\theta}$.  Hence by \cite[Theorem 6.3]{BP1999}, we have $\mu_j^{\uplus j} \to \nu_{\alpha,\theta}$.  Then by Theorem \ref{thm:BP1}, we have $\boxplus_{\cT^{\circ k}}(\mu_{n(\cT)^k}) \to \mathbb{BP}(\cT,\nu_{1,\theta})$.  The proof of Proposition \ref{prop:limit2} is the same.
	
	\section{Open questions} \label{sec:questions}
	
	The following questions around Theorems \ref{thm:BP1} and \ref{thm:limittheorem} remain unanswered.
	
	\begin{question}
		Does the converse implication hold in Theorem \ref{thm:BP1}?  More precisely, let $\cT \in \Tree(N)$ with $n(\cT) > 1$.  If $\boxplus_{\cT^{\circ k_\ell}}(\mu_\ell)$ converges as $\ell \to \infty$, then does $\mu_\ell^{\uplus n(\cT)^{k_\ell}}$ converge?  A positive answer is known for free independence by \cite{BP1999} and in the monotone case, provided that the limit measure is monotonically infinitely divisible, by \cite{AW2014}.
	\end{question}
	
	\begin{question}
		Is the map $\mathbb{BP}(\cT,\cdot)$ injective and is the inverse continuous?  For compactly supported measures, the inverse map was studied using combinatorial methods in \cite[\S 9]{JekelLiu2020}.  We anticipate that the answer to this question and the previous one will be easier in the case of finite variance than in the general case.
	\end{question}
	
	\begin{question}
		What is the correct notion of stable law for $\cT$-free convolution?  Do we get a classification of such laws that is parallel to the classical case?  Of course, this question is one of the main motivations for the previous two questions.
	\end{question}
	
	\begin{question}
		Is there a limit theorem which allows us to bring the translation operation outside the convolution operations?  That is, if $\mu \in \mathcal{C}_{\alpha,\theta}$ or $\mathcal{C}_2$, then can we describe the asymptotic behavior of some sequence $\mu_k \sim \boxplus_{\cT^{\circ k}}(\mu)$?
	\end{question}
	
	\begin{question}
		Does Proposition \ref{prop:Cclosure} generalize to the $\alpha = 1$ and $\alpha = 2$ cases?  What is the correct substitute for Proposition \ref{prop:CandK} in these cases?
	\end{question}
	
	There are many interesting questions about the limiting distributions $\mathbb{BP}(\cT,\nu_{\alpha,\theta})$ themselves.  We know that $\boxplus_{\cT^{\circ k}}(\nu_{\alpha,\theta}^{\uplus n(\cT)^{-k}}) \to \mathbb{BP}(\cT,\nu_{\alpha,\theta})$, and in the case $\alpha \in (0,1) \cup (1,2)$, we also have $\boxplus_{\cT^{\circ k}}(\nu_{\alpha,\theta}^{\uplus n(\cT)^{-k}}) = n(\cT)^{-1/\alpha} \cdot \boxplus_{\cT^{\circ k}}(\nu_{\alpha,\theta})$.  Furthermore, the Stieltjes inversion formula says that under sufficient regularity conditions, the probability density of a measure $\mu$ can be recovered from the Cauchy transform by $\rho(x) = \lim_{\epsilon \to 0^+} -\frac{1}{\pi} G_\mu(x + iy)$.  As an example, we considered the tree $\cT = \{\emptyset,1,2,3,21,31,12,13\}$.  To approximate the density for $\mathbb{BP}(\cT,\nu_{\alpha,\theta})$, we computed
	\[
	-\frac{1}{\pi} \im G_{\boxplus_{\cT^{\circ 6}}(\nu_{\alpha,\theta}^{\uplus 3^{-6}})}(x + i \epsilon)
	\]
	for $\epsilon = 10^{-5}$ and for values of $x$ spaced at intervals of $0.1$, and the results are shown in Figures \ref{fig:pictures1} and \ref{fig:pictures2}.  Experimentally, replacing $6$ by $7$ or shrinking $\epsilon$ did not change the values much.  However, because the size of the tree $\cT^{\circ k}$ increases very quickly with $k$, this approximation scheme has high computational complexity and is thus impractical to evaluate for large $k$.
	
	\begin{question}
		Are there practical numerical error bounds for the convergence of $\boxplus_{\cT^{\circ k}}(\nu_{\alpha,\theta}^{\uplus n(\cT)^{-k}}) \to \mathbb{BP}(\cT,\nu_{\alpha,\theta})$?  Similarly, what is the rate of convergence to $\boxplus_{\cT}(\mu_1,\dots,\mu_N)$ of the approximations given by truncation of $\cT$ to finite trees?  Are there better estimates for special classes of trees?
	\end{question}
	
	Already for a single tree $\cT = \{\emptyset,1,2,3,21,31,12,13\}$, we saw a variety of phenomena occur.  For $\alpha \in (0,1)$ there is a singularity at $0$ in the boolean case, but in the free case the stable laws have analytic densities on their supports \cite[Propositions A.1.2-A.1.4]{BP1999}.  For this $\cT$, the presence or absence of a singularity appears to depend on the value of $\alpha \in (0,1)$.  For $\alpha \in (1,2)$, the distribution can have several local extrema and inflection points.  By contrast, the free case, the stable distributions are unimodal \cite[Proposition A.2.2]{BP1999}; in the monotone and boolean cases, they are either unimodal or bimodal \cite{HS2015}.
	
	\begin{question}
		What can we say about the regularity of the limit distributions $\mathbb{BP}(\cT,\nu_{\alpha,\theta})$?  Do they have analytic densities?  How does this vary with $\cT$, $\alpha$, and $\theta$?  In general, what can we say about the regularity of $\cT$-free convolutions of several measures?  Under what conditions on $\cT$ do the regularity results from the free case \cite{Biane1997b,BV1998,Belinschi2003,BB2004,Belinschi2006b} generalize?
	\end{question}

	\begin{figure}
		
		\begin{tikzpicture}[yscale=5]
			
			\begin{scope}
				
				\node at (-2,0.5) {$\alpha = 1.7$};
				\node at (2,0.5) {$\theta = 0.0$};
				
				\draw[->] (0,-0.05) -- (0,0.5); \node[text=gray] at (0,-0.1) {$0$};
				\draw[<->] (-2.8,0) -- (2.8,0);
				\draw (-0.1,0.2) -- (0.1,0.2); \node[text=gray] at (0.4,0.2) {$0.2$};
				\draw (-0.1,0.4) -- (0.1,0.4); \node[text=gray] at (0.4,0.4) {$0.4$};
				\draw (-1,-0.05) -- (-1,0); \node[text=gray] at (-1,-0.1) {$-1$};
				\draw (-2,-0.05) -- (-2,0); \node[text=gray] at (-2,-0.1) {$-2$};
				\draw (1,-0.05) -- (1,0); \node[text=gray] at (1,-0.1) {$1$};
				\draw (2,-0.05) -- (2,0); \node[text=gray] at (2,-0.1) {$2$};
				
				\draw[blue,smooth] plot coordinates {
					(-2.5000, 0.022745)
					(-2.4000, 0.026702)
					(-2.3000, 0.031700)
					(-2.2000, 0.038113)
					(-2.1000, 0.046481)
					(-2.0000, 0.057592)
					(-1.9000, 0.072593)
					(-1.8000, 0.093097)
					(-1.7000, 0.12115)
					(-1.6000, 0.15851)
					(-1.5000, 0.20423)
					(-1.4000, 0.25003)
					(-1.3000, 0.27997)
					(-1.2000, 0.28397)
					(-1.1000, 0.26847)
					(-1.0000, 0.24628)
					(-0.90000, 0.22799)
					(-0.80000, 0.22027)
					(-0.70000, 0.22483)
					(-0.60000, 0.24357)
					(-0.50000, 0.27865)
					(-0.40000, 0.31016)
					(-0.30000, 0.30499)
					(-0.20000, 0.29316)
					(-0.10000, 0.30895)
					(0.00000, 0.30970)
					(0.10000, 0.30895)
					(0.20000, 0.29316)
					(0.30000, 0.30499)
					(0.40000, 0.31016)
					(0.50000, 0.27865)
					(0.60000, 0.24357)
					(0.70000, 0.22483)
					(0.80000, 0.22027)
					(0.90000, 0.22799)
					(1.0000, 0.24628)
					(1.1000, 0.26847)
					(1.2000, 0.28397)
					(1.3000, 0.27997)
					(1.4000, 0.25003)
					(1.5000, 0.20423)
					(1.6000, 0.15851)
					(1.7000, 0.12115)
					(1.8000, 0.093097)
					(1.9000, 0.072593)
					(2.0000, 0.057592)
					(2.1000, 0.046481)
					(2.2000, 0.038113)
					(2.3000, 0.031700)
					(2.4000, 0.026702)
					(2.5000, 0.022745)
				};
			\end{scope}
			
			\begin{scope}[shift = {(0,-0.9)}]
				
				\node at (-2,0.5) {$\alpha = 1.7$};
				\node at (2,0.5) {$\theta = 0.4$};
				
				\draw[->] (0,-0.05) -- (0,0.5); \node[text=gray] at (0,-0.1) {$0$};
				\draw[<->] (-2.8,0) -- (2.8,0);
				\draw (-0.1,0.2) -- (0.1,0.2); \node[text=gray] at (0.4,0.2) {$0.2$};
				\draw (-0.1,0.4) -- (0.1,0.4); \node[text=gray] at (0.4,0.4) {$0.4$};
				\draw (-1,-0.05) -- (-1,0); \node[text=gray] at (-1,-0.1) {$-1$};
				\draw (-2,-0.05) -- (-2,0); \node[text=gray] at (-2,-0.1) {$-2$};
				\draw (1,-0.05) -- (1,0); \node[text=gray] at (1,-0.1) {$1$};
				\draw (2,-0.05) -- (2,0); \node[text=gray] at (2,-0.1) {$2$};
				
				\draw[blue,smooth] plot coordinates {
					(-2.5000, 0.015319)
					(-2.4000, 0.018272)
					(-2.3000, 0.022127)
					(-2.2000, 0.027281)
					(-2.1000, 0.034375)
					(-2.0000, 0.044479)
					(-1.9000, 0.059467)
					(-1.8000, 0.082779)
					(-1.7000, 0.12090)
					(-1.6000, 0.18543)
					(-1.5000, 0.28898)
					(-1.4000, 0.40342)
					(-1.3000, 0.42439)
					(-1.2000, 0.36062)
					(-1.1000, 0.29578)
					(-1.0000, 0.24349)
					(-0.90000, 0.21397)
					(-0.80000, 0.21059)
					(-0.70000, 0.22043)
					(-0.60000, 0.24928)
					(-0.50000, 0.30611)
					(-0.40000, 0.34203)
					(-0.30000, 0.31057)
					(-0.20000, 0.29065)
					(-0.10000, 0.30978)
					(0.00000, 0.30391)
					(0.10000, 0.29720)
					(0.20000, 0.28340)
					(0.30000, 0.28526)
					(0.40000, 0.27909)
					(0.50000, 0.25514)
					(0.60000, 0.23106)
					(0.70000, 0.21667)
					(0.80000, 0.21144)
					(0.90000, 0.21281)
					(1.0000, 0.21740)
					(1.1000, 0.22071)
					(1.2000, 0.21813)
					(1.3000, 0.20672)
					(1.4000, 0.18680)
					(1.5000, 0.16182)
					(1.6000, 0.13610)
					(1.7000, 0.11265)
					(1.8000, 0.092756)
					(1.9000, 0.076520)
					(2.0000, 0.063500)
					(2.1000, 0.053113)
					(2.2000, 0.044810)
					(2.3000, 0.038134)
					(2.4000, 0.032726)
					(2.5000, 0.028307)
				};
				
			\end{scope}
			
			\begin{scope}[shift = {(0,-2.5)}]
				
				\node at (-2,1.2) {$\alpha = 1.7$};
				\node at (2,1.2) {$\theta = 0.8$};
				
				\draw[->] (0,-0.05) -- (0,1.2); \node[text=gray] at (0,-0.1) {$0$};
				\draw[<->] (-2.8,0) -- (2.8,0);
				\draw (-0.1,0.2) -- (0.1,0.2); \node[text=gray] at (0.4,0.2) {$0.2$};
				\draw (-0.1,0.4) -- (0.1,0.4); \node[text=gray] at (0.4,0.4) {$0.4$};
				\draw (-0.1,0.6) -- (0.1,0.6); \node[text=gray] at (0.4,0.6) {$0.6$};
				\draw (-0.1,0.8) -- (0.1,0.8); \node[text=gray] at (0.4,0.8) {$0.8$};
				\draw (-0.1,1.0) -- (0.1,1.0); \node[text=gray] at (0.4,1.0) {$1.0$};
				\draw (-1,-0.05) -- (-1,0); \node[text=gray] at (-1,-0.1) {$-1$};
				\draw (-2,-0.05) -- (-2,0); \node[text=gray] at (-2,-0.1) {$-2$};
				\draw (1,-0.05) -- (1,0); \node[text=gray] at (1,-0.1) {$1$};
				\draw (2,-0.05) -- (2,0); \node[text=gray] at (2,-0.1) {$2$};
				
				\draw[blue,smooth] plot coordinates {
					(-2.5000, 0.0056975)
					(-2.4000, 0.0069003)
					(-2.3000, 0.0085187)
					(-2.2000, 0.010769)
					(-2.1000, 0.014033)
					(-2.0000, 0.019025)
					(-1.9000, 0.027240)
					(-1.8000, 0.042211)
					(-1.7000, 0.074044)
					(-1.6000, 0.16108)
					(-1.5000, 0.51611)
					(-1.4000, 1.0650)
					(-1.3000, 0.44602)
					(-1.2000, 0.31846)
					(-1.1000, 0.24532)
					(-1.0000, 0.17054)
					(-0.90000, 0.16909)
					(-0.80000, 0.19584)
					(-0.70000, 0.20568)
					(-0.60000, 0.25035)
					(-0.50000, 0.37021)
					(-0.40000, 0.36794)
					(-0.30000, 0.29771)
					(-0.20000, 0.27721)
					(-0.10000, 0.30039)
					(0.00000, 0.28850)
					(0.10000, 0.27783)
					(0.20000, 0.26470)
					(0.30000, 0.26004)
					(0.40000, 0.25025)
					(0.50000, 0.23159)
					(0.60000, 0.21313)
					(0.70000, 0.20026)
					(0.80000, 0.19289)
					(0.90000, 0.18912)
					(1.0000, 0.18660)
					(1.1000, 0.18298)
					(1.2000, 0.17643)
					(1.3000, 0.16609)
					(1.4000, 0.15228)
					(1.5000, 0.13626)
					(1.6000, 0.11960)
					(1.7000, 0.10360)
					(1.8000, 0.089085)
					(1.9000, 0.076389)
					(2.0000, 0.065539)
					(2.1000, 0.056386)
					(2.2000, 0.048712)
					(2.3000, 0.042287)
					(2.4000, 0.036900)
					(2.5000, 0.032371)
				};
				
			\end{scope}
			
			\begin{scope}[shift = {(8,0)}]
				
				\node at (-2,0.5) {$\alpha = 1.2$};
				\node at (2,0.5) {$\theta = 0.0$};
				
				\draw[->] (0,-0.05) -- (0,0.5); \node[text=gray] at (0,-0.1) {$0$};
				\draw[<->] (-2.8,0) -- (2.8,0);
				\draw (-0.1,0.2) -- (0.1,0.2); \node[text=gray] at (0.4,0.2) {$0.2$};
				\draw (-0.1,0.4) -- (0.1,0.4); \node[text=gray] at (0.4,0.4) {$0.4$};
				\draw (-1,-0.05) -- (-1,0); \node[text=gray] at (-1,-0.1) {$-1$};
				\draw (-2,-0.05) -- (-2,0); \node[text=gray] at (-2,-0.1) {$-2$};
				\draw (1,-0.05) -- (1,0); \node[text=gray] at (1,-0.1) {$1$};
				\draw (2,-0.05) -- (2,0); \node[text=gray] at (2,-0.1) {$2$};
				
				\draw[blue,smooth] plot coordinates {
					(-2.5000, 0.044102)
					(-2.4000, 0.048062)
					(-2.3000, 0.052512)
					(-2.2000, 0.057523)
					(-2.1000, 0.063175)
					(-2.0000, 0.069558)
					(-1.9000, 0.076773)
					(-1.8000, 0.084926)
					(-1.7000, 0.094125)
					(-1.6000, 0.10448)
					(-1.5000, 0.11607)
					(-1.4000, 0.12894)
					(-1.3000, 0.14310)
					(-1.2000, 0.15844)
					(-1.1000, 0.17473)
					(-1.0000, 0.19161)
					(-0.90000, 0.20859)
					(-0.80000, 0.22510)
					(-0.70000, 0.24061)
					(-0.60000, 0.25473)
					(-0.50000, 0.26736)
					(-0.40000, 0.27857)
					(-0.30000, 0.28838)
					(-0.20000, 0.29626)
					(-0.10000, 0.30104)
					(0.00000, 0.30260)
					(0.10000, 0.30104)
					(0.20000, 0.29626)
					(0.30000, 0.28838)
					(0.40000, 0.27857)
					(0.50000, 0.26736)
					(0.60000, 0.25473)
					(0.70000, 0.24061)
					(0.80000, 0.22510)
					(0.90000, 0.20859)
					(1.0000, 0.19161)
					(1.1000, 0.17473)
					(1.2000, 0.15844)
					(1.3000, 0.14310)
					(1.4000, 0.12894)
					(1.5000, 0.11607)
					(1.6000, 0.10448)
					(1.7000, 0.094125)
					(1.8000, 0.084926)
					(1.9000, 0.076773)
					(2.0000, 0.069558)
					(2.1000, 0.063175)
					(2.2000, 0.057523)
					(2.3000, 0.052512)
					(2.4000, 0.048062)
					(2.5000, 0.044102)
				};
			\end{scope}
			
			\begin{scope}[shift = {(8,-0.9)}]
				
				\node at (-2,0.5) {$\alpha = 1.2$};
				\node at (2,0.5) {$\theta = 0.4$};
				
				\draw[->] (0,-0.05) -- (0,0.5); \node[text=gray] at (0,-0.1) {$0$};
				\draw[<->] (-3.8,0) -- (2.8,0);
				\draw (-0.1,0.2) -- (0.1,0.2); \node[text=gray] at (0.4,0.2) {$0.2$};
				\draw (-0.1,0.4) -- (0.1,0.4); \node[text=gray] at (0.4,0.4) {$0.4$};
				\draw (-1,-0.05) -- (-1,0); \node[text=gray] at (-1,-0.1) {$-1$};
				\draw (-2,-0.05) -- (-2,0); \node[text=gray] at (-2,-0.1) {$-2$};
				\draw (-3,-0.05) -- (-3,0); \node[text=gray] at (-3,-0.1) {$-3$};
				\draw (1,-0.05) -- (1,0); \node[text=gray] at (1,-0.1) {$1$};
				\draw (2,-0.05) -- (2,0); \node[text=gray] at (2,-0.1) {$2$};
				
				\draw[blue,smooth] plot coordinates {
					(-3.5000, 0.028450)
					(-3.4000, 0.031503)
					(-3.3000, 0.035056)
					(-3.2000, 0.039217)
					(-3.1000, 0.044127)
					(-3.0000, 0.049968)
					(-2.9000, 0.056973)
					(-2.8000, 0.065448)
					(-2.7000, 0.075794)
					(-2.6000, 0.088535)
					(-2.5000, 0.10435)
					(-2.4000, 0.12409)
					(-2.3000, 0.14879)
					(-2.2000, 0.17952)
					(-2.1000, 0.21704)
					(-2.0000, 0.26099)
					(-1.9000, 0.30849)
					(-1.8000, 0.35286)
					(-1.7000, 0.38451)
					(-1.6000, 0.39574)
					(-1.5000, 0.38611)
					(-1.4000, 0.36210)
					(-1.3000, 0.33185)
					(-1.2000, 0.30152)
					(-1.1000, 0.27505)
					(-1.0000, 0.25477)
					(-0.90000, 0.24117)
					(-0.80000, 0.23253)
					(-0.70000, 0.22608)
					(-0.60000, 0.21976)
					(-0.50000, 0.21218)
					(-0.40000, 0.20200)
					(-0.30000, 0.18824)
					(-0.20000, 0.17447)
					(-0.10000, 0.16300)
					(0.00000, 0.15208)
					(0.10000, 0.14171)
					(0.20000, 0.13196)
					(0.30000, 0.12283)
					(0.40000, 0.11433)
					(0.50000, 0.10642)
					(0.60000, 0.099062)
					(0.70000, 0.092224)
					(0.80000, 0.085872)
					(0.90000, 0.079978)
					(1.0000, 0.074515)
					(1.1000, 0.069458)
					(1.2000, 0.064781)
					(1.3000, 0.060460)
					(1.4000, 0.056470)
					(1.5000, 0.052788)
					(1.6000, 0.049392)
					(1.7000, 0.046258)
					(1.8000, 0.043367)
					(1.9000, 0.040699)
					(2.0000, 0.038235)
					(2.1000, 0.035959)
					(2.2000, 0.033855)
					(2.3000, 0.031908)
					(2.4000, 0.030106)
					(2.5000, 0.028436)
				};
				
			\end{scope}
			
			\begin{scope}[shift = {(8,-2.5)}]
				
				\node at (-2,1.2) {$\alpha = 1.2$};
				\node at (2,1.2) {$\theta = 0.8$};
				
				\draw[->] (0,-0.05) -- (0,1.2); \node[text=gray] at (0,-0.1) {$0$};
				\draw[<->] (-3.8,0) -- (2.8,0);
				\draw (-0.1,0.2) -- (0.1,0.2); \node[text=gray] at (0.4,0.2) {$0.2$};
				\draw (-0.1,0.4) -- (0.1,0.4); \node[text=gray] at (0.4,0.4) {$0.4$};
				\draw (-0.1,0.6) -- (0.1,0.6); \node[text=gray] at (0.4,0.6) {$0.6$};
				\draw (-0.1,0.8) -- (0.1,0.8); \node[text=gray] at (0.4,0.8) {$0.8$};
				\draw (-0.1,1.0) -- (0.1,1.0); \node[text=gray] at (0.4,1.0) {$1.0$};
				\draw (-1,-0.05) -- (-1,0); \node[text=gray] at (-1,-0.1) {$-1$};
				\draw (-2,-0.05) -- (-2,0); \node[text=gray] at (-2,-0.1) {$-2$};
				\draw (-3,-0.05) -- (-3,0); \node[text=gray] at (-3,-0.1) {$-3$};
				\draw (1,-0.05) -- (1,0); \node[text=gray] at (1,-0.1) {$1$};
				\draw (2,-0.05) -- (2,0); \node[text=gray] at (2,-0.1) {$2$};
				
				\draw[blue,smooth] plot coordinates {
					(-3.5000, 0.041591)
					(-3.4000, 0.052846)
					(-3.3000, 0.069329)
					(-3.2000, 0.094772)
					(-3.1000, 0.13674)
					(-3.0000, 0.21194)
					(-2.9000, 0.35969)
					(-2.8000, 0.65979)
					(-2.7000, 1.0608)
					(-2.6000, 0.94835)
					(-2.5000, 0.58630)
					(-2.4000, 0.38957)
					(-2.3000, 0.28982)
					(-2.2000, 0.24156)
					(-2.1000, 0.21518)
					(-2.0000, 0.18613)
					(-1.9000, 0.15490)
					(-1.8000, 0.13228)
					(-1.7000, 0.12210)
					(-1.6000, 0.11491)
					(-1.5000, 0.11662)
					(-1.4000, 0.12687)
					(-1.3000, 0.13337)
					(-1.2000, 0.13443)
					(-1.1000, 0.13220)
					(-1.0000, 0.13100)
					(-0.90000, 0.12814)
					(-0.80000, 0.12249)
					(-0.70000, 0.11854)
					(-0.60000, 0.11099)
					(-0.50000, 0.099684)
					(-0.40000, 0.092573)
					(-0.30000, 0.086824)
					(-0.20000, 0.082609)
					(-0.10000, 0.078232)
					(0.00000, 0.073812)
					(0.10000, 0.069774)
					(0.20000, 0.065991)
					(0.30000, 0.062446)
					(0.40000, 0.059126)
					(0.50000, 0.056015)
					(0.60000, 0.053099)
					(0.70000, 0.050362)
					(0.80000, 0.047792)
					(0.90000, 0.045378)
					(1.0000, 0.043108)
					(1.1000, 0.040972)
					(1.2000, 0.038963)
					(1.3000, 0.037073)
					(1.4000, 0.035293)
					(1.5000, 0.033616)
					(1.6000, 0.032037)
					(1.7000, 0.030549)
					(1.8000, 0.029146)
					(1.9000, 0.027824)
					(2.0000, 0.026577)
					(2.1000, 0.025400)
					(2.2000, 0.024289)
					(2.3000, 0.023240)
					(2.4000, 0.022249)
					(2.5000, 0.021312)
					(2.6000, 0.020426)
					(2.7000, 0.019588)
					(2.8000, 0.018795)
				};
			\end{scope}
			
		\end{tikzpicture}
		
		\caption{Approximations of $\mathbb{BP}(\cT,\nu_{\alpha,\theta})$ for $\cT = \{\emptyset, 1, 2, 3, 21, 31, 12, 13\}$ and for $(\alpha,\theta) \in \{1.7,1.2\} \times \{0.0,0.4,0.8\}$.} \label{fig:pictures1}
		
	\end{figure}
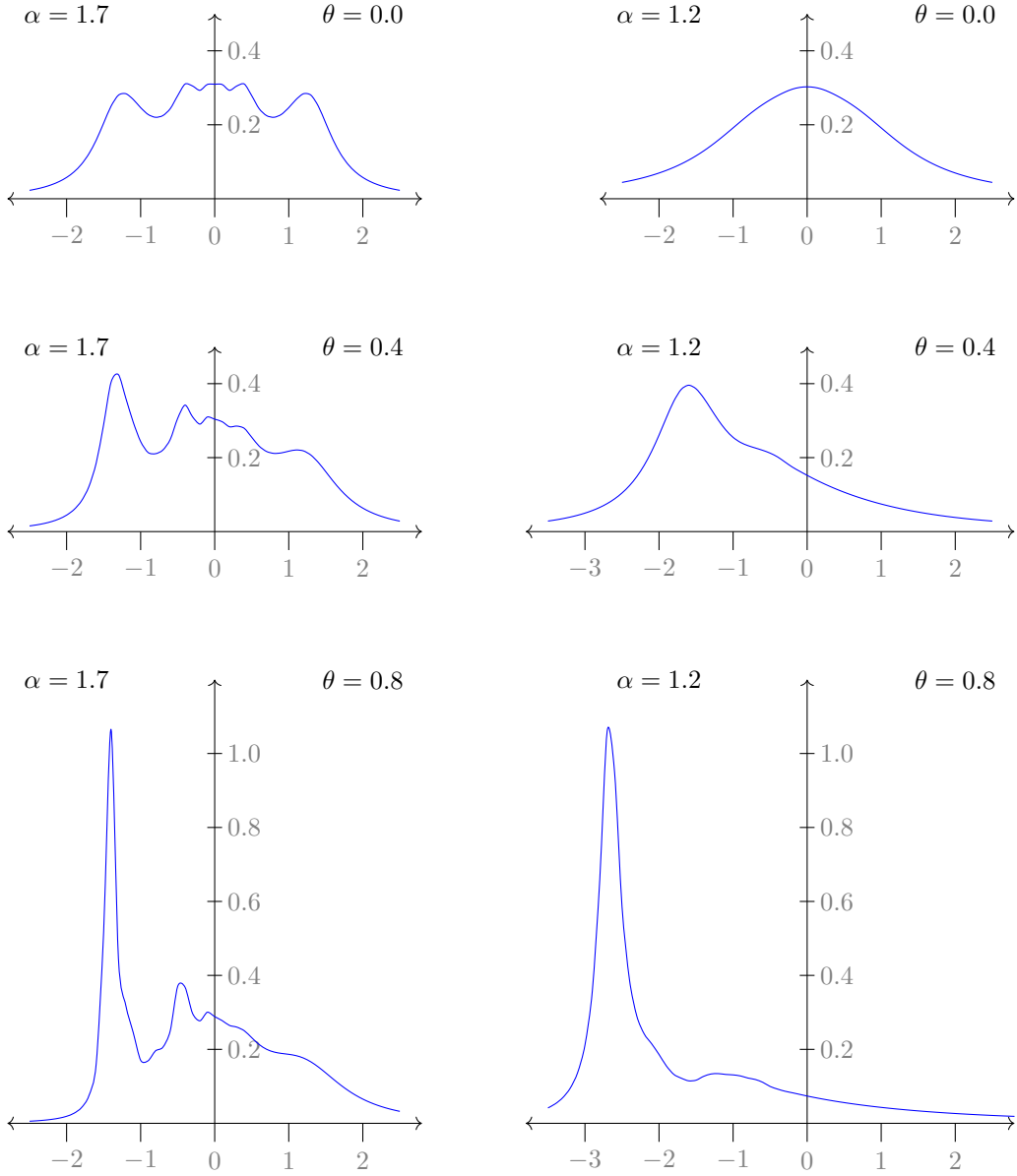
	
	\begin{figure}
		\begin{tikzpicture}[yscale=5]
			
			\begin{scope}
				
				\node at (-2,0.5) {$\alpha = 0.7$};
				\node at (2,0.5) {$\theta = 0.0$};
				
				\draw[->] (0,-0.05) -- (0,0.5); \node[text=gray] at (0,-0.1) {$0$};
				\draw[<->] (-2.8,0) -- (2.8,0);
				\draw (-0.1,0.2) -- (0.1,0.2); \node[text=gray] at (0.4,0.2) {$0.2$};
				\draw (-0.1,0.4) -- (0.1,0.4); \node[text=gray] at (0.4,0.4) {$0.4$};
				\draw (-1,-0.05) -- (-1,0); \node[text=gray] at (-1,-0.1) {$-1$};
				\draw (-2,-0.05) -- (-2,0); \node[text=gray] at (-2,-0.1) {$-2$};
				\draw (1,-0.05) -- (1,0); \node[text=gray] at (1,-0.1) {$1$};
				\draw (2,-0.05) -- (2,0); \node[text=gray] at (2,-0.1) {$2$};
				
				\draw[blue,smooth] plot coordinates {
					(-2.5000, 0.037846)
					(-2.4000, 0.039975)
					(-2.3000, 0.042301)
					(-2.2000, 0.044851)
					(-2.1000, 0.047654)
					(-2.0000, 0.050745)
					(-1.9000, 0.054168)
					(-1.8000, 0.057972)
					(-1.7000, 0.062219)
					(-1.6000, 0.066979)
					(-1.5000, 0.072343)
					(-1.4000, 0.078419)
					(-1.3000, 0.085340)
					(-1.2000, 0.093273)
					(-1.1000, 0.10243)
					(-1.0000, 0.11307)
					(-0.90000, 0.12554)
					(-0.80000, 0.14029)
					(-0.70000, 0.15787)
					(-0.60000, 0.17904)
					(-0.50000, 0.20479)
					(-0.40000, 0.23638)
					(-0.30000, 0.27534)
					(-0.20000, 0.32326)
					(-0.10000, 0.38057)
					(0.00000, 0.43978)
					(0.10000, 0.38057)
					(0.20000, 0.32326)
					(0.30000, 0.27534)
					(0.40000, 0.23638)
					(0.50000, 0.20479)
					(0.60000, 0.17904)
					(0.70000, 0.15787)
					(0.80000, 0.14029)
					(0.90000, 0.12554)
					(1.0000, 0.11307)
					(1.1000, 0.10243)
					(1.2000, 0.093273)
					(1.3000, 0.085340)
					(1.4000, 0.078419)
					(1.5000, 0.072343)
					(1.6000, 0.066979)
					(1.7000, 0.062219)
					(1.8000, 0.057972)
					(1.9000, 0.054168)
					(2.0000, 0.050745)
					(2.1000, 0.047654)
					(2.2000, 0.044851)
					(2.3000, 0.042301)
					(2.4000, 0.039975)
					(2.5000, 0.037846)
				};
			\end{scope}
			
			\begin{scope}[shift = {(0,-0.9)}]
				
				\node at (-2,0.5) {$\alpha = 0.7$};
				\node at (2,0.5) {$\theta = 0.4$};
				
				\draw[->] (0,-0.05) -- (0,0.5); \node[text=gray] at (0,-0.1) {$0$};
				\draw[<->] (-2.8,0) -- (2.8,0);
				\draw (-0.1,0.2) -- (0.1,0.2); \node[text=gray] at (0.4,0.2) {$0.2$};
				\draw (-0.1,0.4) -- (0.1,0.4); \node[text=gray] at (0.4,0.4) {$0.4$};
				\draw (-1,-0.05) -- (-1,0); \node[text=gray] at (-1,-0.1) {$-1$};
				\draw (-2,-0.05) -- (-2,0); \node[text=gray] at (-2,-0.1) {$-2$};
				\draw (1,-0.05) -- (1,0); \node[text=gray] at (1,-0.1) {$1$};
				\draw (2,-0.05) -- (2,0); \node[text=gray] at (2,-0.1) {$2$};
				
				\draw[blue,smooth] plot coordinates {
					(-2.5000, 0.015140)
					(-2.4000, 0.015876)
					(-2.3000, 0.016673)
					(-2.2000, 0.017538)
					(-2.1000, 0.018481)
					(-2.0000, 0.019511)
					(-1.9000, 0.020639)
					(-1.8000, 0.021880)
					(-1.7000, 0.023251)
					(-1.6000, 0.024771)
					(-1.5000, 0.026465)
					(-1.4000, 0.028363)
					(-1.3000, 0.030501)
					(-1.2000, 0.032925)
					(-1.1000, 0.035693)
					(-1.0000, 0.038880)
					(-0.90000, 0.042585)
					(-0.80000, 0.046938)
					(-0.70000, 0.052119)
					(-0.60000, 0.058381)
					(-0.50000, 0.066094)
					(-0.40000, 0.075827)
					(-0.30000, 0.088513)
					(-0.20000, 0.10585)
					(-0.10000, 0.13153)
					(0.00000, 0.18137)
					(0.10000, 0.25507)
					(0.20000, 0.30020)
					(0.30000, 0.32437)
					(0.40000, 0.33160)
					(0.50000, 0.32652)
					(0.60000, 0.31341)
					(0.70000, 0.29572)
					(0.80000, 0.27595)
					(0.90000, 0.25572)
					(1.0000, 0.23603)
					(1.1000, 0.21745)
					(1.2000, 0.20024)
					(1.3000, 0.18448)
					(1.4000, 0.17015)
					(1.5000, 0.15718)
					(1.6000, 0.14546)
					(1.7000, 0.13488)
					(1.8000, 0.12532)
					(1.9000, 0.11669)
					(2.0000, 0.10887)
					(2.1000, 0.10178)
					(2.2000, 0.095345)
					(2.3000, 0.089484)
					(2.4000, 0.084139)
					(2.5000, 0.079253)
				};
			\end{scope}
			
			\begin{scope}[shift = {(0,-1.8)}]
				
				\node at (-2,0.5) {$\alpha = 0.7$};
				\node at (2,0.5) {$\theta = 0.8$};
				
				\draw[->] (0,-0.05) -- (0,0.5); \node[text=gray] at (0,-0.1) {$0$};
				\draw[<->] (-2.8,0) -- (2.8,0);
				\draw (-0.1,0.2) -- (0.1,0.2); \node[text=gray] at (0.4,0.2) {$0.2$};
				\draw (-0.1,0.4) -- (0.1,0.4); \node[text=gray] at (0.4,0.4) {$0.4$};
				\draw (-1,-0.05) -- (-1,0); \node[text=gray] at (-1,-0.1) {$-1$};
				\draw (-2,-0.05) -- (-2,0); \node[text=gray] at (-2,-0.1) {$-2$};
				\draw (1,-0.05) -- (1,0); \node[text=gray] at (1,-0.1) {$1$};
				\draw (2,-0.05) -- (2,0); \node[text=gray] at (2,-0.1) {$2$};
				
				\draw[blue,smooth] plot coordinates {
					(-2.5000, 0.0034438)
					(-2.4000, 0.0035865)
					(-2.3000, 0.0037396)
					(-2.2000, 0.0039041)
					(-2.1000, 0.0040813)
					(-2.0000, 0.0042728)
					(-1.9000, 0.0044801)
					(-1.8000, 0.0047053)
					(-1.7000, 0.0049508)
					(-1.6000, 0.0052192)
					(-1.5000, 0.0055139)
					(-1.4000, 0.0058388)
					(-1.3000, 0.0061987)
					(-1.2000, 0.0065994)
					(-1.1000, 0.0070482)
					(-1.0000, 0.0075542)
					(-0.90000, 0.0081291)
					(-0.80000, 0.0087881)
					(-0.70000, 0.0095514)
					(-0.60000, 0.010447)
					(-0.50000, 0.011513)
					(-0.40000, 0.012810)
					(-0.30000, 0.014429)
					(-0.20000, 0.016536)
					(-0.10000, 0.019476)
					(0.00000, 0.024734)
					(0.10000, 0.048661)
					(0.20000, 0.075027)
					(0.30000, 0.10348)
					(0.40000, 0.13255)
					(0.50000, 0.16062)
					(0.60000, 0.18631)
					(0.70000, 0.20849)
					(0.80000, 0.22638)
					(0.90000, 0.23955)
					(1.0000, 0.24799)
					(1.1000, 0.25198)
					(1.2000, 0.25206)
					(1.3000, 0.24890)
					(1.4000, 0.24319)
					(1.5000, 0.23561)
					(1.6000, 0.22675)
					(1.7000, 0.21709)
					(1.8000, 0.20703)
					(1.9000, 0.19686)
					(2.0000, 0.18681)
					(2.1000, 0.17702)
					(2.2000, 0.16760)
					(2.3000, 0.15862)
					(2.4000, 0.15010)
					(2.5000, 0.14206)
				};
			\end{scope}
			
			\begin{scope}[shift = {(7,0)}]
				
				\node at (-2,0.5) {$\alpha = 0.2$};
				\node at (2,0.5) {$\theta = 0.0$};
				
				\draw[->] (0,-0.05) -- (0,0.5); \node[text=gray] at (0,-0.1) {$0$};
				\draw[<->] (-2.8,0) -- (2.8,0);
				\draw (-0.1,0.2) -- (0.1,0.2); \node[text=gray] at (0.4,0.2) {$0.2$};
				\draw (-0.1,0.4) -- (0.1,0.4); \node[text=gray] at (0.4,0.4) {$0.4$};
				\draw (-1,-0.05) -- (-1,0); \node[text=gray] at (-1,-0.1) {$-1$};
				\draw (-2,-0.05) -- (-2,0); \node[text=gray] at (-2,-0.1) {$-2$};
				\draw (1,-0.05) -- (1,0); \node[text=gray] at (1,-0.1) {$1$};
				\draw (2,-0.05) -- (2,0); \node[text=gray] at (2,-0.1) {$2$};
				
				\clip (-2.5,0) rectangle (2.5,0.45);
				\draw[blue,smooth] plot coordinates {
					(-2.5000, 0.013958)
					(-2.4000, 0.014552)
					(-2.3000, 0.015196)
					(-2.2000, 0.015898)
					(-2.1000, 0.016667)
					(-2.0000, 0.017512)
					(-1.9000, 0.018444)
					(-1.8000, 0.019478)
					(-1.7000, 0.020633)
					(-1.6000, 0.021929)
					(-1.5000, 0.023395)
					(-1.4000, 0.025066)
					(-1.3000, 0.026988)
					(-1.2000, 0.029224)
					(-1.1000, 0.031855)
					(-1.0000, 0.034998)
					(-0.90000, 0.038818)
					(-0.80000, 0.043560)
					(-0.70000, 0.049605)
					(-0.60000, 0.057580)
					(-0.50000, 0.068591)
					(-0.40000, 0.084803)
					(-0.30000, 0.11111)
					(-0.20000, 0.16154)
					(-0.10000, 0.30064)
					(-0.050000, 0.54600)
				};
				
				\draw[blue,smooth] plot coordinates {
					(0.050000, 0.54600)
					(0.10000, 0.30064)
					(0.20000, 0.16154)
					(0.30000, 0.11111)
					(0.40000, 0.084803)
					(0.50000, 0.068591)
					(0.60000, 0.057580)
					(0.70000, 0.049605)
					(0.80000, 0.043560)
					(0.90000, 0.038818)
					(1.0000, 0.034998)
					(1.1000, 0.031855)
					(1.2000, 0.029224)
					(1.3000, 0.026988)
					(1.4000, 0.025066)
					(1.5000, 0.023395)
					(1.6000, 0.021929)
					(1.7000, 0.020633)
					(1.8000, 0.019478)
					(1.9000, 0.018444)
					(2.0000, 0.017512)
					(2.1000, 0.016667)
					(2.2000, 0.015898)
					(2.3000, 0.015196)
					(2.4000, 0.014552)
					(2.5000, 0.013958)
				};
				
			\end{scope}
			
			\begin{scope}[shift = {(7,-0.9)}]
				
				\node at (-2,0.5) {$\alpha = 0.2$};
				\node at (2,0.5) {$\theta = 0.4$};
				
				\draw[->] (0,-0.05) -- (0,0.5); \node[text=gray] at (0,-0.1) {$0$};
				\draw[<->] (-2.8,0) -- (2.8,0);
				\draw (-0.1,0.2) -- (0.1,0.2); \node[text=gray] at (0.4,0.2) {$0.2$};
				\draw (-0.1,0.4) -- (0.1,0.4); \node[text=gray] at (0.4,0.4) {$0.4$};
				\draw (-1,-0.05) -- (-1,0); \node[text=gray] at (-1,-0.1) {$-1$};
				\draw (-2,-0.05) -- (-2,0); \node[text=gray] at (-2,-0.1) {$-2$};
				\draw (1,-0.05) -- (1,0); \node[text=gray] at (1,-0.1) {$1$};
				\draw (2,-0.05) -- (2,0); \node[text=gray] at (2,-0.1) {$2$};
				
				\clip (-2.5,0) rectangle (2.5,0.45);
				\draw[blue,smooth] plot coordinates {
					(-2.5000, 0.0081460)
					(-2.4000, 0.0084911)
					(-2.3000, 0.0088659)
					(-2.2000, 0.0092744)
					(-2.1000, 0.0097213)
					(-2.0000, 0.010212)
					(-1.9000, 0.010754)
					(-1.8000, 0.011356)
					(-1.7000, 0.012027)
					(-1.6000, 0.012780)
					(-1.5000, 0.013631)
					(-1.4000, 0.014602)
					(-1.3000, 0.015719)
					(-1.2000, 0.017017)
					(-1.1000, 0.018545)
					(-1.0000, 0.020369)
					(-0.90000, 0.022586)
					(-0.80000, 0.025338)
					(-0.70000, 0.028845)
					(-0.60000, 0.033471)
					(-0.50000, 0.039856)
					(-0.40000, 0.049256)
					(-0.30000, 0.064508)
					(-0.20000, 0.093745)
					(-0.10000, 0.17444)
					(-0.050000, 0.31704)
					(-0.030000, 0.48481)
				};
				\draw[blue,smooth] plot coordinates {
					(0.050000, 0.77398)
					(0.10000, 0.42827)
					(0.20000, 0.23090)
					(0.30000, 0.15904)
					(0.40000, 0.12147)
					(0.50000, 0.098287)
					(0.60000, 0.082529)
					(0.70000, 0.071109)
					(0.80000, 0.062449)
					(0.90000, 0.055654)
					(1.0000, 0.050179)
					(1.1000, 0.045673)
					(1.2000, 0.041900)
					(1.3000, 0.038694)
					(1.4000, 0.035936)
					(1.5000, 0.033540)
					(1.6000, 0.031437)
					(1.7000, 0.029577)
					(1.8000, 0.027922)
					(1.9000, 0.026437)
					(2.0000, 0.025100)
					(2.1000, 0.023888)
					(2.2000, 0.022786)
					(2.3000, 0.021778)
					(2.4000, 0.020853)
					(2.5000, 0.020002)
				};
				
			\end{scope}
			
			\begin{scope}[shift = {(7,-1.8)}]
				
				\node at (-2,0.5) {$\alpha = 0.2$};
				\node at (2,0.5) {$\theta = 0.8$};
				
				\draw[->] (0,-0.05) -- (0,0.5); \node[text=gray] at (0,-0.1) {$0$};
				\draw[<->] (-2.8,0) -- (2.8,0);
				\draw (-0.1,0.2) -- (0.1,0.2); \node[text=gray] at (0.4,0.2) {$0.2$};
				\draw (-0.1,0.4) -- (0.1,0.4); \node[text=gray] at (0.4,0.4) {$0.4$};
				\draw (-1,-0.05) -- (-1,0); \node[text=gray] at (-1,-0.1) {$-1$};
				\draw (-2,-0.05) -- (-2,0); \node[text=gray] at (-2,-0.1) {$-2$};
				\draw (1,-0.05) -- (1,0); \node[text=gray] at (1,-0.1) {$1$};
				\draw (2,-0.05) -- (2,0); \node[text=gray] at (2,-0.1) {$2$};
				
				\clip (-2.5,0) rectangle (2.5,0.45);
				\draw[blue,smooth] plot coordinates {
					(-2.5000, 0.0026296)
					(-2.4000, 0.0027404)
					(-2.3000, 0.0028608)
					(-2.2000, 0.0029919)
					(-2.1000, 0.0031353)
					(-2.0000, 0.0032929)
					(-1.9000, 0.0034667)
					(-1.8000, 0.0036596)
					(-1.7000, 0.0038746)
					(-1.6000, 0.0041160)
					(-1.5000, 0.0043889)
					(-1.4000, 0.0046997)
					(-1.3000, 0.0050572)
					(-1.2000, 0.0054726)
					(-1.1000, 0.0059613)
					(-1.0000, 0.0065445)
					(-0.90000, 0.0072528)
					(-0.80000, 0.0081314)
					(-0.70000, 0.0092505)
					(-0.60000, 0.010725)
					(-0.50000, 0.012760)
					(-0.40000, 0.015751)
					(-0.30000, 0.020598)
					(-0.20000, 0.029875)
					(-0.10000, 0.055416)
					(-0.050000, 0.10044)
					(-0.020000,0.21260)
					(-0.007000, 0.47718)
				};
				\draw[blue,smooth] plot coordinates {
					(0.10000, 0.55055)
					(0.20000, 0.29904)
					(0.30000, 0.20666)
					(0.40000, 0.15815)
					(0.50000, 0.12814)
					(0.60000, 0.10769)
					(0.70000, 0.092858)
					(0.80000, 0.081594)
					(0.90000, 0.072747)
					(1.0000, 0.065614)
					(1.1000, 0.059739)
					(1.2000, 0.054817)
					(1.3000, 0.050633)
					(1.4000, 0.047033)
					(1.5000, 0.043903)
					(1.6000, 0.041156)
					(1.7000, 0.038726)
					(1.8000, 0.036561)
					(1.9000, 0.034621)
					(2.0000, 0.032872)
					(2.1000, 0.031287)
					(2.2000, 0.029845)
					(2.3000, 0.028526)
					(2.4000, 0.027316)
					(2.5000, 0.026202)};
			\end{scope}
			
		\end{tikzpicture}
		
		\caption{Approximations of $\mathbb{BP}(\cT,\nu_{\alpha,\theta})$ for $\cT = \{\emptyset, 1, 2, 3, 21, 31, 12, 13\}$ and for $(\alpha,\theta) \in \{0.7,0.2\} \times \{0.0,0.4,0.8\}$.} \label{fig:pictures2}
		
	\end{figure}
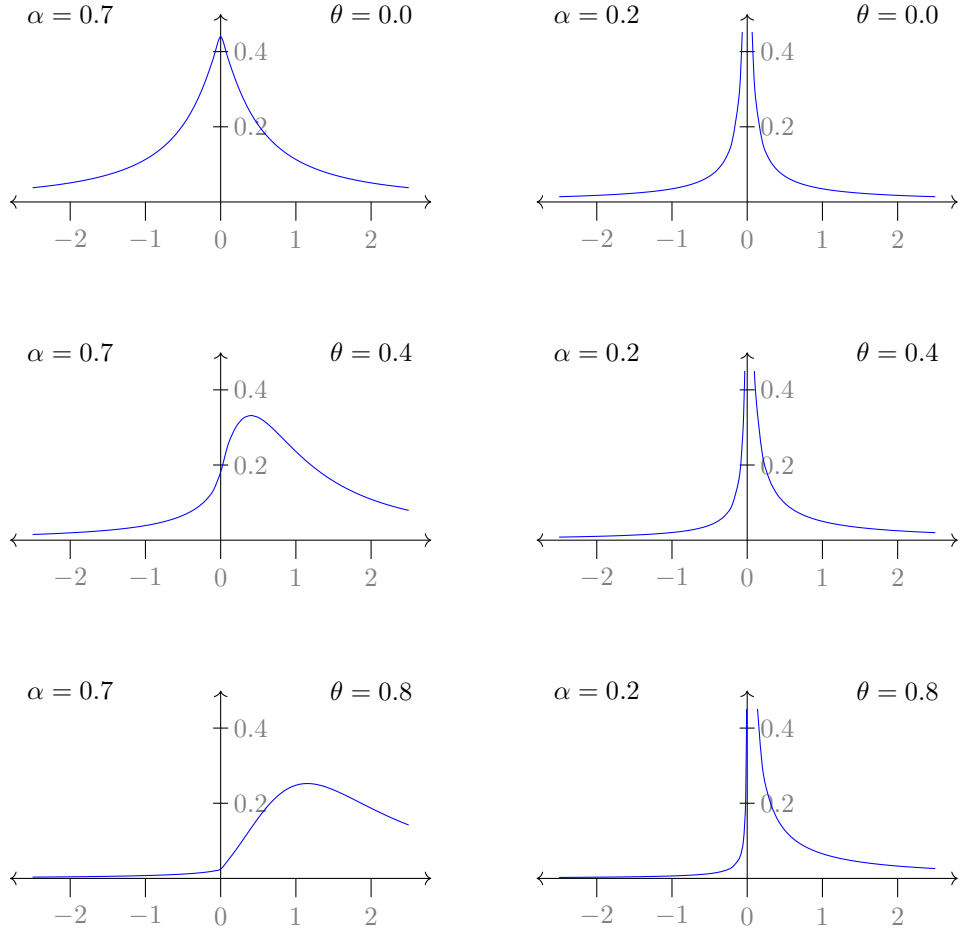
	
	Another open question concerns the operator models for $\cT$-free convolution.  In this paper, we focused exclusively on the complex-analytic viewpoint for $\cT$-free convolutions, even though the original definition of the convolution for compactly supported measures was in terms of addition of ``independent'' bounded self-adjoint operators \cite{JekelLiu2020}.  Moreover, the free convolution of arbitrary measures on $\R$ can be expressed using the addition of freely independent unbounded self-adjoint operators, thanks to the theory of unbounded operators affiliated to a tracial von Neumann algebra \cite{BV1992}.
	
	\begin{question}
		Can the $\cT$-free convolution of arbitrary probability measures on $\R$ be formulated in terms of addition $\cT$-free independent unbounded self-adjoint operators?
	\end{question}
	
	Because arbitrary self-adjoint operators cannot necessarily be added, the challenge is to use the additional structure of $\cT$-free independence (or perhaps of the $\cT$-free product Hilbert space) to show that the sum actually makes sense.  Again, we believe that the solution for finite-variance measures is significantly easier than for the general case.
	
	\bibliographystyle{plain}
	\bibliography{convolutions}

\begin{thebibliography}{10}

\bibitem{AS2004}
David Aldous and J.~Michael Steele.
\newblock The objective method: Probabilistic combinatorial optimization and
  local weak convergence.
\newblock In Harry Kesten, editor, {\em Probability on Discrete Structures},
  volume 110 of {\em Encyclopaedia of Mathematical Sciences (Probability
  Theory)}. Springer, Berlin, Heidelberg, 2004.

\bibitem{ABFN2013}
Michael Anshelevich, Serban~T. Belinschi, Maxime F{'e}vrier, and Alexandru
  Nica.
\newblock Convolution powers in the operator-valued framework.
\newblock {\em Trans. Am. Math. Soc.}, 365:2063--2097, 2013.

\bibitem{AW2014}
Michael Anshelevich and John~D. Williams.
\newblock Limit theorems for monotonic convolution and the {C}hernoff product
  formula.
\newblock {\em International Mathematics Research Notices},
  2014(11):2990--3021, 2014.

\bibitem{AW2016}
Michael Anshelevich and John~D. Williams.
\newblock Operator-valued monotone convolution semigroups and an extension of
  the {B}ercovici-{P}ata bijection.
\newblock {\em Documenta Mathematica}, 21:841--871, 2016.

\bibitem{ABT2019}
Octavio Arizmendi, Miguel Ballesteros, and Francisco Torres-Ayala.
\newblock Conditionally free reduced products of hilbert spaces.
\newblock To appear in Studia Mathematica, 2019.

\bibitem{Belinschi2003}
Serban~T. Belinschi.
\newblock The atoms of the free multiplicative convolution of two probability
  distributions.
\newblock {\em Integral Equations Operator Theory}, 46(4):377--386, 2003.

\bibitem{Belinschi2006}
Serban~T. Belinschi.
\newblock Complex analysis methods in non-commutative probability.
\newblock Ph.D. thesis at University of Indiana, 2006.

\bibitem{Belinschi2006b}
Serban~T. Belinschi.
\newblock A note on regularity for free convolutions.
\newblock {\em Ann. Inst. H. Poincar{\'e} Prob.}, 42:635--648, 2006.

\bibitem{BB2004}
Serban~T. Belinschi and Hari Bercovici.
\newblock Atoms and regularity for measures in a partially defined free
  convolution semigroup.
\newblock {\em Math. Z.}, 248(4):665--674, 2004.

\bibitem{BMS2013}
Serban~T. Belinschi, Tobias Mai, and Roland Speicher.
\newblock Analytic subordination theory of operator-valued free additive
  convolution and the solution of a general random matrix problem.
\newblock {\em Journal f{\"u}r die reine und angewandte Mathematik (Crelles
  Journal)}, 03 2013.

\bibitem{BPV2013}
Serban~T. Belinschi, Mihai Popa, and Victor Vinnikov.
\newblock On the operator-valued analogues of the semicircle, arcsine and
  {B}ernoulli laws.
\newblock {\em Journal of Operator Theory}, 70(1):239--258, 2013.

\bibitem{BGS2002}
A.~{Ben Ghorbal} and M.~Sch{\"u}rmann.
\newblock Non-commutative notions of stochastic independence.
\newblock {\em Math. Proc. Camb. Phil. Soc.}, 133:531--561, 2002.

\bibitem{BP1999}
Hari Bercovici and Vittorino Pata.
\newblock Stable laws and domains of attraction in free probability theory.
\newblock {\em Ann. Math.}, 149:1023--1060, 1999.
\newblock With an appendix by Philippe Biane.

\bibitem{BV1992}
Hari Bercovici and Dan-Virgil Voiculescu.
\newblock L{\`e}vy-{H}incin type theorems for multiplicative and additive free
  convolution.
\newblock {\em Pac. J. Math.}, 153:217--248, 1992.

\bibitem{BV1998}
Hari Bercovici and Dan-Virgil Voiculescu.
\newblock Regularity questions for free convolution.
\newblock In Hari Bercovici and Ciprian~I. Foias, editors, {\em Nonselfadjoint
  Operator Algebras, Operator Theory, and Related Topics}, volume 104 of {\em
  Oper. Theory Adv. Appl.}, pages 37--47. Birkh{\"a}user, Basel, 1998.

\bibitem{Biane1997b}
Philippe Biane.
\newblock On the free convolution with a semi-circular distribution.
\newblock {\em Indiana Univ. Math. J.}, 46(3):705--718, 1997.

\bibitem{Biane1998}
Philippe Biane.
\newblock Processes with free increments.
\newblock {\em Mathematische Zeitschrift}, 227(1):143--174, 1 1998.

\bibitem{Billingsley1999}
Patrick Billingsley.
\newblock {\em Convergence of Probability Measures}.
\newblock Wiley Series in Probability and Statistics. John Wiley \& Sons, Inc.,
  New York, 2 edition, 1999.

\bibitem{BGT1987}
N.~H. Bingham, C.~M. Goldie, and J.~L. Teugels.
\newblock {\em Regular Variation}, volume~27 of {\em Encyclopedia of Mathematic
  and its Applications}.
\newblock Cambridge University Press, Cambridge, 1971.

\bibitem{BLS1996}
Marek Bo{\. z}ejko, Michael Leinert, and Roland Speicher.
\newblock Convolution and limit theorems for conditionally free random
  variables.
\newblock {\em Pacific J. Math.}, 125(2):357--388, 1996.

\bibitem{EH1970}
Clifford~J. Earle and Richard~S. Hamilton.
\newblock A fixed point theorem for holomorphic functions.
\newblock In S.~Smale and S.~S. Chern, editors, {\em Global Analysis},
  volume~16 of {\em Proceedings of Symposia in Pure Mathematics}, pages 61--65.
  American Mathematical Society, Providence, 1970.

\bibitem{Folland1999}
Gerald~B. Folland.
\newblock {\em Real Analysis: Modern Techiques and their Applications}.
\newblock Pure and Applied Mathematics. John Wiley \& Sons, Inc., 2 edition,
  1999.

\bibitem{FL1999}
Uwe Franz and Romuald Lenczewski.
\newblock Limit theorems for the hierarchy of freeness.
\newblock {\em Probab. Math. Stat.}, 19:23--41, 1999.

\bibitem{GK1954}
B.~V. Gnedenko and A.~N. Kolmogorov.
\newblock {\em Limit Distributions for Sums of Independent Random Variables}.
\newblock Addison-Wesley Publ. Co., Cambridge, Mass., 1954.

\bibitem{Hasebe2010a}
Takahiro Hasebe.
\newblock Monotone convolution and monotone infinite divisibility from complex
  analytic viewpoints.
\newblock {\em Infin. Dimens. Anal. Quantum Probab. Relat. Top.},
  13(1):111--131, 2010.

\bibitem{Hasebe2010b}
Takahiro Hasebe.
\newblock Monotone convolution semigroups.
\newblock {\em Studia Math.}, 200:175--199, 2010.

\bibitem{Hasebe2011}
Takahiro Hasebe.
\newblock Conditionally monotone independence i: independence, additive
  convolutions and related convolutions.
\newblock {\em Infinite Dimensional Analysis, Quantum Probability and Related
  Topics}, 14(03):465--516, 2011.

\bibitem{HS2014}
Takahiro Hasebe and Hayato Saigo.
\newblock On operator-valued monotone independence.
\newblock {\em Nagoya Math. J.}, 215:151--167, 2014.

\bibitem{HS2015}
Takahiro Hasebe and Noriyoshi Sakuma.
\newblock Unimodality of boolean and monotone stable distributions.
\newblock {\em Demonstratio Mathematica}, 48(3):424--439, 2015.

\bibitem{Jekel2020}
David Jekel.
\newblock Operator-valued chordal loewner chains and non-commutative
  probability.
\newblock {\em J. Func. Anal.}, 278(10):108452, 2020.

\bibitem{JekelLiu2020}
David Jekel and Weihua Liu.
\newblock An operad of non-commuative independences defined by trees.
\newblock {\em Dissertationes Mathematicae}, 553:1--100, 2020.

\bibitem{Krystek2007}
Anna~Dorota Krystek.
\newblock Infinite divisibility for the conditionally free convolution.
\newblock {\em Infin. Dim. Anal. Quantum Prob. and Relat. Top.},
  10(04):499--522, 2007.

\bibitem{KW2013}
Anna Kula and Janusz Wysocza{\'n}ski.
\newblock An example of a {B}oolean-free type central limit theorem.
\newblock {\em Probab. Math. Statist.}, 33:341--352, 2013.

\bibitem{Lax2002}
Peter~D. Lax.
\newblock Functional analysis.
\newblock 2002.

\bibitem{Leinster2004}
Tom Leinster.
\newblock {\em Higher Operads, Higher Categories}, volume 298 of {\em London
  Mathemical Society Lectures Notes Series}.
\newblock Cambridge University Press, 2004.

\bibitem{Lenczewski1998}
Romuald Lenczewski.
\newblock Unification of independence in quantum probability.
\newblock {\em Infin. Dimens. Anal. Quantum Probab. Relat. Top.}, 1:383--405,
  1998.

\bibitem{Lenczewski2007}
Romuald Lenczewski.
\newblock Decompositions of the free additive convolution.
\newblock {\em Journal of Functional Analysis}, 246(2):330--365, 2007.

\bibitem{Lenczewski2008}
Romuald Lenczewski.
\newblock Operators related to subordination for free multiplicative
  convolutions.
\newblock {\em Indiana Univ. Math. J.}, 57:1055--1103, 2008.

\bibitem{Lenczewski2019}
Romuald Lenczewski.
\newblock Conditionally monotone independence and the associated products of
  graphs.
\newblock {\em Infin. Dimens. Anal. Quantum. Probab. Relat. Top.},
  22(04):1950023, 2019.

\bibitem{Levy1937}
Paul L{\'e}vy.
\newblock {\em Th{\'e}orie de L'addition des Variables Al{\'e}atoires}.
\newblock Gauthier-Villars, Paris, 1937.

\bibitem{Liu2018}
Weihua Liu.
\newblock Relations between convolutions and transforms in operator-valued free
  probability.
\newblock {\em arXiv e-prints}, 09 2018.

\bibitem{Mlotkowski2004}
Wojciech M{\l}otkowski.
\newblock $\lambda$-free probability.
\newblock {\em Infinite-dimensional Analysis, Quantum Probability, and Related
  Topics}, 7:27--41, 2004.

\bibitem{Muraki2000}
Naofumi Muraki.
\newblock Monotonic convolution and monotone {L}{\'e}vy-{H}in{\v c}in formula.
\newblock preprint, 2000.

\bibitem{Muraki2001}
Naofumi Muraki.
\newblock Monotonic independence, monotonic central limit theorem, and
  monotonic law of small numbers.
\newblock {\em Infinite Dimensional Analysis, Quantum Probability, and Related
  Topics}, 04, 2001.

\bibitem{Muraki2003}
Naofumi Muraki.
\newblock The five independences as natural products.
\newblock {\em Infinite Dimensional Analysis, Quantum Probability and Related
  Topics}, 6(3):337--371, 2003.

\bibitem{Muraki2013}
Naofumi Muraki.
\newblock A simple proof of the classification theorem for positive natural
  products.
\newblock {\em Probab. Math. Statist.}, 33(2):315--326, 2013.

\bibitem{Nevanlinna1922}
R.~Nevanlinna.
\newblock Asymptotische entwickelungen beschr{\"a}nkter funktionen und das
  stieltjessche moment-problem.
\newblock {\em Ann. Acad. Sci. Fennicae, A}, 18, 1922.

\bibitem{Nica2009}
Alexandru Nica.
\newblock Multi-variable subordination distributions for free additive
  convolution.
\newblock {\em Journal of Functional Analysis}, 257(2):428 -- 463, 2009.

\bibitem{PV2013}
Mihai Popa and Victor Vinnikov.
\newblock Non-commutative functions and the non-commutative {L}{\'e}vy-{H}in{\v
  c}in formula.
\newblock {\em Adv. Math.}, 236:131--157, 2013.

\bibitem{Speicher1997}
Roland Speicher.
\newblock On universal products.
\newblock In Dan Voiculescu, editor, {\em Free Probability Theory}, volume~12
  of {\em Fields Inst. Commun.}, pages 257--266. Amer. Math. Soc., 1997.

\bibitem{Speicher1998}
Roland Speicher.
\newblock Combinatorial theory of the free product with amalgamation and
  operator-valued free probability theory.
\newblock {\em Mem. Amer. Math. Soc.}, 132(627), 1998.

\bibitem{SW1997}
Roland Speicher and Reza Woroudi.
\newblock Boolean convolution.
\newblock In Dan Voiculescu, editor, {\em Free Probability Theory}, volume~12
  of {\em Fields Inst. Commun.}, pages 267--279. Amer. Math. Soc., 1997.

\bibitem{SW2016}
Roland Speicher and Janusz Wysocza{\'{n}}ski.
\newblock Mixtures of classical and free independence.
\newblock {\em Archiv der Mathematik}, 107(4):445--453, 10 2016.

\bibitem{Voiculescu1985}
Dan-Virgil Voiculescu.
\newblock Symmetries of some reduced free product ${C}^*$-algebras.
\newblock In Huzihiro Araki, Calvin~C. Moore, {\c{S}}erban-Valentin Stratila,
  and Dan-Virgil Voiculescu, editors, {\em Operator Algebras and their
  Connections with Topology and Ergodic Theory}, pages 556--588. Springer
  Berlin Heidelberg, Berlin, Heidelberg, 1985.

\bibitem{Voiculescu1986}
Dan-Virgil Voiculescu.
\newblock Addition of certain non-commuting random variables.
\newblock {\em Journal of Functional Analysis}, 66(3):323--346, 1986.

\bibitem{VoiculescuFE1}
Dan-Virgil Voiculescu.
\newblock The analogues of entropy and {F}isher's information in free
  probability, {I}.
\newblock {\em Comm. Math. Phys.}, 155(1):71--92, 1993.

\bibitem{Voiculescu2000}
Dan-Virgil Voiculescu.
\newblock The coalgebra of the difference quotient and free probability.
\newblock {\em Internat. Math. Res. Notices}, (2):79--106, 2000.

\bibitem{Voiculescu2002b}
Dan-Virgil Voiculescu.
\newblock Analytic subordination consequences of free {M}arkovianity.
\newblock {\em Indiana Univ. Math. J.}, 51:1161--1166, 2002.

\bibitem{Wysoczanski2010}
Janusz Wysocza{\'n}ski.
\newblock bm-independence and bm-central limit theorems associated with
  symmetric cones.
\newblock {\em Infinite Dimensional Analysis, Quantum Probability and Related
  Topics}, 13(03):461--488, 2010.

\bibitem{Zorn1945a}
Max~A. Zorn.
\newblock Characterization of analytic functions in {B}anach spaces.
\newblock {\em Ann. of Math.}, 2, 1945.

\bibitem{Zorn1945b}
Max~A. Zorn.
\newblock {G}{\^a}teaux differentiability and essential boundedness.
\newblock {\em Duke Math. J.}, 12:579--583, 1945.

\bibitem{Zorn1946}
Max~A. Zorn.
\newblock Derivatives and {F}r{\'e}chet differentials.
\newblock {\em Bull. Amer. Math. Soc.}, 52:133--137, 1946.

\end{thebibliography}
	
\end{document}